\documentclass[a4paper,11pt]{article}
\usepackage{fullpage}
\usepackage{amsfonts}
\usepackage{amssymb,amsthm,amsmath,color}
\usepackage{array}
\usepackage{mathrsfs}
\usepackage{dsfont}
\usepackage{lmodern}
\usepackage{hyperref}
\usepackage{graphicx}
\usepackage[utf8]{inputenc}  
\usepackage[T1]{fontenc} 
\usepackage{stmaryrd}  
\numberwithin{equation}{section}

\DeclareMathAlphabet{\mathbbo}{U}{bbold}{m}{n}

\newtheorem{theorem}{Th\'eor\`eme}[section]
\newtheorem{lemma}[theorem]{Lemme}
\newtheorem{Remarque}[theorem]{Remarque}

\newtheorem{cor}[theorem]{Corollaire}

\date{}
\title{Ensembles de petite somme}

\author{RIBLET Robin}
\begin{document}
\maketitle
\begin{abstract}
Ruzsa a démontré une minoration précise de la mesure de la somme de deux ensembles bornés de réels $A$ et $B$ faisant intervenir le ratio $\lambda(A)/\lambda(B)$. De Roton a établi un résultat de structure à propos des ensembles critiques de cette minoration. Ici, nous prouvons une généralisation du travail de de Roton en établissant un résultat dans un voisinage du cas d'égalité. 
\end{abstract}

\section{Introduction et résultat}
Si $A$ et $B$ sont deux ensembles bornés non vides de réels, il est bien connu (cf. \cite{lambda(A+B)>A+B}) que
$$\lambda(A+B)\geqslant\lambda(A)+\lambda(B).$$
Si cette inégalité est optimale pour certains ensembles $A$ et $B$ (si ce sont des intervalles par exemple), il est en revanche possible d'être plus précis en toute généralité en faisant intervenir le ratio $\lambda(A)/\lambda(B)$ et le diamètre de $B$. 

Dans la suite, $A$ et $B$ seront deux ensembles fermés bornés de réels de mesures non nulles. On notera $D_B=\sup B-\inf B$ le diamètre de $B$ et on désignera par $(K,\delta)$ l'unique couple défini par 
\begin{equation}\label{def_(K,delta)}
\begin{cases}
K\in\mathbb{N^*} \ \rm{ et } \ 0\leqslant\delta<1 \\
\frac{\lambda(A)}{\lambda(B)}=\frac{K(K-1)}{2}+K\delta 
\end{cases}
\end{equation}
Ruzsa \cite{Ruzsa_minoration_AplusB} a démontré le théorème suivant.
\begin{theorem}[Ruzsa]\label{thm_Ruzsa}
Soient $A,B\subseteq\mathbb{R}$ deux ensembles bornés non vides tels que $\lambda(B)\neq 0$. Notons $D_B$ le diamètre de $B$ et $(K,\delta)$ défini par \eqref{def_(K,delta)}. On a 
$$\lambda(A+B)\geqslant \lambda(A)+\min\big( D_B,(K+\delta)\lambda(B)\big).$$
\end{theorem}
La minoration du théorème \ref{thm_Ruzsa} est également optimale et de Roton a étudié l'un des cas d'égalité (cf. \cite{Anne_structure}).
\begin{theorem}[de Roton]\label{thm_Anne}
Soient $A,B\subseteq\mathbb{R}$ deux fermés bornés tels que $\lambda(A),\lambda(B)\neq 0$. Notons $D_B$ le diamètre de $B$ et $(K,\delta)$ défini par \eqref{def_(K,delta)}. Si
$$\lambda(A+B)= \lambda(A)+\left( K+\delta\right)\lambda(B)<\lambda(A)+D_B,$$
alors $B$ et $A$ sont des translatés d'ensembles $B_0$ et $A_0$ de la forme
$$B_0=\left[0,b_+\right]\cup\left[D_B-b_-,D_B\right],$$
$$A_0=\bigcup\limits_{\substack{k=1}}^K \left[(k-1)(D_B-b_-),(k-1)D_B+(K-k)b_+ +\delta \lambda(B)\right],$$
où $b_+,b_-\geqslant 0$ et $b_+ +b_-=\lambda(B)$.
\end{theorem}
\vspace{0.4cm}

$$\begin{minipage}[l]{17cm}
\includegraphics[height=13.5cm]{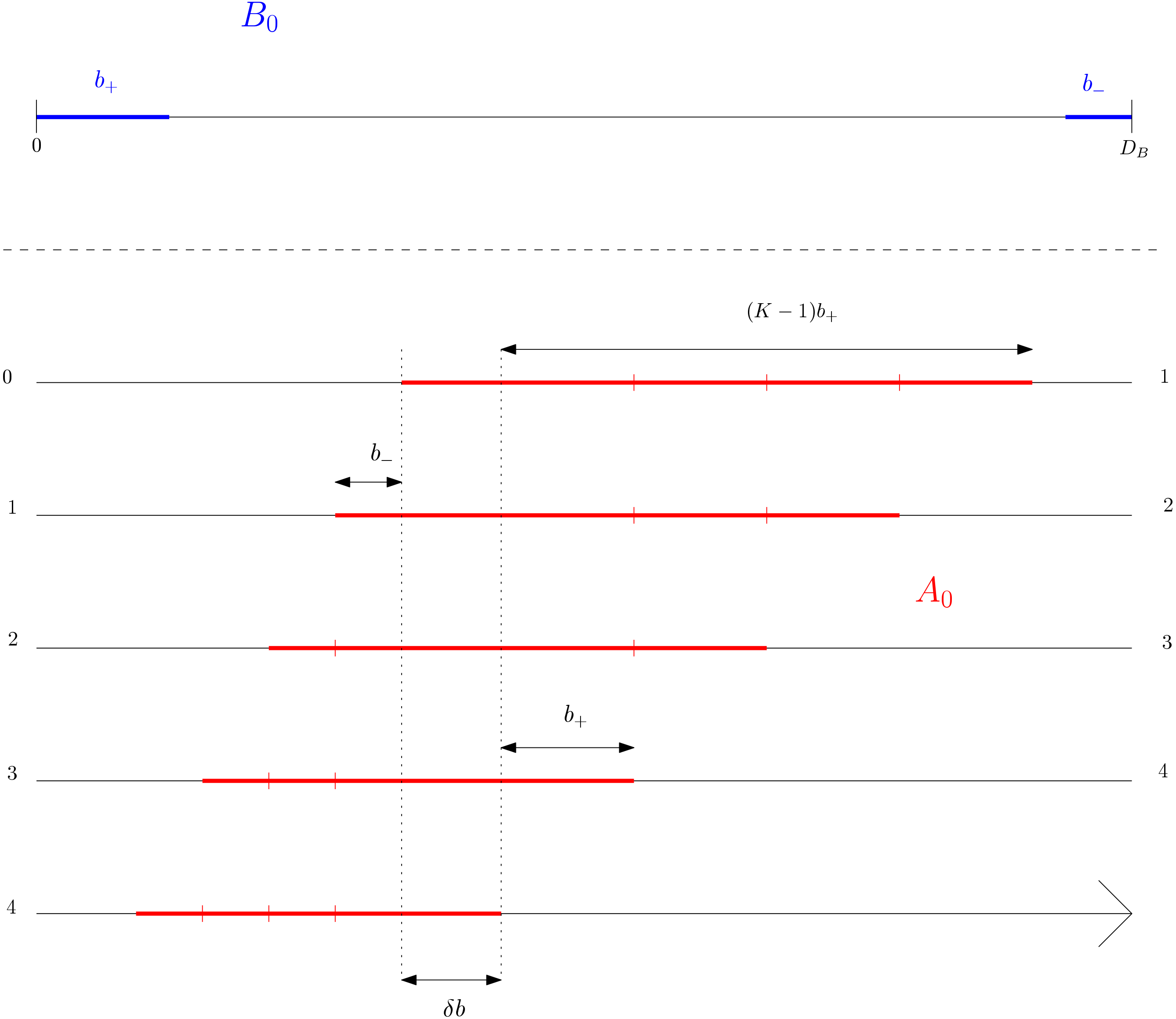}
\end{minipage}$$
\vspace{0.5cm}

Dans cet article, nous nous intéressons au voisinage de ce cas d'égalité, lorsque
$$\lambda(A+B)=\lambda(A)+(K+\delta+\varepsilon )\lambda(B),$$ 
pour un petit $\varepsilon>0$. Nous verrons alors que, si $\varepsilon$ est suffisamment petit, $A$ et $B$ sont respectivement inclus dans des translatés de voisinages de $A_0$ et $B_0$ (où $A_0$ et $B_0$ sont les ensembles définis dans le théorème \ref{thm_Anne}). Ce résultat fait l'objet du théorème \ref{main_result} ci-après.
\begin{theorem}\label{main_result}
Soient $A,B\subseteq\mathbb{R}$ deux ensembles fermés, bornés, de mesures non nulles et tels que $\lambda(B)\leqslant\lambda(A\mod D_B)$, où $D_B$ désigne le diamètre de $B$. Soit $(K,\delta)$ définis par \eqref{def_(K,delta)}.
Si $\lambda(A+B)= \lambda(A)+\left( K+\delta+\varepsilon\right)\lambda(B)$ et 
$$\big(K+\delta+ (K^2\log K+12)\varepsilon\big)\lambda(B)<D_B$$ 
pour $\varepsilon >0$ tel que
$$\varepsilon < \min\left(\left(\dfrac{\delta}{3K}\right)^3,\dfrac{1-\delta}{K^3\log K},\frac{\rho_0}{K \lambda(B)}\right),$$
où $\rho_0=3,1\times 10^{-1549}$, alors $B$ est inclus dans un translaté de $\left[0,b_+\right]\cup\left[D_B-b_-,D_B\right],$ où $b_+,b_-\geqslant 0$ et $b_+ +b_-\leqslant\lambda(B)(1+\varepsilon )$, et $A$ est inclus dans un translaté de $A_0+\left[ 0,\varepsilon \lambda(B) \right]$, où 
$$A_0=\bigcup\limits_{\substack{k=1}}^K \left[(k-1)(D_B-b_-),(k-1)D_B+(K-k)b_+ +\delta \lambda(B)\right] .$$
\end{theorem}
\begin{Remarque}
$A_0$ et $A$ sont de même mesure donc $A_0+\left[ 0,\varepsilon \lambda(B) \right]$ est un voisinage métrique assez fin de $A$.
\end{Remarque}
La borne $\varepsilon\leqslant \frac{\rho_0}{\lambda(B)}$ provient d'un théorème de Candela et de Roton \cite{Anne_Pablo} (théorème \ref{Anne_Pablo} ci-après) dont nous allons nous servir dans ce travail, et qui étend le théorème de Kneser \cite{Kneser} aux ensembles presque critiques dans le tore $\mathbb{T}=\mathbb{R}/\mathbb{Z}$. 

On note $\mu$ la mesure de Haar intérieure sur le tore $\mathbb{T}$ ($\mathbb{T}=\mathbb{R}/\mathbb{Z}$) et pour tout ensemble $E$ et tout scalaire $n$, on pose $nE=\left\lbrace ne \ \vert \ e\in E \right\rbrace.$
\begin{theorem}[Candela et de Roton]\label{Anne_Pablo}
Soit $\rho\in [0,\rho_0]$ où $\rho_0=3,1\times10^{-1549}$. Soient $A,B\subset\mathbb{T}$ satisfaisant 
$$\mu(A+B)=\mu(A)+ \mu(B) +\rho < \frac{1}{2}\left(\mu(A)+ \mu(B) +1\right),$$ 
et $\rho < \mu(B) \leqslant\mu(A)$. Alors il existe trois intervalles $I,J,K\subseteq\mathbb{T}$, avec $I,J$ fermés et $K$ ouvert, et un entier naturel $n$ non nul tels que $nA\subseteq I$, $nB\subseteq J$, $K\subseteq n(A+B)$, et $\mu(I)\leqslant\mu(A)+\rho$, $\mu(J)\leqslant \mu(B) +\rho$, $\mu(K)\geqslant\mu(A)+ \mu(B) $.
\end{theorem}

Toute amélioration de la borne $\rho_0$ qui provient d'un résultat dans $\mathbb{Z}_p$ (le corps fini à $p$ éléments) de Grynkiewicz, améliorera automatiquement le domaine de validité de notre résultat. On conjecture que $\rho_0$ peut atteindre $1$. 
\begin{Remarque}
Dans le théorème \ref{main_result}, nous avons supposé $A$ et $B$ fermés mais quitte à considérer des ensembles fermés $A_n\subseteq A$ et $B_n\subseteq B$ tels que $\lambda(A_n)\rightarrow\lambda(A)$, $\lambda(B_n)\rightarrow\lambda(B)$ et $D_{B_n}\rightarrow D_B$, on peut supposer $A$ et $B$ seulement Lebesgue-mesurables. 
\end{Remarque}

\section{Preuve du théorème \ref{main_result}} 

La démonstration qui reprend les idées et la structure de la preuve développée par de Roton dans le cas d'égalité \cite{Anne_structure}, comportera principalement trois étapes. La première consistera à travailler modulo le dimaètre de $B$ puis à utiliser l'inégalité de Ruzsa et le Théorème \ref{Anne_Pablo} afin de montrer que nos ensembles sont proches de progressions arithmétiques d'intervalles. La deuxième consistera à prouver que $B$ est finalement proche de l'union de deux intervalles seulement. La troisième permettra de dégager la structure de $A$. Dans le cas d'égalité, les sous-ensembles sont des intervalles, alors que nous avons des "erreurs", les sous-ensembles ne remplissent pas entièrement les intervalles et il pourra y avoir des décalages, ce qui complique chaque étape et ajoute des contraintes. Notamment pour la dernière étape, il nous faudra établir la structure de $A$ en trois temps. Tout d'abord, nous dégagerons la structure principale de $A$, qui contient la grande majorité des éléments de $A$, puis nous exhiberons des informations sur le reste des éléments de $A$. Nous finirons par affiner ces informations en retrouvant un point de vue plus global. 

\subsection{Hypothèses}

Notons tout d'abord que quitte à translater $A$ et $B$ et normaliser $B$, on peut supposer que $A$ et $B$ sont deux fermés de mesure non nulle tels que $\min B=0$, $D_B=\sup B-\inf B=1$, et $\min A\in\left[0,1\right[$. Soient $K\in\mathbb{N^*}$ et $0\leqslant\delta<1$ tels que 
\begin{equation}\label{A/B=K(K-1)2+Kdelta}
\frac{\lambda(A)}{\lambda(B)}=\frac{K(K-1)}{2}+K\delta.
\end{equation}
\begin{Remarque}\label{K>1}
Notons que $K$ est nécessairement supérieur ou égal à $2$ car $\lambda(A)\geqslant\lambda(B)$ par hypothèse. 
\end{Remarque}
On suppose que $\lambda(A+B)= \lambda(A)+\left( K+\delta+\varepsilon\right)\lambda(B),$ où $\varepsilon$ est un réel positif tel que
\begin{equation}\label{hyp_1}
\varepsilon < \left(\dfrac{\delta}{3K}\right)^3,
\end{equation}
\begin{equation}\label{hyp_2}
\varepsilon <\dfrac{1-\delta}{K^3\log K},
\end{equation}
\begin{equation}\label{hyp_3}
\varepsilon <\frac{\rho_0}{K \lambda(B)},
\end{equation}
et $\big(K+\delta+ (K^2\log K+12)\varepsilon\big)\lambda(B)<D_B=1$. En fait nous n'aurons besoin que de l'hypothèse plus faible suivante
\begin{equation}\label{hyp_4}
\Big(K+\delta+ \big(K^2\log (K)-K\big(K(1+\log 4)-\log K-7+\log 4\big)+8\big)\varepsilon\Big)\lambda(B)<1,
\end{equation}
ce qui entraine en particulier
\begin{equation}\label{hyp_0}
(K+\delta+2\varepsilon )\lambda(B)<1.
\end{equation}

\subsection{Mesures modulo le diamètre de $B$}

Posons $S=A+B$. Pour tout entier positif $k$ et tout sous-ensemble $E$ de $\mathbb{R_+}$ on définit
$$\tilde{E}_k=\left\lbrace x\in\left[ 0,1\right[ \ \vert \ \#\left\lbrace n\in\mathbb{N} \ \vert \ n+x\in E \right\rbrace \geqslant k \right\rbrace \ \rm{ et } \ K_E=\sup\left\lbrace k\in\mathbb{N} \ \vert \ \tilde{E}_k\neq\varnothing \right\rbrace.$$
Notons que $\tilde{E}_{k+1}\subseteq \tilde{E_k}$ et $\tilde{E_1}=\pi(E)$, où $\pi$ désigne la projection de $\mathbb{R}$ dans $\mathbb{T}$. Tout au long de la preuve du théorème \ref{main_result}, nous allons utiliser le fait que pour tout ensemble borné $E\subset\mathbb{R}_+$, on a
\begin{equation}\label{lambda(E)=somme_mu(Ek)}
\lambda(E)=\sum\limits_{\substack{k=1}}^{K_E}\mu\big( \tilde{E}_k\big).
\end{equation}
Ruzsa (cf. \cite{Ruzsa_minoration_AplusB} ou le lemme 1 dans \cite{Anne_structure}) a utilisé cette égalité afin de prouver l'inégalité suivante. Si $\lambda(A+B)<\lambda(A)+D_B$, on a
\begin{equation}\label{lemme_Ruzsa}
\lambda(A+B)\geqslant\dfrac{K_A+1}{K_A}\lambda(A)+\dfrac{K_A+1}{2}\lambda(B).
\end{equation}
\begin{Remarque}
Comme $D_B=1$, on utilisera indifféremment $\lambda(B)$ et $ \mu(B) $ pour désigner la mesure de $B$. On pose 
$$b=\lambda(B) =  \mu(B) .$$
\end{Remarque}
Nous sommes désormais prêts à entamer la première étape. Tout d'abord, montrons que
\begin{equation}\label{K=K_A}
K_A=K,
\end{equation}
où $K$ est défini par \eqref{A/B=K(K-1)2+Kdelta}.
\begin{proof}
D'une part comme $\tilde{A}_k+B\subseteq\tilde{S}_k$ pour tout $k\in\left\lbrace 1,...,K_A \right\rbrace$, on a $\mu\big( \tilde{S}_k\big) \geqslant \mu\big( \tilde{A}_k\big)+ b$ et donc
$$\lambda(A+B)=\lambda(S)=\sum\limits_{\substack{k=1}}^{K_S}\mu\big( \tilde{S}_k\big) \geqslant \sum\limits_{\substack{k=1}}^{K_A}\mu\big( \tilde{S}_k\big) \geqslant \sum\limits_{\substack{k=1}}^{K_A}\big(\mu\big( \tilde{A}_k\big)+ b \big) \geqslant \lambda(A)+K_A b,$$
d'où $K_A\leqslant K+\delta+\varepsilon$. Or par l'hypothèse \eqref{hyp_1} on a $\varepsilon < 1-\delta$ et donc $ K_A\leqslant K.$
D'autre part, par hypothèse et par \eqref{A/B=K(K-1)2+Kdelta}, on a
$$\lambda(A+B)=\lambda(A)+(K+\delta+\varepsilon )b=\dfrac{K+1}{K}\lambda(A)+\dfrac{K+1}{2}b+\varepsilon b.$$
Ainsi par la minoration \eqref{lemme_Ruzsa}, on a
$$\dfrac{K+1}{K}\lambda(A)+\dfrac{K+1}{2}b+ \varepsilon b\geqslant\dfrac{K_A+1}{K_A}\lambda(A)+\dfrac{K_A+1}{2}b,$$
ce qui entraîne
$$\varepsilon\geqslant \dfrac{K-K_A}{2K_A}\left( K-K_A-1+2\delta\right),$$
or l'hypothèse \eqref{hyp_1} implique $\varepsilon\leqslant \dfrac{\delta}{K}\leqslant \dfrac{\delta}{K_A} ,$ et donc $K_A>K-1$. 
\end{proof}
Cette première égalité va nous permettre, en reprenant la preuve de l'inégalité \eqref{lemme_Ruzsa} de Ruzsa, d'obtenir des contrôles sur les inégalités entre les mesures d'ensembles modulo $D_B$. Ainsi nous saurons que ces inégalités sont proches du cas d'égalité, ce qui nous permettra d'utiliser le théorème \ref{Anne_Pablo} et ainsi obtenir des informations de structure modulo $D_B$ sur nos ensembles.
\subsubsection{Structures de $A$ et $B$ modulo $D_B$, premières informations}

Par l'hypothèse \eqref{A/B=K(K-1)2+Kdelta} et d'après \eqref{K=K_A}, on a
$$\dfrac{K_A+1}{K_A}\lambda(A)+\dfrac{K_A+1}{2}b= \lambda(A)+(K+\delta )b.$$
Définissons $\varepsilon^1_k$, $\varepsilon^2_k$ et $\varepsilon^3_k$ par
\begin{equation}\label{presque_égalités}
\left\lbrace \begin{array}{lll} \mu (\tilde{S_k})= \mu (\tilde{A}_{k-1})+\varepsilon^1_kb \ \ \ (2\leqslant k\leqslant K+1) \\ \mu (\tilde{S_k})= \mu (\tilde{A_k})+\mu (B)+\varepsilon^2_kb \ \ \ (1\leqslant k \leqslant K) \\ \mu (\tilde{S_k})= \varepsilon^3_kb \ \ \ (k\geqslant K +2) \end{array} \right. ,
\end{equation}
Comme $\tilde{A}_{k-1}\subseteq\tilde{S_k}$ (car $0,1\in B$) et comme $\tilde{A_k}+B\subseteq\tilde{S_k}$ on a $\varepsilon^i_k\geqslant 0$ pour tout $i$ et $k$. De plus en utilisant $K_A=K$
\begin{align*}
\lambda(A+B) & = \sum\limits_{\substack{k=1}}^{K_S}\mu(\tilde{S}_k) \\ & = \sum\limits_{\substack{k=1}}^{K+1}\left[\frac{k-1}{K_A}\left(\mu(\tilde{A}_{k-1})+\varepsilon^1_kb\right)+\frac{K_A-k+1}{K_A}\left(\mu(\tilde{A}_{k})+ b +\varepsilon^2_kb\right)\right]+\sum\limits_{\substack{k=K+2}}^{K_S}\varepsilon^3_kb \\ & = \frac{K_A+1}{K_A}\lambda(A)+\frac{K_A+1}{2}b+\left[ \sum_{k\geqslant K+2}\varepsilon_k^3+\sum_{k=0}^K \frac{k}{K}\left(\varepsilon_{k+1}^1+\varepsilon_{K+1-k}^2\right)\right]b \\ & =\lambda(A)+(K+\delta )b +\left[ \sum_{k\geqslant K+2}\varepsilon_k^3+ \sum_{k=0}^K \frac{k}{K}\left(\varepsilon_{k+1}^1+\varepsilon_{K+1-k}^2\right)\right]b.
\end{align*}
Ainsi nécessairement
\begin{equation}\label{eps=somme}
\varepsilon=\sum_{k\geqslant K+2}\varepsilon_k^3+\sum_{k=0}^K \frac{k}{K}\left(\varepsilon_{k+1}^1+\varepsilon_{K+1-k}^2\right).
\end{equation}
donc en particulier
\begin{equation}\label{maj_Somme_eps_3}
\sum_{k\geqslant K+2}\varepsilon_k^3\leqslant\varepsilon
\end{equation}
et donc pour tout $k\geqslant K+2$
\begin{equation}\label{maj_eps_3}
\mu(\tilde{S}_k)=\varepsilon^3_kb\leqslant \varepsilon b .
\end{equation}
Pour tout $k\in\left\{ 2,...,K+1\right\}$, on a également
\begin{equation}\label{maj_eps_1}
\varepsilon^1_k\leqslant \dfrac{K}{k-1}\varepsilon
\end{equation}
et pour tout $k\in\left\{ 1,...,K\right\}$ on a
\begin{equation}\label{maj_eps_2}
\varepsilon^2_k\leqslant \dfrac{K}{K+1-k}\varepsilon.
\end{equation}
Enfin plus brutalement, on a aussi
\begin{equation}\label{1ere_condition}
\varepsilon^i_k\leqslant K\varepsilon\leqslant K\frac{\rho_0}{K b}\leqslant\frac{ \rho_0}{ b},
\end{equation}
où la deuxième inégalité provient de l'hypothèse \eqref{hyp_3}.
Comme pour tout $k\in\left\lbrace 1,...,K \right\rbrace$, $\tilde{A_k}+B\subseteq \tilde{S}_k$, on a d'après la deuxième ligne de \eqref{presque_égalités}
$$\mu ( \tilde{A_k}+B)\leqslant\mu ( \tilde{S_k})\leqslant \mu ( \tilde{A_k})+b+\varepsilon^2_kb.$$ 
Grâce aux hypothèses \eqref{hyp_1}, \eqref{hyp_2}, \eqref{hyp_3} et \eqref{hyp_0}, nous pouvons appliquer le théorème \ref{Anne_Pablo} et ainsi, pour tout $k\in\left\lbrace 1,...,K \right\rbrace$, il existe un entier $m_k>0$ et deux intervalles $I_k$ et $J_k$ dans $\mathbb{T}$ tels que $m_k\tilde{A_k}\subseteq I_k$, $m_k B\subseteq J_k$ avec 
$$
\left\lbrace \begin{array}{ll} \mu (I_k)\leqslant \mu (\tilde{A}_{k})+\varepsilon^2_kb \\ \mu (J_k)\leqslant b+\varepsilon^2_kb \end{array} \right. . 
$$
On choisira pour $I_k$ l'enveloppe convexe de $m_k\tilde{A}_k$ et pour $J_k$ celle de $m_k B$ dans $\mathbb{T}$. Ceci termine la première étape de la preuve. Dans la suite, nous allons montrer que $m_k=1$ pour tout $k$, ce qui implique que $B\mod 1$ est contenu dans un intervalle de $\mathbb{T}$ qu'il remplit presque.

\subsection{$B$ modulo $D_B$ est proche d'un intervalle}

Commençons par prouver que $m_k$ ne dépend pas de $k$. C'est à dire, montrons que $m_{l}=m_k$ pour tous $k$ et $l$ appartenant à $\left\lbrace 1,...,K \right\rbrace$.
\begin{lemma}\label{tildes-presque-intervalles}
Il existe un entier $m>0$ tel que pour tout $k\in\left\{1,...,K\right\}$ on ait $m\tilde{A_k}\subseteq I_k$, $m B\subseteq J_k$ où $I_k$ et $J_k$ sont des intervalles tels que 
$$
\left\lbrace \begin{array}{ll} \mu (I_k)\leqslant \mu (\tilde{A}_{k})+\varepsilon^2_kb \\ \mu (J_k)\leqslant \mu (B)+\varepsilon^2_kb \end{array} \right. . 
$$
\end{lemma}
\begin{proof}
Il s'agit de prouver que $m_k=m_l$ pour tout $k,l\in\left\lbrace 1,...,K \right\rbrace$. Tout d'abord par \eqref{maj_eps_2} puis par \eqref{hyp_0}, on a 
\begin{equation}\label{mu(Jk)<1/2}
\mu (J_k)\leqslant b+\varepsilon^2_kb  \leqslant (1+K\varepsilon) b < \dfrac{1+K\varepsilon}{K+\delta+2\varepsilon}\leqslant\dfrac{1}{2}.
\end{equation}
On va supposer qu'il existe des entiers $k$ et $l$ tels que $m_k<m_l$. On sait que $m_k B\subseteq J_k$ et $m_l B\subseteq J_l$ avec 
$$
\left\lbrace \begin{array}{ll} \mu (J_k)\leqslant \mu (B)+\varepsilon^2_kb \\ \mu (J_l)\leqslant \mu (B)+\varepsilon^2_lb \end{array} \right. .
$$
Comme $m_k^{-1}J_k$ et $m_l^{-1}J_l$ contiennent tous deux $B$, on a
$$\mu\left(m_k^{-1}J_k\cap m_l^{-1}J_l\right)\geqslant  b .$$
Ainsi d'une part, on a
\begin{equation}\label{maj_J_kDeltaJ_l}
\mu\left( \left( m_k^{-1}J_k\right)\Delta \left( m_l^{-1}J_l\right)\right) \leqslant \varepsilon^2_kb+\varepsilon^2_lb.
\end{equation}
D'autre part dans $\mathbb{T}$, $m_k^{-1}J_k$ est composé de $m_k$ intervalles de taille $\mu(J_k)/m_k$ uniformément espacés (deux centres consécutifs sont distants de $1/m_k$). De même $m_l^{-1}J_l$ est composé de $m_l$ intervalles plus petits (de taille $\mu(J_l)/m_l$) dont les centres sont distants de $1/m_l$ ($<1/m_k$). Pour obtenir une absurdité, on montre qu'une proportion importante des centres des intervalles composant $m_l^{-1}J_l$ ne sont pas dans $m_k^{-1}J_k$, ce qui provient du fait que la taille de $\mathbb{T}\setminus m_k^{-1}J_k$ est supérieure à celle de $m_k^{-1}J_k$ (cf. \eqref{mu(Jk)<1/2}). On utilise ensuite que si un centre d'un intervalle $I$ de $m_l^{-1}J_l$ est dans $\mathbb{T}\setminus m_k^{-1}J_k$, alors $\mu\left( I\setminus m_k^{-1}J_k\right)\geqslant\mu(I)/2$, ce qui provient du fait que la distance entre deux intervalles consécutifs de $m_k^{-1}J_k$ est supérieure à la taille d'un intervalle $I$ de $m_l^{-1}J_l$. On obtient ainsi l'inégalité (non optimale mais suffisante) $\mu\left( \left( m_k^{-1}J_k\right)\Delta \left( m_l^{-1}J_l\right)\right) \geqslant b/8$, ce qui contredit \eqref{maj_J_kDeltaJ_l} (par l'hypothèse \eqref{hyp_1}).

La première étape est moins immédiate que les suivantes. Deux cas sont en effet à distinguer. Si deux centres consécutifs d'intervalles de $m_l^{-1}J_l$ sont dans $m_k^{-1}J_k$ (sinon, la moitié au moins des centres est dans $m_k^{-1}J_k$), alors soit ces deux centres sont dans le même intervalle de $m_k^{-1}J_k$, auquel cas un ensemble de centres consécutifs dans $\mathbb{T}\setminus m_k^{-1}J_k$ succède à un ensemble de centres consécutifs dans $m_k^{-1}J_k$
$$\begin{minipage}[l]{17cm}
\includegraphics[height=3.1cm]{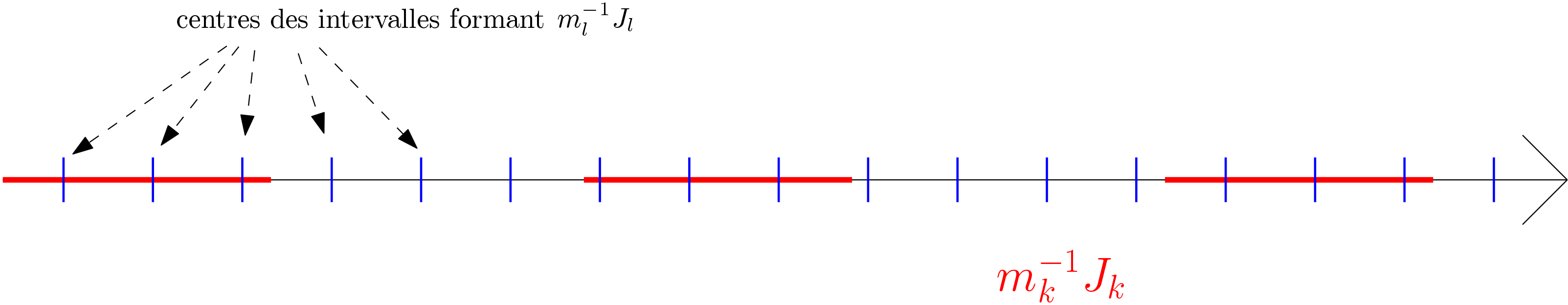}
\end{minipage}$$
soit ces deux centres sont dans deux intervalles consécutifs de $m_k^{-1}J_k$, auquel cas les centres vont se décaler vers le bord droit des intervalles de $m_k^{-1}J_k$ jusqu'à sortir de $m_k^{-1}J_k$ et il faudra alors au moins presque autant de décalages (donc de centres consécutifs) hors de $m_k^{-1}J_k$ pour atteindre à nouveau $m_k^{-1}J_k$.
$$\begin{minipage}[l]{17cm}
\includegraphics[height=5.78cm]{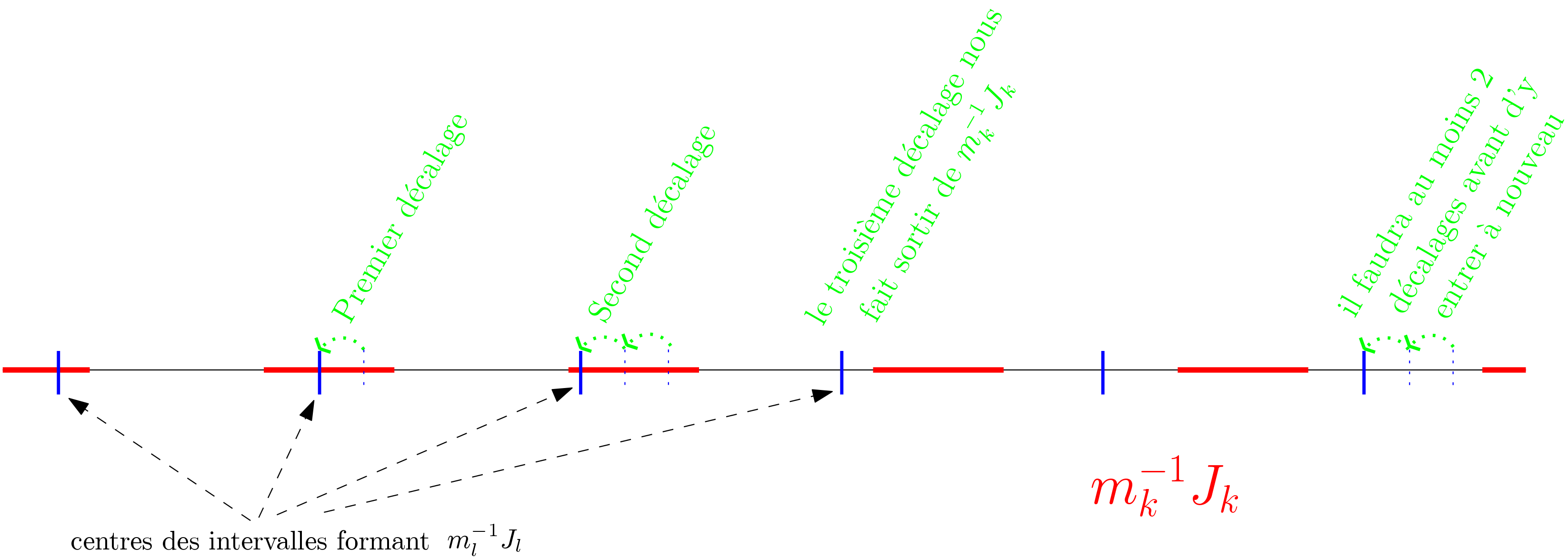}
\end{minipage}$$
\end{proof}
On montre ainsi que $m$ ne dépend pas de $k$\footnote{Le lecteur intéressé pourra consulter https://hal.univ-lorraine.fr/tel-03368154v1 où les détails de cette étape sont donnés.}. 
Commençons par estimer $\mu(J)$, $\mu(\tilde{A}_k)$ et $\mu(I_k)$. Posons $J=\bigcap\limits_{\substack{k=1}}^K J_k$. Par construction $m^{-1}J$ contient $B$, et de plus on a
$$\mu(J)\leqslant b+\min\limits_{\substack{i=1,...,K}}\varepsilon^2_i b, $$
d'où par \eqref{maj_eps_2}
\begin{equation}\label{maj_mu(J)}
\mu ( J)< \left(1+\varepsilon\right)b.
\end{equation}
Enfin comme $\mu(J_k)<1/2$ (cf. preuve du lemme \ref{tildes-presque-intervalles}), on a également
\begin{equation}\label{mu(J)<1/2}
\mu ( J)\leqslant \frac{1}{2}.
\end{equation}
\begin{lemma}\label{mu(A_K)}
On a 
$$\mu\left( \tilde{A}_K\right)=\delta b +\dfrac{1}{K}\sum\limits_{\substack{k=1}}^{K-1} k\left(\varepsilon^1_{k+1} -\varepsilon^2_{k+1}\right) b .$$
\end{lemma}
\begin{proof}
D'après \eqref{presque_égalités}, on a pour tout $k=2,...,K$
$$\mu\left(\tilde{A}_{k-1}\right)+\varepsilon^1_k b = \mu\left(\tilde{S}_{k}\right)=\mu\left(\tilde{A}_{k}\right)+b+\varepsilon^2_k b,$$
donc
$$\mu\left(\tilde{A}_{k-1}\right)=\mu\left(\tilde{A}_{k}\right)+b+\varepsilon^2_k b-\varepsilon^1_k b,$$
d'où en itérant,
\begin{equation}\label{Ak_en_fct_de_AK}
\mu\left(\tilde{A}_{k}\right)=\mu\left(\tilde{A}_{K}\right)+(K-k)b+\sum\limits_{\substack{i=k+1}}^K \left(\varepsilon^2_i -\varepsilon^1_i\right) b.
\end{equation}
On obtient finalement
\begin{align*}
\lambda(A) & = \sum\limits_{\substack{k=1}}^K\mu\left(\tilde{A}_{k}\right)  =\sum\limits_{\substack{k=1}}^K \left[ \mu\left(\tilde{A}_{K}\right)+(K-k)b+\sum\limits_{\substack{i=k+1}}^K \left(\varepsilon^2_i -\varepsilon^1_i\right) b\right] \\ & = K\mu\left(\tilde{A}_{K}\right)+\dfrac{K(K-1)}{2}b+\sum\limits_{\substack{k=1}}^{K-1} k\left(\varepsilon^2_{k+1} -\varepsilon^1_{k+1}\right) b,
\end{align*}
ainsi
\begin{align*}
\mu\left(\tilde{A}_{K}\right) & = \dfrac{\lambda(A)}{K}-\dfrac{K-1}{2}b-\dfrac{1}{K}\sum\limits_{\substack{k=1}}^{K-1} k\left(\varepsilon^2_{k+1} -\varepsilon^1_{k+1}\right) b \\ & = \delta b +\dfrac{1}{K}\sum\limits_{\substack{k=1}}^{K-1} k\left(\varepsilon^1_{k+1} -\varepsilon^2_{k+1}\right) b.
\end{align*}
\end{proof}
Un encadrement précis de la mesure de $I_K$ découle directement de ce lemme et fait l'objet du corollaire suivant.
\begin{cor}\label{mu(I_K)}
On a 
$$\delta b -K\left(\log K -1\right)\varepsilon b \leqslant \mu\left(I_K\right)\leqslant \delta b +\dfrac{1}{K}\sum\limits_{\substack{k=1}}^{K-1} k\left(\varepsilon^1_{k+1} -\varepsilon^2_{k+1}\right) b + \varepsilon^2_K b ,$$
et donc en particulier
$$ \mu\left(I_K\right)\leqslant \delta b +2 \varepsilon b .$$
\end{cor}
\begin{proof}
La minoration provient du fait que $\tilde{A}_{K}\subseteq I_K$, donc par le lemme \ref{mu(A_K)}
$$\mu(I_K)  \geqslant \mu\left(\tilde{A}_{K}\right) \geqslant \delta b -\dfrac{b}{K}\sum\limits_{\substack{k=1}}^{K-1} k\varepsilon^2_{k+1}   \geqslant \delta b -\dfrac{b}{K}\sum\limits_{\substack{k=1}}^{K-1} k\dfrac{K}{K+1-(k+1)}\varepsilon,$$
où la dernière inégalité est obtenue grâce à \eqref{maj_eps_2}. Ainsi
$$\mu(I_K) \geqslant \delta b -b\varepsilon\sum\limits_{\substack{k=1}}^{K-1} \dfrac{k}{K-k}  \geqslant \delta b -b\varepsilon\sum\limits_{\substack{k=1}}^{K-1}\left( \dfrac{K}{k}-1\right)  \geqslant \delta b -b\varepsilon K\left(\log K -1\right).$$
La majoration est immédiate par le lemme \ref{mu(A_K)} puisque $\mu\left(I_K\right)\leqslant \mu\left(\tilde{A}_{K}\right)+ \varepsilon^2_K b$. Ainsi on a par \eqref{eps=somme} et \eqref{maj_eps_2}
\begin{align*}
\mu\left(I_K\right) & \leqslant \delta b +\dfrac{1}{K}\sum\limits_{\substack{k=1}}^{K-1} k\left(\varepsilon^1_{k+1} -\varepsilon^2_{k+1}\right) b + \varepsilon^2_K b \leqslant \delta b +\dfrac{1}{K}\sum\limits_{\substack{k=1}}^{K-1}\varepsilon^1_{k+1} b -\frac{K-1}{K}\varepsilon^2_K b + \varepsilon^2_K b \\ & \leqslant \delta b +\dfrac{1}{K}\sum\limits_{\substack{k=1}}^{K-1}\varepsilon^1_{k+1} b +\frac{1}{K}\varepsilon^2_K b \leqslant \delta b +2 \varepsilon b.
\end{align*}
\end{proof}
À présent afin de simplifier l'argumentation, on va translater $A$ afin qu'il soit suffisamment loin de $0$. Maintenant que nous savons que $m$ ne dépend pas de $k$, et comme $\tilde{A}_{k+1}\subseteq\tilde{A}_{k}$ quel que soit $k\in\left\lbrace 1,...,K-1 \right\rbrace$, on a
$$m\tilde{A}_K\subseteq m\tilde{A}_{K-1}\subseteq ...\subseteq m\tilde{A}_2\subseteq m\tilde{A}_1.$$
\label{def_I_k}
Ainsi, comme $I_k$ est l'enveloppe convexe de $m\tilde{A}_k$, quel que soit $k\in\left\lbrace 2,...,K \right\rbrace$ on a
\begin{equation}\label{I_k_dans_I_k-1}
I_{k}\subseteq I_{k-1}.
\end{equation}
Par \eqref{Ak_en_fct_de_AK}, le lemme \ref{mu(A_K)} et l'hypothèse \eqref{hyp_0}, on a
$$ \mu ( \tilde{A_1})+\varepsilon^2_1 b< 1-b-\varepsilon^2_1 b ,  $$
donc d'après le lemme \ref{tildes-presque-intervalles}, on a
$$\mu(I_1)  < 1-b-\varepsilon^2_1 b  < 1-b.$$
Ainsi il existe $a\in\mathbb{T}$ tel que $d\left( 0, I_1+a\right)>\dfrac{b}{2}$ et donc, quitte à translater $A$ par $\dfrac{a}{m}$, on peut supposer que $d\left( 0, I_1\right)>\dfrac{b}{2}$. En effet, quel que soit $k\in\left\lbrace 1,...,K \right\rbrace$, on a
$$m\widetilde{\left( A+\frac{a}{m}\right)}_k=m \left(\frac{a}{m}+\tilde{A}_k\right)=a+m\tilde{A}_k\subseteq a+I_k.$$
De cette manière, et par \eqref{I_k_dans_I_k-1}, on peut supposer que
\begin{equation}\label{I_k-loin-de-0}
d\left( 0, I_k\right)>\dfrac{b}{2},
\end{equation}
quel que soit $k\in\left\lbrace 1,...,K \right\rbrace$. En particulier, $0\notin I_1$ et $0\notin \tilde{A}_K$. 

Nous allons désormais montrer que $m=1$.

\subsubsection{Premières informations et début de la stratégie pour montrer que $m=1$}

Rappelons que $\pi$ désigne la projection de $\mathbb{R}$ sur $\mathbb{T}$. Pour plus de lisibilité, posons $J'=\pi^{-1}\left( J\right)\cap\left[0,1\right]$ et $I_1'=\pi^{-1}\left( I_1\right)\cap\left[0,1\right].$
Écrivons $J'=J_+\sqcup J_-$ où $J_+$ est un intervalle fermé de $[0,1/2[$ et $J_-$ un intervalle fermé de $[1/2,1]$ (bien définis car $0\in B$ par hypothèse, $B\subseteq J'$ et $\lambda(J')=\mu(J)<\frac{1}{2}$ d'après \eqref{mu(J)<1/2}). On pose $b_+=\mu(J_+)$ et $b_-=\mu(J_-)$. Supposons que $m>1$. Comme $0\in B$ \\
$B=\bigcup\limits_{\substack{l=0}}^m B_l$ où $B_0=B\cap \dfrac{J_+}{m}$, $B_m=B\cap\dfrac{J_-+m-1}{m}$ et $B_l=B\cap \dfrac{l+\big( J_+\cup (J_--1)\big)}{m}$ ($0<l<m$). 
Par \eqref{I_k-loin-de-0}, $0\notin I_1$ et donc $I_1'$ est un intervalle de $\mathbb{R}$. 
\label{def_A_i}
On a donc simplement $A=\bigsqcup\limits_{\substack{l=0}}^L A_l$ où $A_l=A\cap \dfrac{l+I_1'}{m}$ et $L=\max\left\lbrace l \ \vert \ A_l\neq\varnothing \right\rbrace$. 

Écrivons $A+B=S=\bigcup\limits_{\substack{l=0}}^{L+m}S_l$ où $S_l=(A+B)\cap\dfrac{l+\big( J_+\cup (J_--1)\big)+I_1'}{m}$. Remarquons que pour tout $(i,j)\in\left[ 0,m\right]^2$ tel que $A_i\neq\varnothing$, $B_j\neq\varnothing$ et $i+j=l$, on a $A_i+B_j\subseteq S_{l}$. Enfin posons $\mathscr{L}_A=\left\lbrace l\geqslant 0 \ \vert \ A_l\neq\varnothing \right\rbrace$ et $\mathscr{L}_B=\left\lbrace l\geqslant 0 \ \vert \ B_l\neq\varnothing \right\rbrace$. On rappelle que $\max\mathscr{L}_A=L$, que $\max\mathscr{L}_B=m$ et que $\min\mathscr{L}_B=0$. 
$$ $$
$$\begin{minipage}[l]{17cm}
\includegraphics[height=3.8cm]{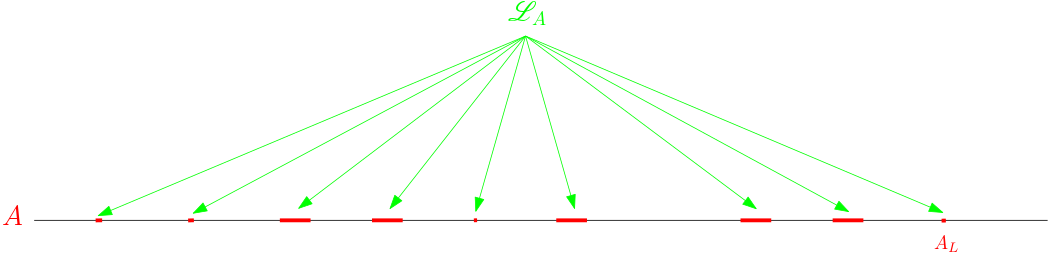}
\end{minipage}$$
$$ $$
$\mathscr{L}_A$ contient tous les morceaux non vides de $A$ mais donc également ceux de très petite mesure. Nous allons également avoir besoin de nous concentrer sur les morceaux offrant une certaine contribution au niveau de la mesure. Pour celà, on pose
$$\mathcal{L}_A=\left\lbrace l\geqslant 0 \ \vert \ \lambda(A_l)\geqslant\dfrac{\mu(I_{K})}{m}-\dfrac{f\left(\varepsilon,K\right)}{m}b \right\rbrace,$$
où $f\left(\varepsilon,K\right)$ est une fonction qu'on optimisera,
$i_A=\min\mathcal{L}_A$, $I_A=\max\mathcal{L}_A$ et $B_M$ est le plus gros "morceau" de $B$, c'est à dire tel que $\lambda(B_M)=\max\left\lbrace \lambda(B_i) \ \vert \ i\in\mathscr{L}_B \right\rbrace$. 
$$ $$
$$\begin{minipage}[l]{17cm}
\includegraphics[height=3.8cm]{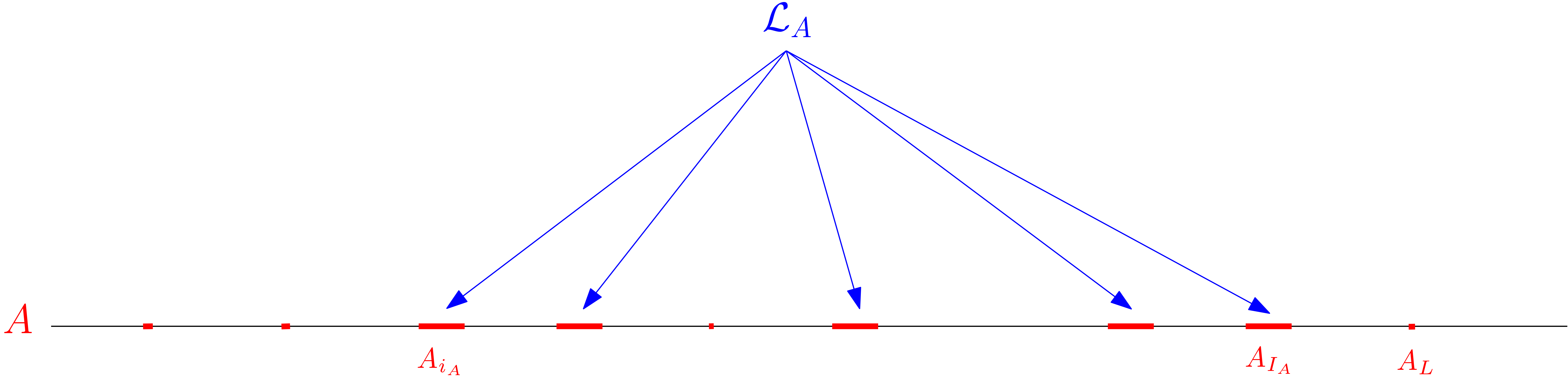}
\end{minipage}$$
$$ $$
\begin{Remarque}\label{pasdepbpourlechoixdef}
Pour que $A_l$ soit de mesure strictement positive quel que soit $l\in\mathcal{L}_A$, il faut que ${\mu(I_K)-f\left(\varepsilon,K\right)b>0}$. Or rappelons que par le corollaire \ref{mu(I_K)}
$$\mu(I_K)-f\left(\varepsilon,K\right)b\geqslant\left(\delta- K\left(\log(K)-1\right)\varepsilon-f\left(\varepsilon,K\right)\right)b.$$
Il suffit donc que $\left(\delta-K\left( \log(K)-1\right)\varepsilon-f\left(\varepsilon,K\right)\right)$ soit strictement positif. On sera amenés à choisir soit $f\left(\varepsilon,K\right)=K\varepsilon^{1/2}$ soit $f\left(\varepsilon,K\right)=K\varepsilon^{1/3}$ et on imposera donc
$$\delta-K\left( \log(K)-1\right)\varepsilon-K\varepsilon^{1/3}>0 ,$$
ce qui est assuré par l'hypothèse \eqref{hyp_1}, $\varepsilon<\left(\dfrac{\delta}{3K}\right)^3$.
\end{Remarque}
Ainsi grâce à cette hypothèse, on a bien \label{hyp} $\dfrac{\mu(I_{K})}{m}-\dfrac{f\left(\varepsilon,K\right)}{m}b>0.$
Finalement, quel que soit $i\in\mathcal{L}_A$, on a bien $\lambda(A_i)>0$, donc $\lambda(S_{i+j})\geqslant\lambda(A_i)+\lambda(B_j)$ pour tout $j\in\mathscr{L}_B $. Ainsi on a
\begin{align*}
\lambda(A+B) & =\lambda(A)+\left(K+\delta+\varepsilon\right)b =\sum\limits_{\substack{l=0}}^{L+m}\lambda(S_l) \\ & \geqslant \sum\limits_{\substack{l=0}}^{i_A}\lambda(S_l) + \sum\limits_{\substack{l=i_A+1}}^{i_A+M}\lambda(S_l) + \sum\limits_{\substack{l=i_A+M+1}}^{I_A+M}\lambda(S_l) + \sum\limits_{\substack{l=I_A+M+1}}^{I_A+m}\lambda(S_l)+\sum\limits_{\substack{l=I_A+m+1}}^{L+m}\lambda(S_l) \\ & \geqslant \sum\limits_{\substack{l=0 \\ l\in\mathscr{L}_A}}^{i_A}\lambda(B_0+A_l) + \sum\limits_{\substack{l=1 \\ l\in\mathscr{L}_B}}^{M}\lambda(A_{i_A}+B_l) + \sum\limits_{\substack{l=i_A+1 \\ l\in\mathscr{L}_A}}^{I_A}\lambda(B_{M}+A_l) \\ & \ \ \ + \sum\limits_{\substack{l=M+1 \\ l\in\mathscr{L}_B}}^{m}\lambda(A_{I_A}+B_l)+\sum\limits_{\substack{l=I_A+1 \\ l\in\mathscr{L}_A}}^{L}\lambda(B_m+A_l) \\ & \geqslant \sum\limits_{\substack{l=0 \\ l\in\mathscr{L}_A}}^{i_A}\left(\lambda(B_0)+\lambda(A_l)\right) + \sum\limits_{\substack{l=1 \\ l\in\mathscr{L}_B}}^{M}\left(\lambda(A_{i_A})+\lambda(B_l)\right) + \sum\limits_{\substack{l=i_A+1 \\ l\in\mathscr{L}_A}}^{I_A}\left(\lambda(B_{M})+\lambda(A_l)\right) \\ & \ \ \ + \sum\limits_{\substack{l=M+1 \\ l\in \mathscr{L}_B}}^{m}\left(\lambda(A_{I_A})+\lambda(B_l)\right) +\sum\limits_{\substack{l=I_A+1 \\ l\in\mathscr{L}_A}}^{L}\left(\lambda(B_m)+\lambda(A_l)\right) \\ & \geqslant \lambda(A)+b + \lambda(B_M)\times\#\left(\mathscr{L}_A\cap \left\llbracket i_A+1, I_A \right\rrbracket\right)+\lambda(A_{i_A})\times\#\left(\mathscr{L}_B\cap \left\llbracket 1, M \right\rrbracket\right)\\ & \ \ \ +\lambda(A_{I_A})\times\#\left(\mathscr{L}_B\cap \left\llbracket M+1, m \right\rrbracket\right) \\ & \geqslant \lambda(A)+b + \#\left(\mathscr{L}_A\cap \left\llbracket i_A+1, I_A \right\rrbracket\right)\lambda(B_M)+\left(\#\mathscr{L}_B-1\right)\left(\dfrac{\mu(I_{K})}{m}-\dfrac{f\left(\varepsilon,K\right)}{m}b\right) \\ & \geqslant \lambda(A)+b + \left(\#\mathcal{L}_A-1\right)\lambda(B_M)+\left(\#\mathscr{L}_B-1\right)\left(\dfrac{\mu(I_{K})}{m}-\dfrac{f\left(\varepsilon,K\right)}{m}b\right) .
\end{align*}
Ainsi, on a
\begin{equation}\label{debut_m=1}
\lambda(A+B)\geqslant \lambda(A)+b + \left(\#\mathcal{L}_A-1\right)\lambda(B_M)+\left(\#\mathscr{L}_B-1\right)\left(\dfrac{\mu(I_{K})}{m}-\dfrac{f\left(\varepsilon,K\right)}{m}b\right) .
\end{equation}
Pour pouvoir conclure avec cette stratégie, nous avons besoin d'une minoration de $\lambda(B_M)$, de $\#\mathscr{L}_B$ et de $\#\mathcal{L}_A$. Ce qui nous amène aux lemmes de la partie suivante.

\subsubsection{Minoration de $\lambda(B_M)$, $\#\mathscr{L}_B$ et $\#\mathcal{L}_A$}

Les minorations de $\lambda(B_M)$ et $\#\mathscr{L}_B$ font chacune l'objet de l'un des deux lemmes suivants mais pour obtenir le moins de contraintes possible, nous différencierons deux cas pour la minoration de $\#\mathcal{L}_A$. Ainsi nous allons établir deux minorations de $\#\mathcal{L}_A$, une valable quand $m$ est grand et l'autre valable quand $m$ est petit. Elles feront l'objet des deux derniers lemmes de cette partie. Nous n'aurons alors qu'à appliquer la minoration correspondante suivant la taille de $m$ afin d'obtenir une absurdité dans le cas $m\geqslant 2$ en ayant le moins de contraintes possible.
\begin{lemma}\label{min_BM}
Si $m\geqslant 2$ et $B_M$ est le plus gros "morceau" de $B$, c'est à dire tel que \\ ${\lambda(B_M)=\max\left\lbrace \mu(B_i) \ \vert \ i\in\mathscr{L}_B \right\rbrace}$, on a 
$$\lambda(B_M)\geqslant\dfrac{1-\varepsilon}{m}b.$$
\end{lemma}
\begin{proof}
Rappelons qu'on a supposé $m\geqslant 2$. De plus, $B\subset m^{-1}J$ et $\lambda(J)\leqslant (1+\varepsilon)b$. D'où 
\begin{align*}
\lambda(B_M) & = \max\left\lbrace \mu(B_i) \ \vert \ i\in\mathscr{L}_B \right\rbrace \geqslant \max\left\lbrace \mu(B_i) \ \vert \ i\in\mathscr{L}_B\setminus \left\lbrace 0,m \right\rbrace\right\rbrace \\ & \geqslant \dfrac{b-\lambda(J)/m}{m-1} \geqslant \dfrac{b-\dfrac{1+\varepsilon}{m}b}{m-1} \geqslant \dfrac{b}{m}\left(1-\dfrac{\varepsilon}{m-1}\right) \geqslant \dfrac{1-\varepsilon}{m}b.
\end{align*}
\end{proof}
\begin{lemma}\label{min_LB}
Si $m\geqslant 2$ et $\mathscr{L}_B=\left\lbrace l\geqslant 0 \ \vert \ B_l\neq\varnothing \right\rbrace$, on a 
$$\#\mathscr{L}_B\geqslant 1+m\left(1-\dfrac{\varepsilon}{1+\varepsilon}\right).$$
\end{lemma}
\begin{proof}
On a $B\subseteq m^{-1}J$, $\mu(J)\leqslant b+\varepsilon b$ et $B_l=B\cap \dfrac{l+J}{m}$ ($0<l<m$), donc $\lambda(B_0)+\lambda(B_m)\leqslant\dfrac{1+\varepsilon}{m}b$ et $\lambda(B_l)\leqslant\dfrac{1+\varepsilon}{m}b$ pour tout $l\in\left\lbrace 1,...,m-1 \right\rbrace$ et ainsi
$$b = \sum\limits_{\substack{l\in\mathscr{L}_B}} \lambda(B_l) = \lambda(B_0)+\lambda(B_m)+\sum\limits_{\substack{l\in\mathscr{L}_B \\ l \neq 0,m}}^{}\lambda(B_l)  \leqslant \dfrac{1+\varepsilon}{m}b+\sum\limits_{\substack{l\in\mathscr{L}_B \\ l \neq 0,m}}\dfrac{1+\varepsilon}{m}b,$$
d'où $\#\left(\mathscr{L}_B\setminus \left\{0,m\right\}\right)+1\geqslant \dfrac{m}{1+\varepsilon}$. Comme $0,1\in B$, on a $0,m\in\mathscr{L}_B$ et donc finalement 
$$\#\mathscr{L}_B\geqslant \dfrac{m}{1+\varepsilon}+1 \geqslant1+m\left(1-\dfrac{\varepsilon}{1+\varepsilon}\right).$$
\end{proof}
Pour établir les deux minorations de $\#\mathcal{L}_A$ (suivant la taille de $m$), nous allons avoir besoin de quelques estimations. Commençons par quelques définitions. On pose
\begin{itemize}
\item $\tilde{A}^l_K=\tilde{A}_K\cap \dfrac{l+I_K'}{m}$ ($0\leqslant l<m$), où $ I_K'=\pi^{-1}(I_K)\cap\left[ 0,1\right] $;
\item $\mathscr{L}=\left\lbrace l\geqslant 0 \ \vert \ \tilde{A}^l_K\neq\varnothing \right\rbrace$ et $ N= \# \mathscr{L} $;
\item $L_f=\left\lbrace l\geqslant 0 \ \vert \ \mu\left(\tilde{A}^l_K\right)\geqslant \dfrac{\mu(I_K)}{m}-\dfrac{f\left(\varepsilon,K\right)}{m}b \right\rbrace$ et $n_f = \# L_f;$
\item $\Omega\left(\tilde{A}_K^l\right)=\left\lbrace i\in\mathscr{L}_A \ \vert \ \left( A_i\mod 1 \right)\cap \tilde{A}_K^l \neq\varnothing\right\rbrace, $ pour tout $l\in \mathscr{L}$;
\item $\sigma_1\left(\tilde{A}_K^l\right)=\#\Omega\left(\tilde{A}_K^l\right)$.
\end{itemize}
$$ $$
$$\begin{minipage}[l]{17cm}
\includegraphics[height=15cm]{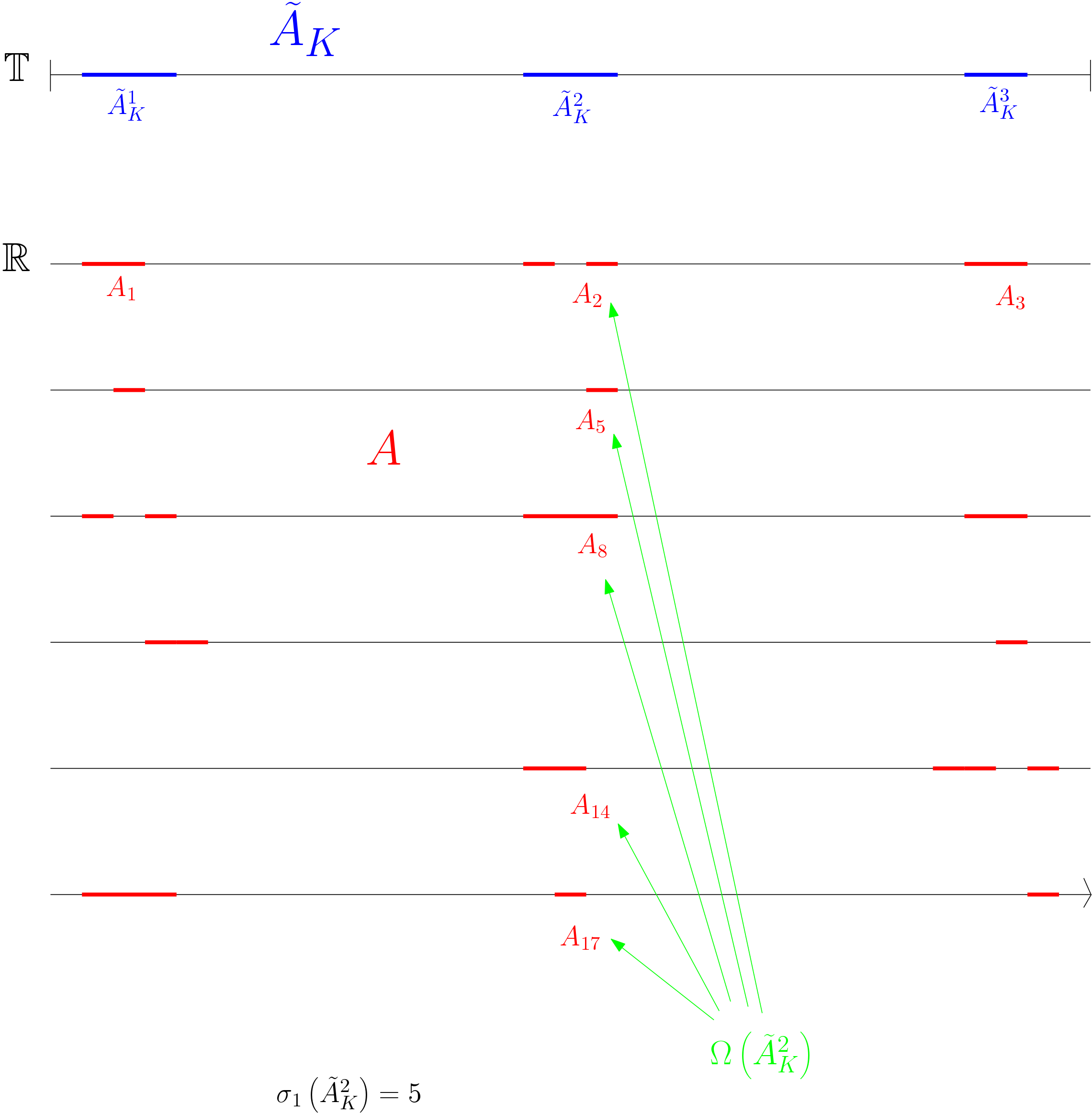}
\end{minipage}$$
$$ $$
Nous allons donner une minoration de $\sigma_1\left(\tilde{A}_K^l\right)$, de $N$ puis de $n_f$. Par définition de $\tilde{A}_K$, pour tout $l\in \mathscr{L}$, on a 
\begin{equation}\label{min_sigma1}
\sigma_1\left(\tilde{A}_K^l\right)\geqslant K.
\end{equation}
D'autre part
$$ \mu(I_K)-\varepsilon^2_Kb \leqslant \mu\left(\tilde{A}_K\right)\leqslant N\frac{\mu(I_K)}{m},$$
d'où
\begin{equation}\label{min_N}
N\geqslant m\left( 1-\frac{\varepsilon^2_K b}{\mu(I_K)}\right) .
\end{equation}
De plus
\begin{align*}
\mu(I_K)-\varepsilon^2_Kb & \leqslant \mu\left(\tilde{A}_K\right) = \sum\limits_{\substack{l=0}}^{m-1}\mu\left(\tilde{A}_K^l\right) \\ & \leqslant n_f\dfrac{\mu(I_K)}{m}+(m-n_f)\left(\dfrac{\mu(I_K)}{m}-\frac{f\left(\varepsilon,K\right)}{m}b\right) \\ & \leqslant n_f\frac{f\left(\varepsilon,K\right)}{m}b+\mu(I_K)-f\left(\varepsilon,K\right)b,
\end{align*}
d'où $n_f\geqslant m-\dfrac{\varepsilon^2_K b }{f\left(\varepsilon,K\right)b/m} ,$ et donc finalement
\begin{equation}\label{min_nf}
n_f\geqslant m\left(1-\dfrac{\varepsilon^2_K}{f\left(\varepsilon,K\right)}\right).
\end{equation}
Nous allons également avoir recourt à la majoration de $\#\mathscr{L}_A$ suivante.
\begin{equation}\label{LA<=Km}
\#\mathscr{L}_A\leqslant Km.
\end{equation}
En effet, on a
\begin{align*}
& \lambda(A)+\left(K+\delta+\varepsilon\right)b = \lambda(A+B) \\ & \geqslant \lambda(B_0)+...+\lambda(B_{M-1})+\#\mathscr{L}_A\times \lambda(B_{M})+ \lambda(A)+\lambda(B_{M+1})+...+\lambda(B_m) \\ & \geqslant \lambda(A)+b-\lambda(B_{M})+\#\mathscr{L}_A\times\lambda(B_{M}).
\end{align*}
Or $\lambda(B_{M})\geqslant\dfrac{1-\varepsilon}{m}b,$ donc
$$\#\mathscr{L}_A \leqslant \dfrac{m\left( K-1+\delta+\varepsilon\right)}{1-\varepsilon}+1  \leqslant Km+\dfrac{m\left( \varepsilon (K+1)+\delta-1\right)}{1-\varepsilon}+1 .$$
Mais $\varepsilon<\dfrac{1-\delta}{K+1}$ car par hypothèse $\varepsilon<\left(\dfrac{1-\delta}{K}\right)^2$ et $K\geqslant 2$ (cf. remarque \ref{K>1}). Ainsi $\varepsilon (K+1)+\delta-1<0$ et donc $\#\mathscr{L}_A< Km +1$. D'où finalement $\#\mathscr{L}_A\leqslant Km.$

Nous pouvons désormais établir les deux minorations de $\#\mathcal{L}_A$. Le lemme suivant va nous donner une bonne minoration de $\#\mathcal{L}_A$ quand $m$ est assez petit.
\begin{lemma}\label{min_LA_m<}
Si $2\leqslant m<\dfrac{\mu(I_K)}{\varepsilon^2_K b}$ et $\mathcal{L}_A=\left\lbrace l\geqslant 0 \ \vert \ \lambda(A_l)\geqslant\dfrac{\mu(I_{K})}{m}-\dfrac{f\left(\varepsilon,K\right)}{m}b \right\rbrace$, on a 
$$\#\mathcal{L}_A \geqslant Km\left(1-\dfrac{\varepsilon^2_K}{f\left(\varepsilon,K\right)}\right).$$
\end{lemma}
\begin{proof}
$\tilde{A}_K\subseteq m^{-1}I_K$ et $\mu(I_K)\leqslant \mu(\tilde{A}_K)+\varepsilon_K^2 b$. Par \eqref{min_N}, on a $ N \geqslant m\left( 1-\frac{\varepsilon^2_K b}{\mu(I_K)}\right),$ et donc $N>m-1$ car  $m<\dfrac{\mu(I_K)}{\varepsilon^2_K b}$ par hypothèse du lemme. De plus, comme $N$ est un entier inférieur ou égal à $m$ par définition, on a
\begin{equation}\label{N=m}
N=m.
\end{equation}
Ainsi les $m$ morceaux de $\tilde{A}_K$ sont non vides et ils résultent chacun d'une contribution d'au moins $K$ morceaux de $A$ (par définition de $\tilde{A}_K$) et donc $\#\mathscr{L}_A$ va contenir au moins $Km$ éléments. En effet, par \eqref{min_sigma1}, on a
$$ \#\mathscr{L}_A \geqslant \sum\limits_{\substack{l\in \mathscr{L}}}\sigma_1\left(\tilde{A}_K^l\right) \geqslant N K,$$
d'où, d'après \eqref{N=m}
\begin{equation}\label{L_A>=Km}
\#\mathscr{L}_A\geqslant Km.
\end{equation}
Ainsi, d'après \eqref{L_A>=Km} et \eqref{LA<=Km}, on a
\begin{equation}\label{LA=Km}
\#\mathscr{L}_A= Km.
\end{equation}
De cette manière, en reprenant nos calculs, on a $Km=\#\mathscr{L}_A=\sum\limits_{\substack{l\in \mathscr{L}}}\sigma_1\left(\tilde{A}_K^l\right),$ et donc nécessairement, comme $N=m$ d'après \eqref{N=m}, pour tout $l\in\mathscr{L}$ on a $ \sigma_1\left(\tilde{A}_K^l\right)=K$, ce qui implique finalement, pour tout $i\in\Omega\left(\tilde{A}_K^l\right)$
\begin{equation}
\lambda(A_i)\geqslant\mu\left( \tilde{A}_K^l\right).
\end{equation}
En particulier, pour tout $l\in L_f$, on a pour tout $i\in\Omega\left(\tilde{A}_K^l\right)$
$$\lambda(A_i)\geqslant \dfrac{\mu(I_{K})}{m}-\dfrac{f\left(\varepsilon,K\right)}{m}b,$$
et donc $i\in\mathcal{L}_A$. De cette manière $ \#\mathcal{L}_A \geqslant \sum\limits_{\substack{l\in L_f}}\sigma_1\left(\tilde{A}_K^l\right)$, d'où
\begin{equation}\label{L_A>Kn_f}
\#\mathcal{L}_A\geqslant Kn_f .
\end{equation} 
Finalement, avec \eqref{min_nf}, on obtient la minoration $\#\mathcal{L}_A\geqslant Km\left(1-\dfrac{\varepsilon^2_K}{f\left(\varepsilon,K\right)}\right).$
\end{proof}
Nous pourrons utiliser cette minoration de $\#\mathcal{L}_A$ si $m$ est suffisamment petit et il nous restera donc à traiter le cas où $m$ est grand. Pour cela, nous utiliserons une autre minoration de $\#\mathcal{L}_A$ qui fait l'objet du lemme suivant.
\begin{lemma}\label{min_LA_m>}
Si $ m\geqslant\dfrac{\mu(I_K)}{\varepsilon^2_K b}$ et $\mathcal{L}_A=\left\lbrace l\geqslant 0 \ \vert \ \lambda(A_l)\geqslant\dfrac{\mu(I_{K})}{m}-\dfrac{f\left(\varepsilon,K\right)}{m}b \right\rbrace$, on a 
$$\#\mathcal{L}_A \geqslant Km\left(1-\dfrac{g\left(\varepsilon,K\right)}{f\left(\varepsilon,K\right)}-K\varepsilon\left( \dfrac{1}{g\left(\varepsilon,K\right)}-\frac{b}{\mu(I_K)}\right)\right),$$
où $g$ est telle que $0<g(\varepsilon,K)<f(\varepsilon,K)$ et est à optimiser.
\end{lemma}
\begin{proof}
On a $\tilde{A}_K\subseteq m^{-1}I_K$ et $\mu(I_K)\leqslant \mu(\tilde{A}_K)+\varepsilon_K^2 b$. Nous allons avoir besoin de définir des objets similaires à ceux apparaissant dans la preuve du lemme précédent. On pose $\tilde{A}^l_K=\tilde{A}_K\cap \dfrac{l+I_K'}{m}$ ($0\leqslant l<m$), où $ I_K'=\pi^{-1}(I_K)\cap\left[ 0,1\right] $. Soit $g(\varepsilon,K)$ une fonction positive telle que $g(\varepsilon,K)<f(\varepsilon,K)$ que nous déterminerons plus tard. Rappelons que $n_g=\#L_g$ et que par \eqref{min_sigma1}, pour tout $l\in L_g$, on a $\sigma_1\left(\tilde{A}_K^l\right)\geqslant K.$ Ainsi pour tout $l\in\ L_g$, on définit $t_l=\sigma_1\left(\tilde{A}_K^l\right)- K.$ Par \eqref{LA<=Km}, on a
\begin{align*}
Km & \geqslant \#\mathscr{L}_A  \geqslant \sum\limits_{\substack{l\in L_g}}\sigma_1\left(\tilde{A}_K^l\right)+ K\left( N-n_g\right) \\ & \geqslant \sum\limits_{\substack{l\in L_g}} \left( K+t_l\right) + K\left( N-n_g\right) \geqslant Kn_g+\sum\limits_{\substack{l\in L_g}} t_l + K\left( N-n_g\right) \geqslant KN+\sum\limits_{\substack{l\in L_g}} t_l .
\end{align*}
Ainsi $\sum\limits_{\substack{l\in L_g}} t_l\leqslant K\left( m-N\right)$ et, en utilisant \eqref{min_N}, on obtient 
\begin{equation}\label{somme_tl}
\sum\limits_{\substack{l\in L_g}} t_l\leqslant Km\frac{\varepsilon^2_K b}{\mu(I_K)}.
\end{equation}
Enfin on pose pour tout $l\in L_g$
$$\sigma_2\left(\tilde{A}_K^l\right)=\#\left\lbrace A_i \ \vert \ \left( A_i\mod 1\right) \cap \tilde{A}_K^l\neq\varnothing \ \text{ et } \ \lambda\left( A_i\cap \tilde{A}_K^l\right)\geqslant\dfrac{\mu(I_{K})}{m}-\dfrac{f\left(\varepsilon,K\right)}{m}b  \right\rbrace.$$
Ainsi pour tout $l\in L_g$ on a
\begin{align*}
& K\left(\dfrac{\mu(I_K)}{m}-\dfrac{g\left(\varepsilon,K\right)}{m}b\right) \leqslant \sum\limits_{\substack{\ell\in\Omega\left(\tilde{A}_K^l\right)}}\lambda\left( A_\ell\cap \tilde{A}_K^l\right) \\ & \leqslant \left(\sigma_1\left(\tilde{A}_K^l\right)-\sigma_2\left(\tilde{A}_K^l\right)\right)\left(\dfrac{\mu(I_K)}{m}-\dfrac{f\left(\varepsilon,K\right)}{m}b\right)+\sigma_2\left(\tilde{A}_K^l\right)\dfrac{\mu(I_K)}{m} \\ & \leqslant \left(K+t_l-\sigma_2\left(\tilde{A}_K^l\right)\right)\left(\dfrac{\mu(I_K)}{m}-\dfrac{f\left(\varepsilon,K\right)}{m}b\right) +\sigma_2\left(\tilde{A}_K^l\right)\dfrac{\mu(I_K)}{m},
\end{align*}
d'où
$$\sigma_2\left(\tilde{A}_K^l\right)\geqslant K\dfrac{f\left(\varepsilon,K\right)-g\left(\varepsilon,K\right)}{f\left(\varepsilon,K\right)}-\dfrac{t_l\big(\mu(I_K)-f\left(\varepsilon,K\right)b\big)}{f\left(\varepsilon,K\right)b}.$$
Et finalement 
\begin{align*}
\#\mathcal{L}_A & \geqslant \sum\limits_{\substack{l\in L_g}}\sigma_2\left(\tilde{A}_K^l\right) \geqslant \sum\limits_{\substack{l\in L_g}} \left( K\dfrac{f\left(\varepsilon,K\right)-g\left(\varepsilon,K\right)}{f\left(\varepsilon,K\right)}-\dfrac{t_l\big(\mu(I_K)-f\left(\varepsilon,K\right)b\big)}{f\left(\varepsilon,K\right)b}\right) \\ & \geqslant m\left(1-\dfrac{\varepsilon^2_K}{g\left(\varepsilon,K\right)}\right)K\dfrac{f\left(\varepsilon,K\right)-g\left(\varepsilon,K\right)}{f\left(\varepsilon,K\right)}-\dfrac{\mu(I_K)-f\left(\varepsilon,K\right)b}{f\left(\varepsilon,K\right)b}\sum\limits_{\substack{l\in L_g}}t_l \\ & \geqslant m\left(1-\dfrac{\varepsilon^2_K}{g\left(\varepsilon,K\right)}\right)K\dfrac{f\left(\varepsilon,K\right)-g\left(\varepsilon,K\right)}{f\left(\varepsilon,K\right)}-\dfrac{\mu(I_K)-f\left(\varepsilon,K\right)b}{f\left(\varepsilon,K\right)b}Km\dfrac{\varepsilon^2_K b}{\mu(I_K)} \\ & \geqslant m\left(1-\dfrac{\varepsilon^2_K}{g\left(\varepsilon,K\right)}\right)K\left(1-\dfrac{g\left(\varepsilon,K\right)}{f\left(\varepsilon,K\right)}\right)-\left(\dfrac{1}{f\left(\varepsilon,K\right)}-\frac{b}{\mu(I_K)}\right)Km\varepsilon^2_K ,
\end{align*}
où la troisième inégalité utilise \eqref{min_nf} et la quatrième \eqref{somme_tl}. Ainsi 
\begin{align*}
\#\mathcal{L}_A & \geqslant Km\left(1-\dfrac{g\left(\varepsilon,K\right)}{f\left(\varepsilon,K\right)}-\varepsilon^2_K\left( \dfrac{1}{g\left(\varepsilon,K\right)}-\dfrac{1}{f\left(\varepsilon,K\right)}+\dfrac{1}{f\left(\varepsilon,K\right)}-\frac{b}{\mu(I_K)}\right)\right) \\ & \geqslant Km\left(1-\dfrac{g\left(\varepsilon,K\right)}{f\left(\varepsilon,K\right)}-\varepsilon^2_K\left( \dfrac{1}{g\left(\varepsilon,K\right)}-\frac{b}{\mu(I_K)}\right)\right) ,
\end{align*} 
Enfin on conclut en utilisant \eqref{maj_eps_2}
$$\#\mathcal{L}_A \geqslant Km\left(1-\dfrac{g\left(\varepsilon,K\right)}{f\left(\varepsilon,K\right)}-K\varepsilon\left( \dfrac{1}{g\left(\varepsilon,K\right)}-\frac{b}{\mu(I_K)}\right)\right).$$
\end{proof}
Cette minoration de $\#\mathcal{L}_A$ est moins bonne que la précédente mais nous ne l'utiliserons que lorsque $m$ est supérieur à $\dfrac{\mu(I_K)}{b\varepsilon^2_K}$ et nous pourrons donc utiliser l'inégalité $m\geqslant \dfrac{\mu(I_K)}{b\varepsilon^2_K}$ pour pallier cette perte. 

Toutes nos minorations sont établies, nous sommes désormais en mesure de conclure. Comme nous l'avons expliqué dans la partie précédente, nous allons distinguer deux cas suivant la taille de $m$. 

\subsubsection{Preuve de $m=1$ lorsque $m<\dfrac{\mu(I_K)}{\varepsilon^2_K b}$}

Nous allons raisonner par l'absurde et on suppose donc dans cette partie que $2\leqslant m<\dfrac{\mu(I_K)}{\varepsilon^2_K b}$ et nous pourrons donc utiliser le lemme \ref{min_LA_m<}. Rappelons l'inégalité \eqref{debut_m=1}
$$\lambda(A+B)\geqslant \lambda(A)+b + \left(\#\mathcal{L}_A-1\right)\lambda(B_M)+\left(\#\mathscr{L}_B-1\right)\left(\dfrac{\mu(I_{K})}{m}-\dfrac{f\left(\varepsilon,K\right)}{m}b\right),$$
et donc avec les lemmes \ref{min_BM}, \ref{min_LB}, \ref{min_LA_m<}, et l'hypothèse $\lambda(A+B)=\lambda(A)+(K+\delta+\varepsilon)b$, on obtient
\begin{align*}
(K+\delta+\varepsilon)b & \geqslant b + Km\left(1-\dfrac{\varepsilon^2_K}{f\left(\varepsilon,K\right)}\right)\dfrac{1-\varepsilon}{m}b-\dfrac{1-\varepsilon}{m}b \\ & \ \ \ +m\left(1-\dfrac{\varepsilon}{1+\varepsilon}\right)\left(\dfrac{\mu(I_{K})}{m}-\dfrac{f\left(\varepsilon,K\right)}{m}b\right).
\end{align*}
Ainsi en utilisant le corollaire \ref{mu(I_K)}, on a 
\begin{align*}
(K+\delta+\varepsilon) & \geqslant 1 + K\left( 1-\dfrac{\varepsilon^2_K}{f\left(\varepsilon,K\right)}-\varepsilon\left( 1-\dfrac{\varepsilon^2_K}{f\left(\varepsilon,K\right)}\right)\right)-\dfrac{1-\varepsilon}{m}  \\ & \ \ \ +m\left(1-\dfrac{\varepsilon}{1+\varepsilon}\right)\left(\dfrac{\delta-K\left( \log(K)-1\right)\varepsilon}{m}-\dfrac{f\left(\varepsilon,K\right)}{m}\right)\\ & \geqslant 1 + K\left( 1-\dfrac{\varepsilon^2_K}{f\left(\varepsilon,K\right)}-\varepsilon\left( 1-\dfrac{\varepsilon^2_K}{f\left(\varepsilon,K\right)}\right)\right)-\dfrac{1-\varepsilon}{m}  \\ & \ \ \ +\left(1-\dfrac{\varepsilon}{1+\varepsilon}\right)\left(\delta-K\left( \log(K)-1\right)\varepsilon-f\left(\varepsilon,K\right)\right).
\end{align*}
Finalement, on a donc
\begin{align*}
0 & \geqslant  1 - K\left( \dfrac{\varepsilon^2_K}{f\left(\varepsilon,K\right)}+\varepsilon\left( 1-\dfrac{\varepsilon^2_K}{f\left(\varepsilon,K\right)}\right)\right)-\dfrac{1-\varepsilon}{m}  \\ & \ \ \ -\left(1-\dfrac{\varepsilon}{1+\varepsilon}\right)\left(K\left( \log(K)-1\right)\varepsilon+f\left(\varepsilon,K\right)\right)-\delta\dfrac{\varepsilon}{1+\varepsilon}-\varepsilon \\ & \geqslant  1 - K\left( \dfrac{\varepsilon^2_K}{f\left(\varepsilon,K\right)}+\varepsilon\right)-\dfrac{1-\varepsilon}{m} -\big(K\left( \log(K)-1\right)\varepsilon+f\left(\varepsilon,K\right)\big)-\delta\varepsilon-\varepsilon.
\end{align*}
Ainsi par \eqref{maj_eps_2}
$$K\left( \dfrac{K\varepsilon}{f\left(\varepsilon,K\right)}+\varepsilon\right)+K\left( \log(K)-1\right)\varepsilon+f\left(\varepsilon,K\right)+\delta\varepsilon+\varepsilon \geqslant  1-\frac{1}{m},$$
et comme on a supposé $m\geqslant2$
$$K\left( \dfrac{K\varepsilon}{f\left(\varepsilon,K\right)}+\varepsilon\right)+K\left( \log(K)-1\right)\varepsilon+f\left(\varepsilon,K\right)+\delta\varepsilon+\varepsilon \geqslant \frac{1}{2}.$$
Pour aboutir à une absurdité, on voudrait que le terme de gauche soit le plus petit possible, ce qui nous conduit à choisir $f\left(\varepsilon,K\right)=\sqrt{\varepsilon}K$ (ce qui ne pose pas de souci, cf. remarque \ref{pasdepbpourlechoixdef}). On obtient alors
\begin{equation}\label{contradiction_m=1_m<}
\varepsilon\left(K\log K+\delta+1\right)+2K\sqrt{\varepsilon} \geqslant  \frac{1}{2}.
\end{equation}
Or d'autre part, $\varepsilon<\left(\dfrac{\delta}{3K}\right)^3$ par l'hypothèse \eqref{hyp_1}, donc
\begin{align*}
\varepsilon\left(K\log K+\delta+1\right)+2K\sqrt{\varepsilon} & < \left(\dfrac{\delta}{3K}\right)^3\left(K\log K+\delta+1\right)+2K\left(\dfrac{\delta}{3K}\right)^{3/2} \\ & < \dfrac{K\log K+2}{27K^3}+\frac{2}{3\sqrt{3K}} <\frac{1}{2} \ \ \ \text{quel que soit $K\geqslant2$},
\end{align*}
ce qui contredit \eqref{contradiction_m=1_m<} et nous donne donc l'absurdité. On ne peut donc pas avoir $m\geqslant 2$ et donc nécessairement $m=1$.

\subsubsection{Le cas $m\geqslant\dfrac{\mu(I_K)}{\varepsilon^2_K b}$ est impossible}

On suppose donc ici que $m$ est supérieur ou égal à $\dfrac{\mu(I_K)}{\varepsilon^2_K b}$ et on pourra donc utiliser le lemme \ref{min_LA_m>}. Rappelons l'inégalité \eqref{debut_m=1} 
$$\lambda(A+B)\geqslant \lambda(A)+b + \left(\#\mathcal{L}_A-1\right)\lambda(B_M)+\left(\#\mathscr{L}_B-1\right)\left(\dfrac{\mu(I_{K})}{m}-\dfrac{f\left(\varepsilon,K\right)}{m}b\right),$$
donc avec les lemmes \ref{min_BM}, \ref{min_LB}, \ref{min_LA_m>}, et l'hypothèse $\lambda(A+B)=\lambda(A)+(K+\delta+\varepsilon)b$, on obtient
\begin{align*}
(K+\delta+\varepsilon)b & \geqslant b + Km\left(1-\dfrac{g\left(\varepsilon,K\right)}{f\left(\varepsilon,K\right)}-K\varepsilon\left( \dfrac{1}{g\left(\varepsilon,K\right)}-\frac{b}{\mu(I_K)}\right)\right)\dfrac{1-\varepsilon}{m}b-\dfrac{1-\varepsilon}{m}b \\ & \ \ \ +m\left(1-\dfrac{\varepsilon}{1+\varepsilon}\right)\left(\dfrac{\mu(I_{K})}{m}-\dfrac{f\left(\varepsilon,K\right)}{m}b\right).
\end{align*}
Ainsi en utilisant le corollaire \ref{mu(I_K)}, on a 
\begin{align*}
(K+\delta+\varepsilon) & \geqslant 1 + K(1-\varepsilon)\left( 1-\dfrac{g\left(\varepsilon,K\right)}{f\left(\varepsilon,K\right)}-\dfrac{K\varepsilon}{g\left(\varepsilon,K\right)}+\dfrac{K\varepsilon b}{\mu(I_{K})}\right)-\dfrac{1-\varepsilon}{m}  \\ & \ \ \ +m\left(1-\dfrac{\varepsilon}{1+\varepsilon}\right)\left(\dfrac{\delta-K\left( \log(K)-1\right)\varepsilon}{m}-\dfrac{f\left(\varepsilon,K\right)}{m}\right)\\ & \geqslant 1 + K(1-\varepsilon)\left( 1-\dfrac{g\left(\varepsilon,K\right)}{f\left(\varepsilon,K\right)}-\dfrac{K\varepsilon}{g\left(\varepsilon,K\right)}+\dfrac{K\varepsilon}{\delta+2\varepsilon}\right)-\dfrac{1-\varepsilon}{m}  \\ & \ \ \ +\left(1-\dfrac{\varepsilon}{1+\varepsilon}\right)\big(\delta-K\left( \log(K)-1\right)\varepsilon-f\left(\varepsilon,K\right)\big).
\end{align*}
Finalement, on a donc
\begin{align*}
0 & \geqslant  1 -K(1-\varepsilon)\left( \dfrac{g\left(\varepsilon,K\right)}{f\left(\varepsilon,K\right)}+\dfrac{K\varepsilon}{g\left(\varepsilon,K\right)}-\dfrac{K\varepsilon}{\delta+2\varepsilon}\right)-\dfrac{1-\varepsilon}{m}  \\ & \ \ \ -\left(1-\dfrac{\varepsilon}{1+\varepsilon}\right)\big(K\left( \log(K)-1\right)\varepsilon+f\left(\varepsilon,K\right)\big)-\delta\dfrac{\varepsilon}{1+\varepsilon}-\varepsilon \\ & \geqslant  1 - K(1-\varepsilon)\left( \dfrac{g\left(\varepsilon,K\right)}{f\left(\varepsilon,K\right)}+\dfrac{K\varepsilon}{g\left(\varepsilon,K\right)}-\dfrac{K\varepsilon}{\delta+2\varepsilon}\right)-\dfrac{1-\varepsilon}{m}  \\ & \ \ \ -\big(K\left( \log(K)-1\right)\varepsilon+f\left(\varepsilon,K\big)\right)-\delta\varepsilon-\varepsilon .
\end{align*}
Ainsi 
$$K(1-\varepsilon)\left( \dfrac{g\left(\varepsilon,K\right)}{f\left(\varepsilon,K\right)}+\dfrac{K\varepsilon}{g\left(\varepsilon,K\right)}-\dfrac{K\varepsilon}{\delta+2\varepsilon}\right)+K\left( \log(K)-1\right)\varepsilon+f\left(\varepsilon,K\right)+\delta\varepsilon+\varepsilon \geqslant  1-\frac{1-\varepsilon}{m},$$
et comme on a supposé $m\geqslant \dfrac{\mu(I_K)}{\varepsilon^2_K b}$, on a
$$1-\frac{1-\varepsilon}{m}\geqslant 1-(1-\varepsilon)\frac{\varepsilon^2_K b}{\mu(I_K)}\geqslant 1-(1-\varepsilon)\frac{K\varepsilon}{\delta+2\varepsilon},$$
où la dernière inégalité provient de \eqref{maj_eps_2} et du corollaire \ref{mu(I_K)}. Ainsi
$$K(1-\varepsilon)\left( \dfrac{g\left(\varepsilon,K\right)}{f\left(\varepsilon,K\right)}+\dfrac{K\varepsilon}{g\left(\varepsilon,K\right)}\right)-(K-1)(1-\varepsilon)\dfrac{K\varepsilon}{\delta+2\varepsilon}+K\left( \log(K)-1\right)\varepsilon+f\left(\varepsilon,K\right)+\delta\varepsilon+\varepsilon \geqslant  1.$$
Pour aboutir à une absurdité, on voudrait que le terme de gauche soit le plus petit possible, ce qui nous conduit à choisir $f\left(\varepsilon,K\right)=K\varepsilon^{1/3}$ et $g\left(\varepsilon,K\right)=K\varepsilon^{2/3}$.
\begin{Remarque}
On peut faire ce choix car on a bien pour tout $K\geqslant 2$ et tout $\varepsilon > 0$ 
\begin{itemize}
\item $g\left(\varepsilon,K\right) <f\left(\varepsilon,K\right)$. 
\item $\mu(I_K)-f\left(\varepsilon,K\right)b>0 $, cf. remarque \ref{pasdepbpourlechoixdef}.
\end{itemize}
\end{Remarque} 
De cette manière on a 
$$3K\varepsilon^{1/3}(1-\varepsilon)-(K-1)(1-\varepsilon)\dfrac{K\varepsilon}{\delta+2\varepsilon}+K\left( \log(K)-1\right)\varepsilon+\delta\varepsilon+\varepsilon \geqslant  1,$$
or ceci est absurde par la remarque \ref{K>1} et les hypothèses \eqref{hyp_1} et \eqref{hyp_2}.

Ainsi $m=1$, et donc $B$ est inclus dans un intervalle de $\mathbb{T}$. Comme il contient $0$, nous connaissons la structure de $B$ dans $\mathbb{R}$ :
$$ B=B_0\sqcup B_1,$$
où $B_0\subseteq \left[0,b_+\right]$, $B_1\subseteq \left[1-b_-,1\right]$, $\left\lbrace b_+,1-b_- \right\rbrace\subset B$ et $b_+ +b_-\leqslant b+\min\limits_{\substack{i=1,...,K}}\varepsilon^2_ib$.
$$ $$
Voici un exemple de ce à quoi peut ressembler $B$ : \\
\\
$$\begin{minipage}[l]{16cm}
\includegraphics[height=1.47cm]{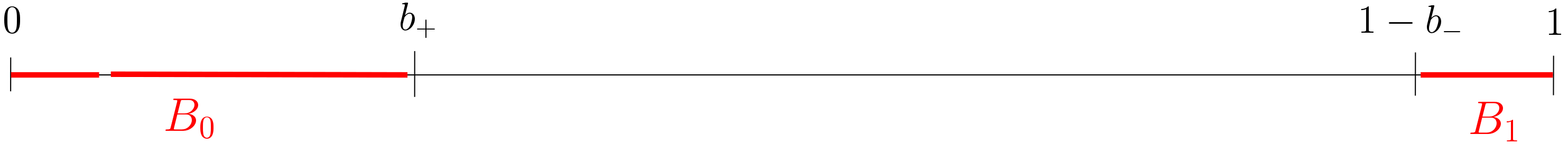}
\end{minipage}$$
$$ $$
D'après \eqref{maj_eps_2}, en particulier on a $\min\limits_{\substack{i=1,...,K}}\varepsilon^2_ib\leqslant\varepsilon b$ et donc
\begin{equation}\label{maj_b++b-}
b_+ +b_-\leqslant b+\varepsilon b
\end{equation}
\begin{equation}\label{maj_b+}
b_+\leqslant \lambda(B_0)+\varepsilon b
\end{equation}
et
\begin{equation}\label{maj_b-}
b_-\leqslant \lambda(B_1)+\varepsilon b.
\end{equation}

\subsection{Structure principale de A}

Comme conclu précédemment, on a $ B=B_0\sqcup B_1,$ où $B_0\subseteq \left[0,b_+\right]$, $B_1\subseteq \left[1-b_-,1\right]$, $b_+,1-b_-\in B$ et $b_+ +b_-\leqslant b+\min\limits_{\substack{i=1,...,K}}\varepsilon^2_ib$.
De même, $\tilde{A}_K$ est inclus dans un intervalle de $\mathbb{T}$, qui lui, ne contient pas $0$ (cf. \eqref{I_k-loin-de-0}). Ainsi $ \tilde{A}_K\subseteq\left[i_K^-,i_K^+\right],$ où $\left\lbrace i^-_K,i^+_K \right\rbrace\subset\tilde{A}_K $, et $i_K^+ -i_K^- =\mu(I_K)\leqslant\mu\left(\tilde{A}_K\right)+\varepsilon^2_Kb$. On rappelle que $\mathscr{L}_A=\left\lbrace l\geqslant 0 \ \vert \ A_l\neq\varnothing \right\rbrace$, que pour tout $x\in \tilde{A}_k$, $\mathscr{L}_x = \left\lbrace n\in\mathbb{N} \ \vert \ n+x\in A \right\rbrace$ et que pour tout $k\in\left\lbrace 1,...,K \right\rbrace$, on a défini $\Omega\left(\tilde{A}_k\right)$ par
$$\Omega\left(\tilde{A}_k\right)=\left\lbrace i\in\mathscr{L}_A \ \vert \ \left( A_i \mod 1 \right)\cap \tilde{A}_k \neq \varnothing \right\rbrace .$$
Nous utiliserons fréquemment la remarque suivante. Pour tout ensemble $E\subset\mathbb{R}$ mesurable et borné, on a 
\begin{equation}\label{lemme+(0,1)}
\lambda\big( E+\left\lbrace 0,1 \right\rbrace\big)\geqslant\lambda(E)+\mu\big(\pi(E)\big).
\end{equation}

Nous allons commencer par déterminer l'emplacement des éléments de $A$ se projetant modulo $1$ sur $\tilde{A}_K$. Par définition, chaque élément de $\tilde{A}_K$ s'exprime dans $K$ étages différents de $A$. Ici, on appelle "étage", tout segment entre deux entiers consécutifs. Il s'agit de déterminer ces étages en premier lieu.

Dans un second temps, nous résoudrons la même question pour tous les ensembles $\tilde{A}_k$ jusque $k=1$ afin d'obtenir la structure principale de $A$. 

\subsubsection{ Étape 0 : contribution de $\tilde{A}_K$ dans $A$}

\begin{lemma}\label{omega(tilde(AK)}
Il existe $a\in\mathbb{N}$ tel que $\Omega\left(\tilde{A}_K\right)=\left\llbracket a,a+K-1 \right\rrbracket.$
\end{lemma}
\begin{proof}
On a $ \tilde{A}_K\subseteq\left[i_K^-,i_K^+\right],$
où $i_K^+ -i_K^- =\mu(I_K)\leqslant\mu\left(\tilde{A}_K\right)+\varepsilon^2_Kb$. Donc pour tout $x\in\tilde{A}_K$, il existe $y\in\tilde{A}_K$ tel que $y\neq x$ et $\left| x-y \right|\leqslant\varepsilon^2_Kb$. Soit un tel $y$ et supposons sans perdre en généralité que $x<y$. Soit $l\in\left( \mathscr{L}_x\cup \mathscr{L}_y\right)$, on a d'un côté 
$$\mu\left(\left(l+\left[y,x+b_+\right]\right)\cap S_l\right)\geqslant -(y-x)+\mu(B_0), $$
où on rappelle que $S_l=(A+B)\cap\frac{l+\big( J_+\cup (J_--1)\big)+I_1'}{m}$. De l'autre côté,
$$\mu\left(\left(1+l+\left[y-b_-,x\right]\right)\cap S_{l+1}\right)\geqslant -(y-x)+\mu(B_1). $$
Notons que le fait que les segments puissent être vides ou que les membres de droite puissent être négatifs ne pose pas de problème pour la suite de l'argumentation. Ainsi
$$\mu\left(\tilde{S}_{\#\left( \mathscr{L}_x\cup \mathscr{L}_y\right)}\right)\geqslant -(y-x)+\mu(B_0)-(y-x)+\mu(B_1)\geqslant b-2(y-x)\geqslant b-2\varepsilon^2_Kb.$$
Or d'après \eqref{presque_égalités} et le lemme \ref{mu(A_K)}, on a 
$$ \mu\left(\tilde{S}_{K+1}\right)= \mu\left(\tilde{A}_K\right)+\varepsilon^1_K b =\delta b +\dfrac{1}{K}\sum\limits_{\substack{k=1}}^{K-1} k\left(\varepsilon^1_{k+1} -\varepsilon^2_{k+1}\right) b + \varepsilon^1_K b.$$
Ainsi 
$$
\mu\left(\tilde{S}_{\#\left( \mathscr{L}_x\cup \mathscr{L}_y\right)}\right)-\mu\left(\tilde{S}_{K+1}\right)\geqslant \Big(1-\delta-\frac{K+1}{K}\varepsilon^2_K-\varepsilon^1_K- \dfrac{1}{K}\sum\limits_{\substack{k=1}}^{K-1}k\varepsilon^1_{k+1}\Big)b>0,
$$
par \eqref{eps=somme}, \eqref{maj_eps_1}, \eqref{maj_eps_2} et l'hypothèse \eqref{hyp_2}. Ainsi $\#\left( \mathscr{L}_x\cup \mathscr{L}_y\right)\leqslant K$. Mais $x,y\in\tilde{A}_K$ donc $\#\mathscr{L}_x\geqslant K$ et $\#\mathscr{L}_y\geqslant K$. Donc
$$\mathscr{L}_x=\mathscr{L}_y=:\mathscr{L},$$
et de proche en proche, on obtient $\mathscr{L}_x=\mathscr{L}$ quel que soit $x\in\tilde{A}_K$. Comme $0,1\in B$, $l+\tilde{A}_K\subseteq S_l$ et $l+1+\tilde{A}_K\subseteq S_{l+1}$ quel que soit $l\in\mathscr{L}$, alors 
$$\mu\left(\tilde{S}_{\#\left(\mathscr{L}\cup\left(\mathscr{L}+1\right)\right)}\right)\geqslant\mu\left(\tilde{A}_K\right) .$$
D'autre part $\mu\left(\tilde{S}_{K+2}\right)=\varepsilon^3_{K+2}b<\mu\left(\tilde{A}_K\right),$ d'où $\#\left(\mathscr{L}\cup\left(\mathscr{L}+1\right)\right)\leqslant K+1=\#\mathscr{L}+1$, ce qui implique que $\mathscr{L}$ est composé d'entiers consécutifs. Et donc finalement, il existe $a\in\mathbb{N}$ tel que $\Omega\left(\tilde{A}_K\right)=\left\llbracket a,a+K-1 \right\rrbracket.$
\end{proof}
Nous venons de montrer que $\left\llbracket a,a+K-1 \right\rrbracket+\tilde{A}_K\subset A$.

Afin de traduire sur $A$ et $S$ les informations dont nous disposons sur $\left\lbrace \tilde{A}_k\right\rbrace_{k=1}^K$ et $\left\lbrace\tilde{S}_k\right\rbrace_{k=1}^{K_S}$, nous définissons les "projetés inverses" $\pi^{-1}_A$ et $\pi^{-1}_S$ comme suit : Pour tout sous-ensemble $E$ de $\mathbb{T}$ on pose
$$\pi^{-1}_A\left( E\right) =\left\lbrace x\in A \ \vert \ x\mod 1\in E \right\rbrace , $$
et $\pi^{-1}_S\left( E\right) =\left\lbrace x\in S \ \vert \ x\mod 1\in E \right\rbrace$. Avec ces définitions, on a directement 
\begin{equation}\label{pi^-1(tilde(A_K))}
\lambda\left(\pi^{-1}_A\left( \tilde{A}_K\right)\right)=K\mu\left(\tilde{A}_K\right),
\end{equation}
et avec le lemme \ref{omega(tilde(AK)} 
\begin{equation}\label{pi^-1(tilde(A_K)+(0,1))}
\lambda\left(\pi^{-1}_A\left( \tilde{A}_K\right)+\left\lbrace 0,1 \right\rbrace\right)=(K+1)\mu\left(\tilde{A}_K\right).
\end{equation}
Ces égalités sont claires car $K=K_A$ mais pour tout entier $k\in\left\lbrace 1,...,K-1 \right\rbrace$, l'égalité devient
\begin{equation}\label{pi^-1(tilde(A)}
\lambda\left(\pi^{-1}_A\left( \tilde{A}_k\right)\right)=k\mu\left(\tilde{A}_k\right)+\sum\limits_{\substack{l=k+1}}^{K}\mu\left( \tilde{A}_l \right).
\end{equation}
En effet $\pi^{-1}_A\left( \tilde{A}_k\right)$ est l'ensemble des éléments de $A$ qui se répètent au moins $k$ fois modulo $1$, c'est donc la réunion disjointe, pour $l$ variant de $k$ à $K_A$, des ensembles des éléments de $A$ qui se répètent exactement $l$ fois modulo $1$. La mesure de l'ensemble des éléments de $A$ qui se répètent exactement $l$ fois modulo $1$ est égale à $l\mu\big(\tilde{A}_l\setminus\tilde{A}_{l+1}\big)$, et ainsi on a
$$\lambda\left(\pi^{-1}_A\left( \tilde{A}_k\right)\right)  = \sum\limits_{\substack{l=k}}^{K_A}l\mu\left( \tilde{A}_l\setminus \tilde{A}_{l+1} \right) = \sum\limits_{\substack{l=k}}^{K_A}l\mu\left( \tilde{A}_l \right)-\sum\limits_{\substack{l=k}}^{K_A}l\mu\left( \tilde{A}_{l+1} \right)  =k\mu(\tilde{A}_k)+ \sum\limits_{\substack{l=k+1}}^{K_A}\mu\left( \tilde{A}_l \right), $$
ce qui prouve la formule car $K_A=K$. 

De même, pour tout entier $k\geqslant 1$, on a la formule
\begin{equation}\label{pi^-1(tilde(S)}
\lambda\left(\pi^{-1}_S\left( \tilde{S}_k\right)\right)=k\mu\left(\tilde{S}_k\right)+\sum\limits_{\substack{i=k+1}}^{K+1}\mu\left(\tilde{S}_{i}\right)+\sum\limits_{\substack{i\geqslant K+2}}\varepsilon^3_i b.
\end{equation}
\begin{proof}
Il suffit de suivre le même procédé puis d'utiliser $\eqref{maj_eps_3}$
$$
\lambda\left(\pi^{-1}_S\left( \tilde{S}_k\right)\right)= k\mu\left(\tilde{S}_k\right)+\sum\limits_{\substack{i=k+1}}^{K_S}\mu\left(\tilde{S}_{i}\right) =k\mu\left(\tilde{S}_k\right)+\sum\limits_{\substack{i=k+1}}^{K+1}\mu\left(\tilde{S}_{i}\right)+\sum\limits_{\substack{i\geqslant K+2}}\varepsilon^3_i b.$$
\end{proof}
Nous pouvons écrire $\pi^{-1}_A\left(\tilde{A}_k\right)$ sous différentes formes
$$ \pi^{-1}_A\left(\tilde{A}_k\right) =\left\lbrace x\in A \ \vert \ x\mod 1\in\tilde{A}_k  \right\rbrace = \left\lbrace x\in A \ \vert \ \#\big(\left\lbrace x+\mathbb{N} \right\rbrace\cap A\big)\geqslant k \right\rbrace =\left\lbrace x+\mathscr{L}_x \ \vert \ x\in \tilde{A}_k \right\rbrace ,$$
et donc en particulier on a $\pi^{-1}_A\left(\tilde{A}_K\right)=\left\llbracket a,a+K-1 \right\rrbracket+\tilde{A}_K$.

\subsubsection{Stratégie et Étape 1 : $k=K$}

Écrivons $A=\bigsqcup\limits_{\substack{k=1}}^K \dot{A_k},$ où $\dot{A_K}=\pi^{-1}_A\left( \tilde{A}_{K}\right)=\left\llbracket a,a+K-1 \right\rrbracket+\tilde{A}_K $ et quel que soit $k\in\left\lbrace 1,...,K-1 \right\rbrace$
$$ \dot{A_k}=\pi^{-1}_A\left( \tilde{A}_{k}\setminus \tilde{A}_{k+1}\right)=\left\lbrace x\in A \ \vert \ \#\mathscr{L}_x=k \right\rbrace  = \pi^{-1}_A\left( \tilde{A}_{k}\right)\setminus \pi^{-1}_A\left( \tilde{A}_{k+1}\right).$$
Nous allons comparer $A$ à l'ensemble $A'=\bigsqcup\limits_{\substack{k=1}}^K A'_k,$ où $A'_K=\left\llbracket a,a+K-1 \right\rrbracket+ I_K,$ et pour tout $k\in\left\lbrace 1,...,K-1 \right\rbrace$
$$A'_k=\Big(\left\llbracket a,a+k-1 \right\rrbracket+ I_K +\left( K-k\right)\big( \left[0,b_+\right]\sqcup\left[1-b_-,1\right]\big)\Big)\setminus A'_{k+1}.$$
Voici une représentation de $A'$ pour $K=5$ :
$$\begin{minipage}[l]{17cm}
\includegraphics[height=9cm]{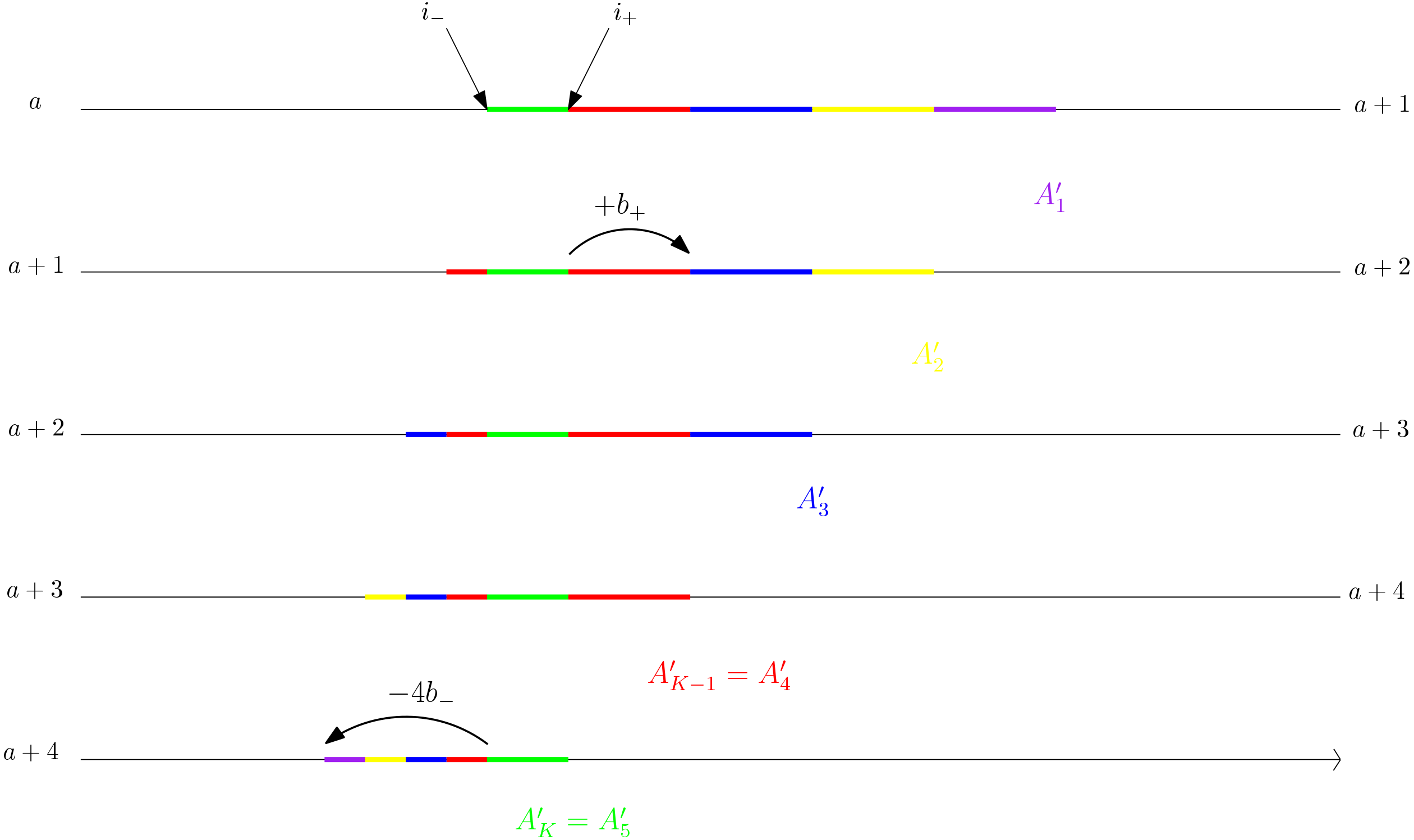}
\end{minipage}$$
$A'_K$ est $I_K$ au $K$ étages de $\Omega\left(\tilde{A}_K\right)$, puis pour tout $k<K$, $A'_{k-1}$ est l'ajout de $\left[0,b_+\right]$ à droite de chacun des $k-1$ premiers étages de $A'_k$ et de $\left[-b_-,0\right]$ à gauche de chacun des $k-1$ derniers étages de $A'_k$. Il faut le voir comme un ensemble "idéal" (sans petit trou) proche de l'ensemble $A_0$ des théorèmes \ref{thm_Anne} et \ref{main_result}.

Pour comparer $A$ et $A'$, nous montrerons que pour tout $k\in\left\lbrace 1,...,K \right\rbrace$, $A'_k$ et $\dot{A_k}$ sont très proches. Tout d'abord, on a $\tilde{A}_K\subseteq I_K$ et $\Omega(\tilde{A}_K)=\left\llbracket a,a+K-1 \right\rrbracket$ donc $\dot{A_K}\subseteq A'_K$ et donc
\begin{equation}\label{etape_K}
\lambda\left(A'_K\cap \dot{A_K}\right)=\lambda\big( \dot{A_K}\big).
\end{equation} 
De plus $\mu(I_K)\leqslant\mu\left(\tilde{A}_K\right)+\varepsilon^2_K b$ donc $\lambda\left(A'_K\setminus \dot{A_K}\right)\leqslant K\varepsilon^2_K b$. Pour les autres indices, le raisonnement sera similaire bien que plus technique.

\subsubsection{Étapes suivantes.}

Soit $k\in\left\lbrace 2,...,K \right\rbrace$, on a $\dot{A_{k-1}}=\pi^{-1}_A\left( \tilde{A}_{k-1}\setminus \tilde{A}_{k}\right),$ et 
$$ A'_{k-1}=\Big(\left\llbracket a,a+k-2 \right\rrbracket+ I_K +\left( K-k+1\right)\big( \left[0,b_+\right]\sqcup\left[1-b_-,1\right]\big)\Big)\setminus A'_{k}.$$
Pour montrer que ces deux ensembles sont proches, nous commençons par montrer que $\ddot{A_{k-1}}=\dot{A_{k-1}}+\left\lbrace 0,1 \right\rbrace$ est proche de
\begin{align*}
A''_{k-1} & =A'_{k-1}+\left\lbrace 0,1 \right\rbrace = \Big(\left\llbracket a,a+k-1 \right\rrbracket+ I_K +(K-k+1)\big( \left[0,b_+\right]\sqcup\left[1-b_-,1\right]\big)\Big)\setminus \big( A'_{k}+\left\lbrace 0,1 \right\rbrace\big) \\ & = \Big( A'_k+\big( \left[0,b_+\right]\sqcup\left[1-b_-,1\right]\big)\Big)\setminus \big( A'_{k}+\left\lbrace 0,1 \right\rbrace\big) .
\end{align*}
En effet, \eqref{presque_égalités} impliquera que chacun de ces deux ensembles est proche de $\dot{S_{k}}=\pi^{-1}_S\left(\tilde{S}_{k}\setminus\tilde{S}_{k+1}\right)$. Nous allons donc contrôler $\lambda\left(\ddot{A_{k-1}}\cap \dot{S_{k}}\right)$ puis $\lambda\left(A''_{k-1}\cap \dot{S_{k}}\right)$ pour pouvoir contrôler $\lambda\Big(\big( A''_{k-1}\cap\dot{S_{k}}\big)\cap\big( \ddot{A_{k-1}}\cap \dot{S_{k}}\big)\Big)$ et finalement $\lambda\left(A'_{k-1}\cap \dot{A_{k-1}}\right)$. 
Lors des première et troisième étapes, nous utilisons le fait que si $E$ et $F$ sont deux ensembles bornés tous deux inclus dans $H$ (borné également) alors 
\begin{equation}\label{EinterF}
\lambda\big( E\cap F \big)=\lambda(E)+\lambda(F)-\lambda\big( E\cup F\big)\geqslant \lambda(E)+\lambda(F)-\lambda(H).
\end{equation}
\begin{enumerate}
\item \textbf{Minoration de} $\boldsymbol{\lambda\left(\ddot{A_{k-1}}\cap \dot{S_{k}}\right) :}$ \\
\\
Comme $0,1\in B$, $\pi^{-1}_A\big(\tilde{A}_k\big)+\left\lbrace 0,1 \right\rbrace\subseteq\pi^{-1}_S\big(\tilde{S}_{k+1}\big)$ donc d'une part
$$\dot{S_{k}}=\pi^{-1}_S\big(\tilde{S}_k\big)\setminus \pi^{-1}_S\big(\tilde{S}_{k+1}\big)\subseteq \pi^{-1}_S\big(\tilde{S}_k\big)\setminus\left(\pi^{-1}_A\big(\tilde{A}_k\big)+\left\lbrace 0,1 \right\rbrace\right) ,$$ 
et d'autre part $\ddot{A_{k-1}}=\dot{A_{k-1}}+\left\lbrace 0,1 \right\rbrace\subseteq \pi^{-1}_S\big(\tilde{S}_k\big)$. De plus par définition 
$$\dot{A_{k-1}}= \pi^{-1}_A\big(\tilde{A}_{k-1}\big)\setminus \pi^{-1}_A\big(\tilde{A}_k\big) =\pi^{-1}_A\big(\tilde{A}_{k-1}\setminus\tilde{A}_k\big),$$
donc $\left(\dot{A_{k-1}}+\mathbb{N}\right)\cap\left(\tilde{A}_k+\mathbb{N}\right)=\varnothing $, ainsi $\ddot{A_{k-1}}\cap \left(\pi^{-1}_A\big(\tilde{A}_k\big)+\left\lbrace 0,1 \right\rbrace\right)=\varnothing $ (car $\pi^{-1}_A\big(\tilde{A}_k\big)+\left\lbrace 0,1 \right\rbrace\subset \tilde{A}_k+\mathbb{N}$ ), d'où
$$\ddot{A_{k-1}}\subseteq \pi^{-1}_S\big(\tilde{S}_k\big)\setminus\left(\pi^{-1}_A\big(\tilde{A}_k\big)+\left\lbrace 0,1 \right\rbrace\right).$$
On peut à présent utiliser \eqref{EinterF} avec $H=\pi^{-1}_S\big(\tilde{S}_k\big)\setminus\left(\pi^{-1}_A\big(\tilde{A}_k\big)+\left\lbrace 0,1 \right\rbrace\right)$, $E=\ddot{A_{k-1}}$ et $F=\dot{S_{k}}$. On obtient
$$\lambda\left(\ddot{A_{k-1}}\cap \dot{S_{k}}\right) \geqslant \lambda\left(\ddot{A_{k-1}}\right) + \lambda\left(\dot{S_{k}}\right) - \lambda\Big(\pi^{-1}_S\big(\tilde{S}_k\big)\setminus\left(\pi^{-1}_A\big(\tilde{A}_k\big)+\left\lbrace 0,1 \right\rbrace\right)\Big) .$$
ainsi
$$\lambda\left(\ddot{A_{k-1}}\cap \dot{S_{k}}\right) \geqslant \lambda\left(\ddot{A_{k-1}}\right) + \lambda\left(\dot{S_{k}}\right) -\lambda\Big(\pi^{-1}_S\big(\tilde{S}_k\big)\Big)+\lambda\left(\pi^{-1}_A\big(\tilde{A}_k\big)+\left\lbrace 0,1 \right\rbrace\right) , $$
et par \eqref{lemme+(0,1)}
$$\lambda\left( \pi^{-1}_A\big(\tilde{A}_k\big)+\left\lbrace 0,1 \right\rbrace\right)\geqslant\lambda\left(\pi^{-1}_A\big(\tilde{A}_k\big)\right)+\mu\left(\tilde{A}_k\right) ,$$
donc
$$\lambda\left(\ddot{A_{k-1}}\cap \dot{S_{k}}\right) \geqslant \lambda\left(\ddot{A_{k-1}}\right) + \lambda\left(\dot{S_{k}}\right) -\Bigg(\lambda\Big(\pi^{-1}_S\big(\tilde{S}_k\big)\Big)-\lambda\Big(\pi^{-1}_A\big(\tilde{A}_k\big)\Big)\Bigg)+\mu\left(\tilde{A}_k\right).$$
De plus \eqref{pi^-1(tilde(A)} et \eqref{pi^-1(tilde(S)} impliquent
\begin{align*}
& \lambda\Big(\pi^{-1}_S\big(\tilde{S}_k\big)\Big)-\lambda\left(\pi^{-1}_A\big(\tilde{A}_k\big)\right)\\ & = k\mu\left(\tilde{S}_k\right)+\sum\limits_{\substack{i=k+1}}^{K+1}\mu\left(\tilde{S}_{i}\right)+\sum\limits_{\substack{i\geqslant K+2}}\varepsilon^3_i b -k\mu\left(\tilde{A}_k\right)-\sum\limits_{\substack{i=k+1}}^{K}\mu\left(\tilde{A}_{i}\right) \\ & = \sum\limits_{\substack{i=k}}^{K}\left(\mu\left(\tilde{S}_{i+1}\right)-\mu\left(\tilde{A}_{i}\right)\right) +k\mu\left(\tilde{S}_k\right)+\sum\limits_{\substack{i\geqslant K+2}}\varepsilon^3_i b -(k-1)\mu\left(\tilde{A}_k\right) \\ & = \sum\limits_{\substack{i=k}}^{K}\varepsilon^1_{i+1}b +k\mu\left(\tilde{S}_k\right)+\sum\limits_{\substack{i\geqslant K+2}}\varepsilon^3_i b -(k-1)\mu\left(\tilde{A}_k\right),
\end{align*}
d'où
$$\lambda\left(\ddot{A_{k-1}}\cap \dot{S_{k}}\right) \geqslant \lambda\left(\ddot{A_{k-1}}\right) + \lambda\left(\dot{S_{k}}\right)+k\mu\left(\tilde{A}_k\right)-k\mu\left(\tilde{S}_k\right)-\sum\limits_{\substack{i=k}}^{K}\varepsilon^1_{i+1}b-\sum\limits_{\substack{i\geqslant K+2}}\varepsilon^3_i b $$
et finalement
\begin{equation}\label{etape_k-1_1}
\lambda\left(\ddot{A_{k-1}}\cap \dot{S_{k}}\right) \geqslant k\mu\left(\tilde{A}_{k-1}\right) + \lambda\left(\dot{S_{k}}\right)-k\mu\left(\tilde{S}_k\right)-\sum\limits_{\substack{i=k}}^{K}\varepsilon^1_{i+1}b-\sum\limits_{\substack{i\geqslant K+2}}\varepsilon^3_i b ,
\end{equation}
puisque $\lambda\left(\ddot{A_{k-1}}\right)=k\left(\mu\left(\tilde{A}_{k-1}\right)-\mu\left( \tilde{A}_k\right)\right)$. \\
\\
\item \textbf{Minoration de} $\boldsymbol{\lambda\left(A''_{k-1}\cap \dot{S_{k}}\right) :}$ \\
\\
$\dot{S_{k}}=\pi^{-1}_S\left(\tilde{S}_{k}\setminus\tilde{S}_{k+1}\right)=\pi^{-1}_S\left(\tilde{S}_{k}\right)\setminus\pi^{-1}_S\left(\tilde{S}_{k+1}\right)$ donc
$$
\lambda\left( A''_{k-1}\cap \dot{S_{k}}\right) = \lambda\left( A''_{k-1}\cap \pi^{-1}_S\left(\tilde{S}_{k}\right)\right) - \lambda\left( A''_{k-1}\cap \pi^{-1}_S\left(\tilde{S}_{k+1}\right)\right) .$$
Or $\pi^{-1}_A\left(\tilde{A}_k\right)+\left\lbrace 0,1 \right\rbrace \subseteq \pi^{-1}_S\left(\tilde{S}_{k+1}\right)$ car $0,1\in B$ et $\left(\pi^{-1}_A\left(\tilde{A}_k\right)+\left\lbrace 0,1 \right\rbrace\right)\cap A''_{k-1}=\varnothing$, donc
\begin{align*}
\lambda\left( A''_{k-1}\cap \pi^{-1}_S\left(\tilde{S}_{k+1}\right)\right) & \leqslant \lambda\big( \pi^{-1}_S\left(\tilde{S}_{k+1}\right)\big)-\lambda\left( \pi^{-1}_A\left(\tilde{A}_k\right)+\left\lbrace 0,1 \right\rbrace\right)  \\ & \leqslant \lambda\left(\pi^{-1}_S\left(\tilde{S}_{k+1}\right)\right) -\lambda\left(\pi^{-1}_A\big(\tilde{A}_k\big)\right) -\mu\left(\tilde{A}_k\right),
\end{align*}
par \eqref{lemme+(0,1)}. De plus par \eqref{pi^-1(tilde(A)} et \eqref{pi^-1(tilde(S)}, on a
\begin{align*}
\lambda\left(\pi^{-1}_S\left(\tilde{S}_{k+1}\right)\right)-\lambda\left(\pi^{-1}_A\big(\tilde{A}_k\big)\right)= & (k+1)\mu\left(\tilde{S}_{k+1}\right)+\sum\limits_{\substack{i=k+2}}^{K+1}\mu\left(\tilde{S}_{i}\right)+\sum\limits_{\substack{i\geqslant K+2}}\varepsilon^3_i b  \\ & \ - \Big(  k\mu\left(\tilde{A}_k\right)+\sum\limits_{\substack{i=k+1}}^{K}\mu\left(\tilde{A}_{i}\right)\Big) \\ & \leqslant \mu\left(\tilde{A}_k\right)+(k+1)\varepsilon^1_{k+1}b+\sum\limits_{\substack{i=k+2}}^{K+1}\varepsilon^1_i b+\sum\limits_{\substack{i\geqslant K+2}}\varepsilon^3_i b.
\end{align*}
D'autre part, comme $i^-_K-(K-k-1)b_-,i^+_K+(K-k-1)b_+ \in\tilde{A}_k $ puisque $i^-_K,i^+_K\in\tilde{A}_K $, et $b_+,1-b_-\in B$ (donc les bords sont atteints à chaque étape), on a
$$\lambda\left( A''_{k-1}\cap \pi^{-1}_S\left(\tilde{S}_{k}\right)\right)\geqslant k b  ,$$
donc finalement
\begin{equation}\label{etape_k-1_2}
\lambda\left( A''_{k-1}\cap \dot{S_{k}}\right)\geqslant k b -(k+1)\varepsilon^1_{k+1}b-\sum\limits_{\substack{i=k+2}}^{K+1}\varepsilon^1_i b-\sum\limits_{\substack{i\geqslant K+2}}\varepsilon^3_i b .
\end{equation}
\item \textbf{Minoration de} $\boldsymbol{\lambda\left(\big( A''_{k-1}\cap\dot{S_{k}}\big)\cap\big( \ddot{A_{k-1}}\cap \dot{S_{k}}\big)\right) :}$ \\
\\
D'après \eqref{EinterF}, puis \eqref{etape_k-1_1} et \eqref{etape_k-1_2}, on a 
\begin{align*}
\lambda\Big( \big( A''_{k-1}\cap\dot{S_{k}}\big)\cap\big( \ddot{A_{k-1}}\cap \dot{S_{k}}\big)\Big) & \geqslant \lambda\big( A''_{k-1}\cap\dot{S_{k}}\big)+\lambda\big( \ddot{A_{k-1}}\cap \dot{S_{k}}\big) -\lambda\big(\dot{S_{k}}\big) \\ & \geqslant k b -(k+1)\varepsilon^1_{k+1}b-\sum\limits_{\substack{i=k+2}}^{K+1}\varepsilon^1_i b-\sum\limits_{\substack{i\geqslant K+2}}\varepsilon^3_i b+k\mu\left(\tilde{A}_{k-1}\right) \\ & \ \ \  + \lambda\left(\dot{S_{k}}\right)-k\mu\left(\tilde{S}_k\right)-\sum\limits_{\substack{i=k}}^{K}\varepsilon^1_{i+1}b-\sum\limits_{\substack{i\geqslant K+2}}\varepsilon^3_i b-\lambda\big(\dot{S_{k}}\big) \\ & \geqslant k b -(k+2)\varepsilon^1_{k+1}b-2\sum\limits_{\substack{i=k+2}}^{K+1}\varepsilon^1_i b-2\sum\limits_{\substack{i\geqslant K+2}}\varepsilon^3_i b \\ & \ \ \  +k\mu\left(\tilde{A}_{k-1}\right)-k\mu\left(\tilde{S}_k\right) ,
\end{align*}
d'où, comme $\mu(\tilde{A}_{k-1})\leqslant\mu(\tilde{S}_k)$ pour tout $k\geqslant 2$, on obtient
\begin{equation}\label{etape_k-1_3}
\lambda\Big( \big( A''_{k-1}\cap\dot{S_{k}}\big)\cap\big( \ddot{A_{k-1}}\cap \dot{S_{k}}\big)\Big)\geqslant k b-k\varepsilon^1_k b -(k+2)\varepsilon^1_{k+1}b-2\sum\limits_{\substack{i=k+2}}^{K+1}\varepsilon^1_i b-2\sum\limits_{\substack{i\geqslant K+2}}\varepsilon^3_i b .
\end{equation}
\item \textbf{Minoration de} $\boldsymbol{\lambda\left(A'_{k-1}\cap \dot{A_{k-1}}\right) :}$ \\
\\
Soit $x\in \tilde{A}_{k-1}\setminus\tilde{S}_{k+1}$. Comme $x+\mathscr{L}_x\subset A$ et $0,1\in B$, on a $x+\mathscr{L}_x+\left\lbrace 0,1 \right\rbrace\subset S$. De plus, $\#\mathscr{L}_x = k-1$ donc $\#\left(\mathscr{L}_x+\left\lbrace 0,1 \right\rbrace\right)\geqslant k$. Or $x\notin \tilde{S}_{k+1}$ donc $\#\left(\mathscr{L}_x+\left\lbrace 0,1 \right\rbrace\right)\leqslant k$ et donc finalement
$$\#\left(\mathscr{L}_x+\left\lbrace 0,1 \right\rbrace\right)=k=\# \mathscr{L}_x+1,$$
et donc nécessairement $\mathscr{L}_x$ est composé d'entiers consécutifs. Or $A'_{k-1}$ est également composé d'éléments à des étages consécutifs par construction et $A''_{k-1}=A'_{k-1}+\left\lbrace 0,1 \right\rbrace$, donc si $i>1$
\begin{align*}
& \mu\left( x\in \tilde{A}_{k-1}\setminus\tilde{S}_{k+1} \ \vert \ \#\left\lbrace \big(x+\mathscr{L}_x+\left\lbrace 0,1 \right\rbrace\big)\cap A''_{k-1} \right\rbrace =i\right) \\ & \leqslant \mu \left( x\in \tilde{A}_{k-1}\setminus\tilde{S}_{k+1} \ \vert \ \#\left\lbrace \big(x+\mathscr{L}_x\big)\cap A'_{k-1} \right\rbrace =i-1\right) .
\end{align*}
Finalement
\begin{align*}
\lambda\left(A'_{k-1}\cap \dot{A_{k-1}}\right) & \geqslant \lambda\left(A'_{k-1}\cap \dot{A_{k-1}}\cap\dot{S_{k}} \right)  \\ & \geqslant \lambda\big( A''_{k-1}\cap\dot{S_{k}}\cap \ddot{A_{k-1}}\big)-\mu\Big(\big( A''_{k-1}\cap\dot{S_{k}}\cap \ddot{A_{k-1}}\big)\mod 1\Big) \\ & \geqslant \lambda\big( A''_{k-1}\cap\dot{S_{k}}\cap \ddot{A_{k-1}}\big)-\big(1+\min\limits_{\substack{i=1,...,K}}\varepsilon^2_i\big) b  \\ & \geqslant k b-k\varepsilon^1_k b -(k+2)\varepsilon^1_{k+1}b-2\sum\limits_{\substack{i=k+2}}^{K+1}\varepsilon^1_i b-2\sum\limits_{\substack{i\geqslant K+2}}\varepsilon^3_i b-\big(1+\min\limits_{\substack{i=1,...,K}}\varepsilon^2_i\big)b \\ & \geqslant (k-1)b - \left( k\varepsilon^1_k +(k+2)\varepsilon^1_{k+1}+2\sum\limits_{\substack{i=k+2}}^{K+1}\varepsilon^1_i +2\sum\limits_{\substack{i\geqslant K+2}}\varepsilon^3_i+\min\limits_{\substack{i=1,...,K}}\varepsilon^2_i\right) b.
\end{align*}
\end{enumerate}
Ainsi pour tout $k\in\left\lbrace 1,...,K-1 \right\rbrace$
$$\lambda\left(A'_{k}\cap \dot{A_{k}}\right) \geqslant kb - \left( (k+1)\varepsilon^1_{k+1} +(k+3)\varepsilon^1_{k+2}+2\sum\limits_{\substack{i=k+3}}^{K+1}\varepsilon^1_i +2\sum\limits_{\substack{i\geqslant K+2}}\varepsilon^3_i+\min\limits_{\substack{i=1,...,K}}\varepsilon^2_i\right) b ,$$
et donc
\begin{equation}\label{etape_<K-1}
\lambda\left(A'_{k}\cap \dot{A_{k}}\right) \geqslant \lambda\big( \dot{A_{k}}\big) - \left( k\varepsilon^2_{k+1}+\varepsilon^1_{k+1} +(k+1)\varepsilon^1_{k+2}+2\sum\limits_{\substack{i=k+2}}^{K+1}\varepsilon^1_i +2\sum\limits_{\substack{i\geqslant K+2}}\varepsilon^3_i+\min\limits_{\substack{i=1,...,K}}\varepsilon^2_i\right) b
\end{equation}
car $ \lambda\big( \dot{A_{k}}\big)=k\big( 1+\varepsilon^2_{k+1}-\varepsilon^1_{k+1}\big) b$ par \eqref{presque_égalités}.

\subsubsection{Conclusion : Structure principale de $A$.}

Nous sommes désormais en mesure de donner une localisation précise de la majorité des éléments de $A$. Rappelons que $A=\bigsqcup\limits_{\substack{k=1}}^K \dot{A_k}$ et qu $A'=\bigsqcup\limits_{\substack{k=1}}^K A'_k$. Par \eqref{etape_K} et \eqref{etape_<K-1}, on a 
\begin{align*}
\lambda\big(A\cap A'\big) & \geqslant\sum\limits_{\substack{k=1}}^K\lambda\big(\dot{A_k}\cap A'_k\big)= \lambda\big(\dot{A}_K\big)+\sum\limits_{\substack{k=1}}^{K-1}\lambda\big(\dot{A_k}\cap A'_k\big) \\ & \geqslant \lambda(A)-\sum\limits_{\substack{k=1}}^{K-1} \left( k\varepsilon^2_{k+1}+\varepsilon^1_{k+1} +(k+1)\varepsilon^1_{k+2}+2\sum\limits_{\substack{i=k+2}}^{K+1}\varepsilon^1_i +2\sum\limits_{\substack{i\geqslant K+2}}\varepsilon^3_i+\min\limits_{\substack{i=1,...,K}}\varepsilon^2_i\right) b .
\end{align*}
Or après calculs et en utilisant \eqref{eps=somme} on peut obtenir l'inégalité
$$\lambda\big(A\cap A'\big)  \geqslant  \lambda(A)-3K\varepsilon +\sum\limits_{\substack{k=1}}^{\lfloor K/4\rfloor} (K-4k)\varepsilon^2_{K+1-k} ,$$
et donc
$$\lambda(A)-\lambda\big(A\cap A'\big) \leqslant 3K\varepsilon +\sum\limits_{\substack{k=1}}^{\lfloor K/4\rfloor} (K-4k)\varepsilon^2_{K+1-k} .$$
Enfin par \eqref{maj_eps_2}, on obtient 
\begin{equation}\label{AetA'>}
\lambda\big(A\cap A'\big)\geqslant \lambda(A)-\Big(K^2\log\left( K\right)-K\big(K(1+\log 4)-7\big) \Big)\varepsilon b .
\end{equation}
On a donc montré que $A$ était principalement inclus dans
$$A'= a'+\bigsqcup\limits_{\substack{k=0}}^{K-1} \left[ k-kb_-, k+\mu\left(I_K\right) +(K-k-1)b_+ \right] .$$
Or par le corollaire \ref{mu(I_K)} puis par \eqref{eps=somme}
$$
\mu\left(I_K\right)  \leqslant \delta b +\dfrac{1}{K}\sum\limits_{\substack{k=1}}^{K-1} k\left(\varepsilon^1_{k+1} -\varepsilon^2_{k+1}\right) b + \varepsilon^2_K b \leqslant \delta b +\dfrac{1}{K}\sum\limits_{\substack{k=1}}^{K-1} k\varepsilon^1_{k+1} b + \frac{1}{K}\varepsilon^2_K b \leqslant \left(\delta +\varepsilon\right) b .$$
Ainsi $\mu\left(I_K\right) \leqslant \left(\delta  +\varepsilon\right) b$ et 
$$A' \subseteq a'+\bigsqcup\limits_{\substack{k=0}}^{K-1} \left[ k-kb_-, k+\left(\delta  +\varepsilon\right) b +(K-k-1)b_+ \right] ,$$
où $a'=\min \tilde{A}_K$. Notons que l'hypothèse \eqref{hyp_0} implique que la réunion est bien disjointe. Pour plus de clarté, nous redéfinissons $A'$ en ce nouvel ensemble :
$$A' = a'+\bigsqcup\limits_{\substack{k=0}}^{K-1} \left[ k-kb_-, k+\left(\delta  +\varepsilon\right) b +(K-k-1)b_+ \right] .$$
Voici une représentation de $A'$ pour $K=5$.
$$\begin{minipage}[l]{16cm}
\includegraphics[height=8cm]{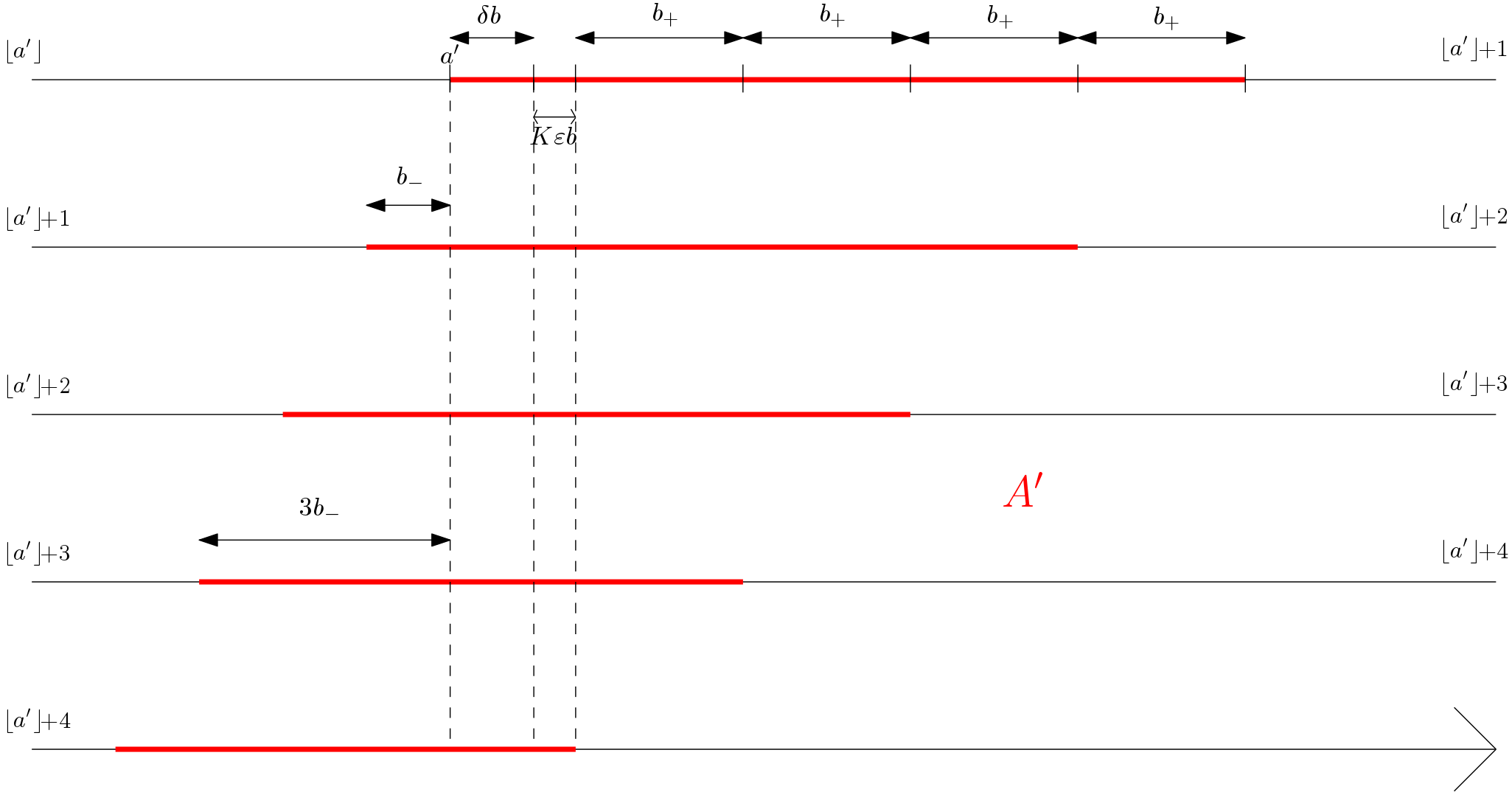}
\end{minipage}$$
$$ $$

\subsection{Structure totale de A}

\subsubsection{Stratégie et premiers résultats}

La partie précédente démontre qu'il existe un réel $a'$ tel que 
$$\lambda\big(A\cap A'\big)\geqslant \lambda(A)-\Big(K^2\log\left( K\right)-K\big(K(1+\log 4)-6\big) \Big)\varepsilon b ,$$
où
$$A'= a'+\bigsqcup\limits_{\substack{k=0}}^{K-1} \left[ k-kb_-, k+\delta b+\varepsilon b +(K-k-1)b_+ \right] .$$
Ainsi l'essentiel de $A$ se concentre dans cette réunion d'intervalles $A'$. Nous allons désormais prouver que tous les éléments de $A$ sont inclus dans un petit voisinage métrique de $A'$. Commençons par établir une minoration de la mesure de $(A\cap A')+B$.
\begin{lemma}\label{AetA'+B>}
On a
$$\lambda\big( (A\cap A')+B\big)\geqslant \lambda(A\cap A')+(K+\delta')b,$$
où $\delta'=\delta-\big(K\log\left( K\right)-K(1+\log 4)+7\big)\varepsilon$.
\end{lemma}
\begin{proof}
D'après \eqref{AetA'>}, on a
$$\lambda\big(A\cap A'\big)\geqslant \lambda(A)-\Big(K^2\log\left( K\right)-K\big(K(1+\log 4)-7\big) \Big)\varepsilon b ,$$
donc
\begin{align*}
\dfrac{\lambda\big(A\cap A'\big)}{b} & \geqslant \dfrac{\lambda(A)}{b}-\dfrac{\Big(K^2\log\left( K\right)-K\big(K(1+\log 4)-7\big) \Big)\varepsilon b}{b} \\ & \geqslant \frac{K(K-1)}{2}+K\delta-\Big(K^2\log\left( K\right)-K\big(K(1+\log 4)-7\big) \Big)\varepsilon    \\ & > \frac{K(K-1)}{2},
\end{align*}
où la dernière ligne est assurée par l'hypothèse \eqref{hyp_1}. D'autre part,
$$\dfrac{\lambda\big(A\cap A'\big)}{b}\leqslant\dfrac{\lambda (A)}{b}=\frac{K(K-1)}{2}+K\delta.$$
Ainsi il existe $\delta'$ tel que 
$$0<\delta-\big(K\log\left( K\right)-K(1+\log 4)+7\big)\varepsilon\leqslant\delta'\leqslant\delta<1,$$
et 
$$ \dfrac{\lambda\big(A\cap A'\big)}{b}=\frac{K(K-1)}{2}+K\delta'.$$
Donc finalement, comme $(K+\delta')b\leqslant (K+\delta)b<1$ par \eqref{hyp_0}, on peut utiliser le théorème de Ruzsa (théorème \ref{thm_Ruzsa}), et on a
$$\lambda\big( (A\cap A')+B\big)\geqslant \lambda(A\cap A')+(K+\delta')b,$$
où $\delta'\geqslant\delta-\big(K\log\left( K\right)-K(1+\log 4)+7\big)\varepsilon$. En redéfinissant naturellement $\delta'$, on obtient le lemme.
\end{proof}
Établissons maintenant un contrôle de la contribution de $(A\setminus A')+B$ à $S$.
\begin{lemma}\label{maj_erreur_AsansA'}
Soit $\tau_{K} =  K(K+1)\log (K)-K\big(K(1+\log 4)-6+\log 4\big)+7$. On a
$$\lambda\Big( \big((A\setminus A')+B\big)\setminus \big( (A\cap A')+B\big)\Big)\leqslant \tau_{K}\varepsilon b.$$
\end{lemma}
\begin{proof}
Par hypothèse, on a $\lambda(A+B)= \lambda(A)+\left( K+\delta+\varepsilon\right)b$. De plus
\begin{align*}
\lambda(A+B) & = \lambda\Big( \big((A\cap A')\sqcup (A\setminus A') \big)+B\Big) \\ & = \lambda\big( (A\cap A')+B\big)+ \lambda\Big( \big((A\setminus A')+B\big)\setminus \big( (A\cap A')+B\big)\Big),
\end{align*}
ainsi, directement par le lemme \ref{AetA'+B>}, on obtient
\begin{align*}
\lambda\Big( \big((A\setminus A')+B\big)\setminus \big( (A\cap A')+B\big)\Big) & \leqslant  \lambda(A)+\left( K+\delta+\varepsilon\right)b-\left(\lambda(A\cap A')+(K+\delta')b\right) \\ & \leqslant \lambda\left( A\right)-\lambda(A\cap A')+(\delta-\delta'+\varepsilon)b \\ & \leqslant  \lambda\left( A\right)- \lambda(A\cap A')+\big(K\log\left( K\right)-K(1+\log 4)+7\big)\varepsilon b ,
\end{align*}
et donc finalement par \eqref{AetA'>}, on a
\begin{align*}
& \lambda\Big( \big((A\setminus A')+B\big)\setminus \big( (A\cap A')+B\big)\Big) \\ & \leqslant \Big(K^2\log\left( K\right)-K\big(K(1+\log 4)-7\big) \Big)\varepsilon b + \big(K\log\left( K\right)-K(1+\log 4)+7\big)\varepsilon b \\ & \leqslant K\big((K+1)\log (K)-K(1+\log 4)+6-\log 4\big)\varepsilon b  +7\varepsilon b.
\end{align*}
\end{proof}
Le lemme \ref{maj_erreur_AsansA'} nous donnera un contrôle de l'erreur commise en remplaçant $A$ par $A\cap A'$. Il permettra en particulier pour $x\in A\setminus A'$ de contrôler la contribution de $x+B$ hors de $A'+B$ et de conclure que $x$ est proche de $A'$. Nous sommes donc parés pour exhiber une région limite autour de $A'$ en dehors de laquelle aucun élément de $A$ ne peut se trouver. Commençons par établir une région limite autour de $\left[\min A',\max A'\right]$.

\subsubsection{Le cas des éléments de $A\setminus \left[\min A',\max A'\right]$}

On rappelle que 
$$ A'= a'+\bigsqcup\limits_{\substack{k=0}}^{K-1} \left[ k-kb_-, k+(\delta+\varepsilon) b +(K-k-1)b_+ \right] ,$$
et on pose pour tout $k\in\left\lbrace 0,...,K-1 \right\rbrace$
$$ A'_k=a'+\left[ k-kb_-, k+(\delta +\varepsilon)b +(K-k-1)b_+ \right],$$
de sorte que $A'=\bigsqcup\limits_{\substack{k=0}}^{K-1} A'_k$
(on peut voir $A'_k$ comme le $k$-ème étage de $A'$). 
$$ $$
$$\begin{minipage}[l]{17cm}
\includegraphics[height=8cm]{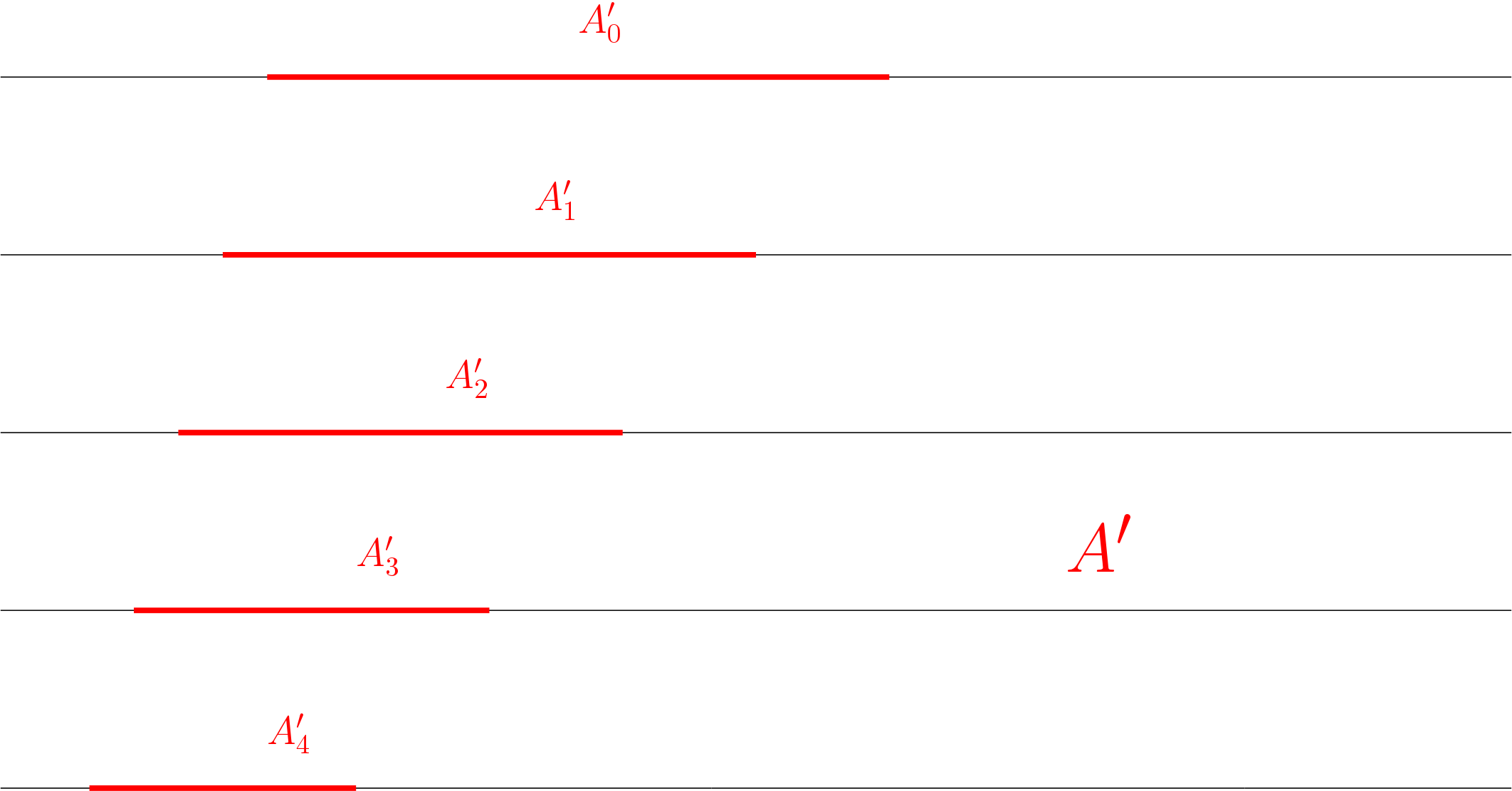}
\end{minipage}$$
$$ $$
Nous allons commencer par montrer que les éléments de $A\setminus A'$ ne peuvent pas être trop inférieurs à $\min A'$. 
\begin{lemma}\label{AsansA'pas_trop_avant}
Soit $x\in A$ tel que $x<\min A'$, on a nécessairement
$$x\geqslant \min A'-(\tau_{K}+1)\varepsilon b ,$$
où on rappelle que $\tau_{K}= (K(K+1)\log (K)-K\big(K(1+\log 4)-6+\log 4\big)+7 $.
\end{lemma}
\begin{proof}
Soit $x\in A$ tel que $x<\min A'$. D'après le lemme \ref{maj_erreur_AsansA'}, on a
\begin{equation}\label{maj_x+BsansA'+B}
\lambda\big( (x+B)\setminus (A'+B)\big)\leqslant\lambda\Big(  \big((A\setminus A')+B\big)\setminus\big((A\cap A')+B\big)\Big)\leqslant \tau_{K} b\varepsilon<b,
\end{equation}
car par l'hypothèse \eqref{hyp_1}, on a $\varepsilon<\left(\frac{\delta}{3K}\right)^3\leqslant\frac{1}{\tau_K}$. 
Ainsi nécessairement $(x+B)\cap (A'+B)\neq\varnothing$. De plus, comme $x<\min A'$ et $B\subseteq\left[ 0,1\right]$, $x+B\subseteq\left[x,\min A'+1\right[$ et donc $x+B$ ne peut intersecter que les deux premiers morceaux de $A'+B$
$$(x+B)\cap (A'+B)=(x+B)\cap\big( (A'_0+B_0)\sqcup (A'_1+B_0)\big),$$
(car on a par construction $A'_0+B_1\subseteq A'_1+B_0$).
Aussi, $x+B=(x+B_0)\sqcup(x+B_1)$ donc 
$$ \big( (x+B_0)\sqcup(x+B_1) \big) \cap \big( (A'_0+B_0)\sqcup (A'_1+B_0)\big)\neq\varnothing,$$
et comme $x<\min A'$, $(x+B_0)\cap(A'_1+B_0)=\varnothing$, et donc on a
$$  \big((x+B_0)\cap(A'_0+B_0)\big)\sqcup\big((x+B_1)\cap(A'_0+B_0)\big)\sqcup\big((x+B_1)\cap(A'_1+B_0)\big)\neq\varnothing.$$
Nous allons discuter selon la taille de $B_0$. 
Commençons par une remarque : sous l'hypothèse \eqref{hyp_1}, on a $\varepsilon<\left(\frac{\delta}{3K}\right)^3$, ce qui implique quel que soit $K\geqslant 2$
\begin{equation}\label{eps<1/2tauK}
\varepsilon<\frac{1}{2(\tau_K+1)}.
\end{equation}
\textbf{\underline{A) Premier cas } : $\lambda(B_0)> \tau_{K}\varepsilon b $}. \\
\\
Si $\lambda(B_0)> \tau_{K}\varepsilon b $, par \eqref{maj_x+BsansA'+B}, on a 
$$\lambda\big( (x+B_0)\cap (A'+B)\big)\geqslant\lambda(B_0)-\tau_{K}\varepsilon b,$$
or comme $(x+B_0)\subseteq x+\left[0,b_+\right]$, $(A'+B)\subseteq\left[\min A',+\infty\right[$ et $x<\min A'$, on a
$$\lambda\big( \left[ \min A',x+b_+\right]\big)=\lambda\big( \left[ x,x+b_+\right]\cap \left[\min A',+\infty\right[\big)\geqslant\lambda(B_0)-\tau_{K}\varepsilon b.$$
Ainsi comme $b_+\leqslant\lambda(B_0)+\varepsilon b$ par \eqref{maj_b+}, on a 
$$x\geqslant\lambda(B_0)-\tau_K\varepsilon b -(b_+-\min A')\geqslant \min A'-(\tau_{K}+1)\varepsilon b.$$

\textbf{\underline{B) Second cas } : $\lambda(B_0)\leqslant \tau_{K}\varepsilon b $}. \\
\\
Si $\lambda(B_0)\leqslant \tau_{K}\varepsilon  b $, alors $\lambda(B_1)\geqslant b- \tau_{K}\varepsilon  b >\tau_{K}\varepsilon  b $ par \eqref{eps<1/2tauK} et donc nécessairement par le lemme \ref{maj_erreur_AsansA'}, on a
$$(x+B_1)\cap (A'+B)\neq\varnothing.$$
On rappelle que par construction $A'_0+B_1\subseteq A'_1+B_0$, il y a donc trois possibilités :

a) $(x+B_1)\cap (A'_0+B_0)\neq\varnothing$ et $(x+B_1)\cap (A'_1+B_0)\neq\varnothing$

b) $(x+B_1)\cap (A'_0+B_0)\neq\varnothing$ et $(x+B_1)\cap (A'_1+B_0)=\varnothing$

c) $(x+B_1)\cap (A'_0+B_0)=\varnothing$ et $(x+B_1)\cap (A'_1+B_0)\neq\varnothing$ \\
que nous allons traiter séparément. \\
\\
\textbf{a) le cas $\boldsymbol{(x+B_1)\cap (A'_0+B_0)\neq\varnothing$ et $(x+B_1)\cap (A'_1+B_0)\neq\varnothing}$ est impossible.} \\
$$ $$
$$\begin{minipage}[l]{17cm}
\includegraphics[height=1.5cm]{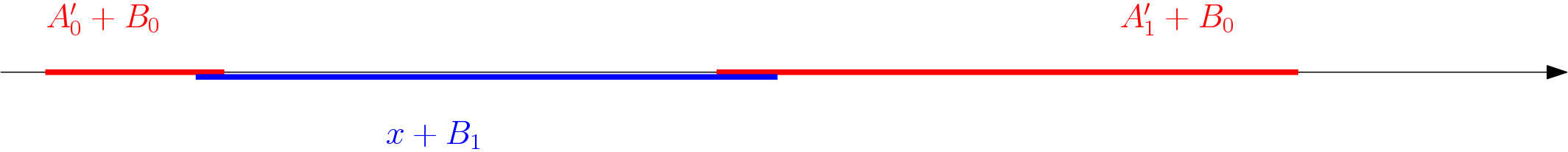}
\end{minipage}$$
$$ $$
En effet, pour qu'on soit dans cette configuration, il faudrait que l'écart entre les deux premiers morceaux de $A'$ soit inférieur au diamètre de $B_1$, or nous sommes dans le second cas et $\lambda(B_0)$ est petit donc les deux premiers morceaux $A'_0$ et $A'_1$ sont petits et donc l'écart entre eux est grand. C'est l'idée que nous allons développer ici. 

On suppose donc que $(x+B_1)\cap (A'_0+B_0)\neq\varnothing$ et $(x+B_1)\cap (A'_1+B_0)\neq\varnothing$. Ainsi nécessairement, $\left[\max A'_0+b_+ , \min A'_1 \right]\subseteq x+\left[1-b_-,1\right]$ et donc par \eqref{EinterF}
\begin{align*}
\lambda\big( (x+B_1)\setminus(A'+B)\big) & \geqslant \lambda\big( (x+B_1)\cap\left[\max A'_0+b_+ , \min A'_1 \right]\big) \\ & \geqslant \lambda\big(\left[\max A'_0+b_+ , \min A'_1 \right]\big)+\lambda(x+B_1)-\lambda\big(x+\left[1-b_-,1\right]\big) \\ & \geqslant \lambda\big(\left[\max A'_0+b_+ , \min A'_1 \right]\big)-\varepsilon b.
\end{align*}
Or
\begin{align*}
\lambda\big( \left[ \max A'_0+b_+ , \min A'_1 \right]\big) & =1-\big( \delta b+\varepsilon b+(K-1)b_++b_+\big)-b_- \ \ \text{ par définition de } A' \\ & \geqslant 1-\big( \delta b+\varepsilon b+(K-1)b_+\big)-(b+\varepsilon b) \ \ \text{ par \eqref{maj_b-}} \\ & \geqslant 1-\big( \delta b+\varepsilon b+(K-1)(b-b_-+\varepsilon b)\big)-(b+\varepsilon b) \ \text{ par \eqref{maj_b++b-}} \\ & \geqslant 1-\big( K+\delta+2\varepsilon \big)b+(K-1)b_- -(K-1)\varepsilon b \\ & \geqslant (K-1)b_- -2(K-1)\varepsilon b \ \ \text{ par \eqref{hyp_0}} \\ & \geqslant  (K-1)\big( b_--2\varepsilon b\big) \\ & \geqslant  \big( b_--2\varepsilon b\big) \ \ \text{ car $K\geqslant 2$ (cf. remarque \ref{K>1})}.
\end{align*}
Ainsi
$$\lambda\big( (x+B)\setminus(A'+B)\big)\geqslant \lambda\big( (x+B_1)\setminus(A'+B)\big)\geqslant b_--2\varepsilon b,$$
et donc finalement par le lemme \ref{maj_erreur_AsansA'} et parce qu'on est dans le cas \textbf{B}
$$\tau_{K}\varepsilon b \geqslant b-\tau_{K}\varepsilon  b-2\varepsilon b,$$
ce qui contredit \eqref{eps<1/2tauK}. Ce cas de figure est donc impossible. \\
\\
\textbf{b) le cas $\boldsymbol{(x+B_1)\cap (A'_0+B_0)\neq\varnothing$ et $(x+B_1)\cap (A'_1+B_0)=\varnothing}$ est impossible.} \\
\\
On suppose par l'absurde que $x\in A$ est tel que $x<\min A'$, $\lambda(B_0)\leqslant \tau_{K}\varepsilon  b $ et 
$$ \left\lbrace \begin{array}{ll} (x+B_1)\cap (A'_0+B_0)\neq\varnothing \\ (x+B_1)\cap (A'_1+B_0)=\varnothing \end{array} \right. .$$
$$ $$
$$\begin{minipage}[l]{17cm}
\includegraphics[height=1.4cm]{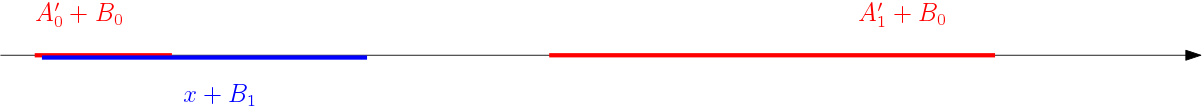}
\end{minipage}$$
$$ $$
Commençons par voir que $(x+B_0)\cap (A'+B)=\varnothing$. Supposons que ce ne soit pas le cas. Comme $(x+B_1)\cap (A'_0+B_0)\neq\varnothing$, la distance entre $B_0$ et $B_1$ est inférieure à la mesure de $A'_0+B_0$ car $A'_0+B_0$ est un intervalle.
$$ $$
$$\begin{minipage}[l]{17cm}
\includegraphics[height=2.5cm]{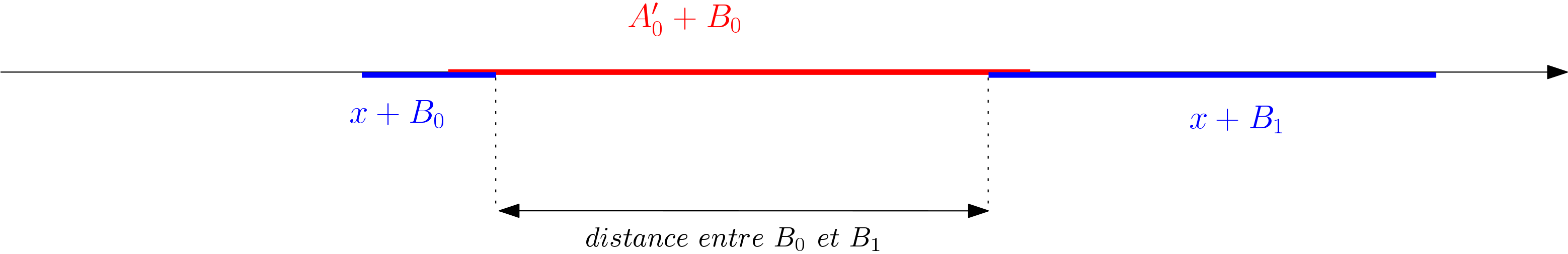}
\end{minipage}$$
$$ $$
Or $d\big( B_0 , B_1\big)\geqslant 1-b_+-b_- \geqslant 1-b-\varepsilon b,$ et $\lambda (A'_0+B_0)\leqslant\delta b+K\varepsilon b +Kb_+$. Donc 
$$ 1-b-\varepsilon b \leqslant \delta b+K\varepsilon b +Kb_+\leqslant (\delta+2K\varepsilon +K)b-Kb_-,$$
par \eqref{maj_b++b-}, donc comme $(K+\delta+\varepsilon)b<1-\varepsilon b$ par l'hypothèse \eqref{hyp_0}, on a en utilisant $\lambda(B_0)\leqslant \tau_{K}\varepsilon  b$,
$$\left( K(1-\tau_K \varepsilon)-1\right) b\leqslant Kb_--b\leqslant (2K-1)\varepsilon b, $$
ce qui contredit l'hypothèse \eqref{hyp_1} quel que soit $K\geqslant 2$. Ainsi on a bien $(x+B_0)\cap (A'+B)=\varnothing$. 
Donc d'après le lemme \ref{maj_erreur_AsansA'}, 
\begin{align*}
\tau_{K}\varepsilon  b & \geqslant \lambda\big( (x+B)\setminus (A'+B)\big) \geqslant \lambda\big( (x+B_0)\setminus (A'+B)\big)+\lambda\big( (x+B_1)\setminus (A'+B)\big) \\ & \geqslant \lambda(B_0)+\lambda\big( (x+B_1)\setminus (A'_0+B_0)\big) \geqslant \lambda(B_0)+\lambda(B_1)-\lambda \big((A'_0+B_0)\cap(x+B_1)\big) \\ & \geqslant b-\lambda (A'_0+B_0),
\end{align*}
or par construction de $A'$ puis par \eqref{maj_b+} puis parce qu'on est dans le cas \textbf{B}, on a
\begin{align*}
\lambda (A'_0+B_0) & \leqslant\delta b+K\varepsilon b+(K-1)b_++b_+ \leqslant \delta b+K\varepsilon b+K(\lambda(B_0)+\varepsilon b)\\ & \leqslant \delta b+\big(  K+K(\tau_K+1\big)\varepsilon b.
\end{align*}
Ainsi
$$\tau_{K}\varepsilon  b \geqslant (1-\delta) b-\big(  K+K(\tau_K+1\big)\varepsilon b,$$
ce qui contredit l'hypothèse \eqref{hyp_2}. Finalement, on doit donc avoir : \\
\\
\textbf{c)  $\boldsymbol{(x+B_1)\cap (A'_0+B_0)=\varnothing$ et $(x+B_1)\cap (A'_1+B_0)\neq\varnothing}$.}
$$ $$
$$\begin{minipage}[l]{17cm}
\includegraphics[height=1.3cm]{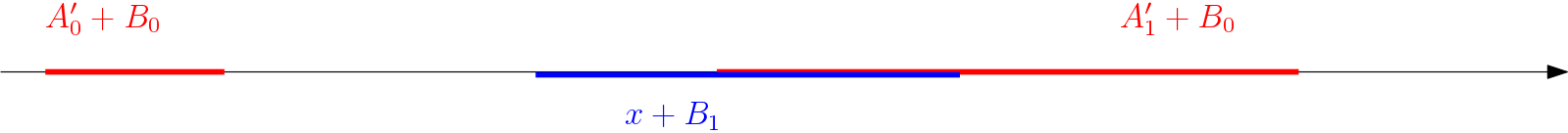}
\end{minipage}$$
$$ $$
Le lemme \ref{maj_erreur_AsansA'} assure que $ \lambda\big( (x+B)\setminus (A'+B)\big)\leqslant \tau_{K}\varepsilon  b$ mais comme $(x+B_1)\cap (A'_0+B_0)=\varnothing$ et $(x+B_1)\cap (A'_1+B_0)\neq\varnothing$, on a
$$ \lambda\big( (x+B)\setminus (A'+B)\big)=\lambda\big( (x+B_0)\setminus (A'_0+B_0)\big)+\lambda\big( (x+B_1)\setminus (A'_1+B_0)\big).$$
Ainsi
\begin{align*}
\tau_{K}\varepsilon b & \geqslant \lambda\big( (x+B_0)\setminus (A'_0+B_0)\big)+\lambda\big( (x+B_1)\setminus (A'_1+B_0)\big) \\ & \geqslant \lambda(B_0)-\lambda\big( (x+B_0)\cap (A'_0+B_0)\big) +\lambda(B_1)-\lambda\big( (x+B_1)\cap (A'_1+B_0)\big) \\ & \geqslant b-\lambda\big( \left[x,x+b_+\right] \cap\left[ \min A',+\infty\right[\big) -\lambda\big( \left[1+x-b_-,1+x\right] \cap\left[ 1+\min A'-b_-,+\infty\right[\big).
\end{align*}
Finalement, soit $\left[x,x+b_+\right] \cap\left[ \min A',+\infty\right[$ est non vide et
\begin{align*}
\tau_{K}\varepsilon  b & \geqslant b-\lambda\big( \left[\min A',x+b_+\right]\big)-\lambda\big( \left[1+\min A' -b_-,1+x\right]\big) \\ & \geqslant b-(x+b_+-\min A')-(1+x-(1+\min A'-b_-)) \\ & \geqslant 2(\min A'-x)-\varepsilon b \ \ \ \text{ d'après \eqref{maj_b++b-}},
\end{align*}
d'où
$$x\geqslant \min A'-\frac{1}{2}\big( \tau_{K}\varepsilon +\varepsilon\big) b.$$
Soit $\left[x,x+b_+\right] \cap\left[ \min A',+\infty\right[$ est vide et 
\begin{align*}
\tau_{K}\varepsilon  b & \geqslant b-\lambda\big( \left[1+\min A' -b_-,1+x\right]\big) \\ & \geqslant b-(1+x-(1+\min A'-b_-)) \\ & \geqslant \min A'-x+b-b_- \\ & \geqslant \min A'-x-\varepsilon b \ \ \ \text{ d'après \eqref{maj_b-}},
\end{align*}
d'où
$$x\geqslant \min A'-\big( \tau_{K}\varepsilon +\varepsilon\big) b.$$
Cette inégalité est donc valable dans tous les cas, ce qui termine la preuve.
\end{proof}
Traitons désormais le cas des éléments supérieurs à $\max A'$.
\begin{lemma}\label{AsansA'pastropaprès}
Soit $x\in A$ tel que $x>\max A'$, on a nécessairement
$$x\leqslant \max A'+(\tau_{K}+1)\varepsilon b,$$
où on rappelle que $\tau_K=(K(K+1)\log (K)-K\big(K(1+\log 4)-6+\log 4\big)+7$.
\end{lemma}
\begin{proof}
La preuve est immédiate en utilisant la transformation $\chi$ définie par
\begin{align*}
\chi \ : \ & \overline{\mathcal{P}}\left(\mathbb{R}\right) \longrightarrow \overline{\mathcal{P}}\left(\mathbb{R}\right) \\ & E \longmapsto \sup E-E
\end{align*}
et le lemme \ref{AsansA'pas_trop_avant}. Il faut simplement veiller à ce que $\chi(A')=\chi(A)'$, ce qui est bien le cas\footnote{Pour plus de détail voir https://hal.univ-lorraine.fr/tel-03368154v1 où l'argument y est développé.}.
\end{proof}

\subsubsection{Le cas des éléments $x$ de $A\setminus A'$ dans $\left]\max A'_k,\min A'_{k+1}\right[$ pour $k\in\left\lbrace 0,...,K-1 \right\rbrace$}

\begin{lemma}\label{A_dans_ecartsdumilieu}
Si $k\in\left\lbrace 0,...,K-1 \right\rbrace$ et si $x$ est un élément de $A$ tel que $x\in \left]\max A'_k,\min A'_{k+1}\right[$, alors
$$d\left( x,A'\right)\leqslant (\tau_{K}+1)\varepsilon b,$$
où on rappelle que $\tau_{K}= (K(K+1)\log (K)-K\big(K(1+\log 4)-6+\log 4\big)+7.$
\end{lemma}
\begin{proof} 
Soit $k\in\left\lbrace 0,...,K-1 \right\rbrace$ et $x\in A\cap\left]\max A'_k,\min A'_{k+1}\right[$. On peut supposer que $\lambda(B_0)\geqslant b/2$. Sinon cela signifie que $\lambda(B_1)\geqslant b/2$ et donc quitte à appliquer $\chi$, on se ramène à $\chi(B_0)\geqslant b/2$. Nous allons raisonner par l'absurde, supposons que $d\left( x, A'\right)>(\tau_{K}+1)\varepsilon b$, ce qui implique que 
$$\lambda\big(\left]\max A'_k,\min A'_{k+1}\right[\big)>2(\tau_{K}+1)\varepsilon b,$$ et donc 
$$x\in \left]\max A'_k+(\tau_{K}+1)\varepsilon b,\min A'_{k+1}-(\tau_{K}+1)\varepsilon b\right[ .$$
On a 
\begin{equation}\label{avantlesquatrecas}
\lambda\big(\left( x+B_0\right)\setminus (A'+B)\big) \geqslant \lambda\Big(\left( x+B_0\right)\cap \left]\max A'_k+b_+,\min A'_{k+1}\right[\Big).
\end{equation}
Or on rappelle que le lemme \ref{maj_erreur_AsansA'} donne la majoration
$$\lambda\Big( \big((A\setminus A')+B\big)\setminus \big( (A\cap A')+B\big)\Big)\leqslant \tau_{K}\varepsilon b.$$
Nous allons voir que ces deux dernières inégalités sont en contradiction. Pour cela nous allons distinguer les quatre configurations possibles : 
\begin{itemize}
\item $\left( x+B_0\right)\cap \left]-\infty,\max A'_k+b_+\right[=\varnothing$ et $\left( x+B_0\right)\cap \left]\min A'_{k+1},+\infty\right[=\varnothing$
\item $\left( x+B_0\right)\cap \left]-\infty,\max A'_k+b_+\right[\neq\varnothing$ et $\left( x+B_0\right)\cap \left]\min A'_{k+1},+\infty\right[\neq\varnothing$
\item $\left( x+B_0\right)\cap \left]-\infty,\max A'_k+b_+\right[\neq\varnothing$ et $\left( x+B_0\right)\cap \left]\min A'_{k+1},+\infty\right[=\varnothing$
\item $\left( x+B_0\right)\cap \left]-\infty,\max A'_k+b_+\right[=\varnothing$ et $\left( x+B_0\right)\cap \left]\min A'_{k+1},+\infty\right[\neq\varnothing$
\end{itemize}
et nous allons voir que quelle que soit la configuration dans laquelle on se trouve, $\lambda\big(\left( x+B_0\right)\setminus (A'+B)\big)$ est supérieur à $\tau_{K}\varepsilon b$ ce qui nous conduira donc à une absurdité. \\
\\
\textbf{\underline{Cas 1} : $\boldsymbol{\left( x+B_0\right)\cap \left]-\infty,\max A'_k+b_+\right[=\varnothing$ et $\left( x+B_0\right)\cap \left]\min A'_{k+1},+\infty\right[=\varnothing}$} \\
\\
$$\begin{minipage}[l]{17cm}
\includegraphics[height=1.9cm]{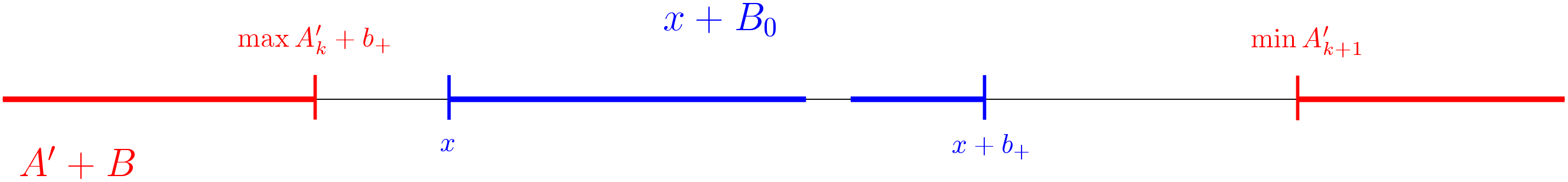}
\end{minipage}$$
\\
Dans ce cas on a alors $x+B_0\subseteq\left]\max A'_k+b_+,\min A'_{k+1}\right[$, et donc
$$\lambda\Big(\left( x+B_0\right)\cap \left]\max A'_k+b_+,\min A'_{k+1}\right[\Big)\geqslant \lambda(B_0)\geqslant\frac{b}{2} .$$
Or $\frac{b}{2}> \tau_{K}\varepsilon b$, car $\varepsilon<\left(\frac{\delta}{3K}\right)^3$ par l'hypothèse \eqref{hyp_1} et $\left(\frac{\delta}{3K}\right)^3\leqslant\frac{1}{2\tau_K}$ quel que soit $K\geqslant 2$. Ainsi finalement
$$\lambda\Big(\left( x+B_0\right)\cap \left]\max A'_k+b_+,\min A'_{k+1}\right[\Big)> \tau_{K}\varepsilon  b.$$
\textbf{\underline{Cas 2} : $\boldsymbol{\left( x+B_0\right)\cap \left]-\infty,\max A'_k+b_+\right[\neq\varnothing$ et $\left( x+B_0\right)\cap \left]\min A'_{k+1},+\infty\right[\neq\varnothing}$} \\
\\
$$\begin{minipage}[l]{17cm}
\includegraphics[height=1.6cm]{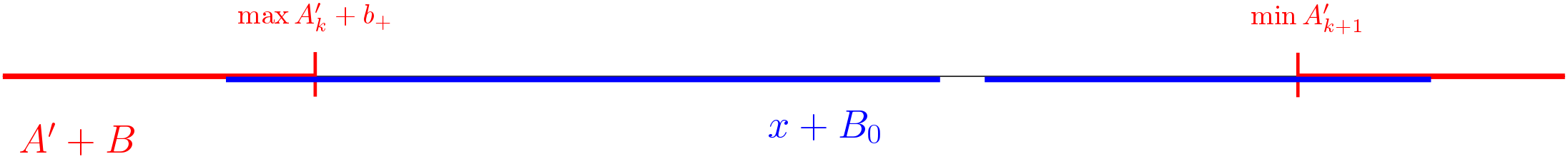}
\end{minipage}$$
\\
On rappelle ici que $B_0\subseteq\left[0,b_+\right]$ et $b_+\leqslant b+\varepsilon b$ (cf. \eqref{maj_b+}), donc
\begin{align*}
\lambda\Big(\left( x+B_0\right)\cap \left]\max A'_k+b_+,\min A'_{k+1}\right[\Big) & \geqslant \lambda\Big(\left[ x,x+b_+\right]\cap \left]\max A'_k+b_+,\min A'_{k+1}\right[\Big) \\ & \ \ \ - \lambda\Big(\left[ x,x+b_+\right]\setminus\left( x+B_0\right)\Big)\\ & \geqslant \lambda\Big( \left]\max A'_k+b_+,\min A'_{k+1}\right[\Big)-\varepsilon b,
\end{align*}
or
\begin{align*}
\lambda\Big( \left]\max A'_0+b_+,\min A'_{1}\right[\Big) & =1-(k+1)b_--(K-k)b_+-\delta b-\varepsilon b \\ & \geqslant 1-(b+\varepsilon b)-kb_-(K-k-1)b_+-\delta b-\varepsilon b \ \ \text{ par \eqref{maj_b++b-}} \\ & \geqslant 1-K(b+\varepsilon b)-\delta b-\varepsilon b \ \ \text{ par \eqref{maj_b+} et \eqref{maj_b-}} \\ & \geqslant 1-\big(K+\delta+(\tau_K+K+1)\varepsilon \big) b+\tau_{K}\varepsilon  b \\ & > \tau_{K}\varepsilon  b,
\end{align*}
où la dernière ligne provient de l'hypothèse \eqref{hyp_4}. Ainsi finalement
$$\lambda\Big(\left( x+B_0\right)\cap \left]\max A'_k+b_+,\min A'_{k+1}\right[\Big)> \tau_{K}\varepsilon  b.$$
\textbf{\underline{Cas 3} : $\boldsymbol{\left( x+B_0\right)\cap \left]-\infty,\max A'_k+b_+\right[\neq\varnothing$ et $\left( x+B_0\right)\cap \left]\min A'_{k+1},+\infty\right[=\varnothing}$} \\
\\
$$\begin{minipage}[l]{17cm}
\includegraphics[height=1.6cm]{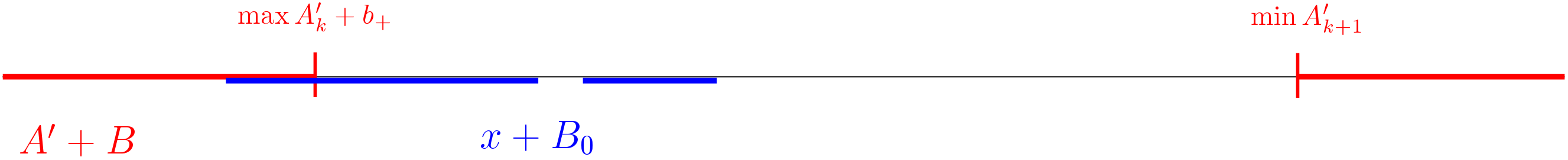}
\end{minipage}$$
\\
On a 
\begin{align*}
 & \lambda\Big(\left( x+B_0\right)\cap \left]\max A'_k+b_+,\min A'_{k+1}\right[\Big) \\ & \geqslant \lambda\Big( \big(x+\left[ 0,b_+\right]\big)\cap \left]\max A'_k+b_+,\min A'_{k+1}\right[\Big)  - \lambda\Big( \big( x+\left[ 0,b_+\right]\big)\setminus (x+B_0)\Big) \\ & \geqslant \lambda\Big( \left]\max A'_k+b_+,x+b_+\right]\Big) - \Big(\lambda\big( \left[0,b_+\right]\big)-\lambda(B_0)\Big),
\end{align*}
et comme $x\in \left]\max A'_k+(\tau_{K}+1)\varepsilon b,\min A'_{k+1}-(\tau_{K}+1)\varepsilon b\right[$, $B_0\subseteq\left[0,b_+\right]$ et $b_+\leqslant b+\varepsilon b$, on a
$$ \lambda\Big(\left( x+B_0\right)\cap \left]\max A'_k+b_+,\min A'_{k+1}\right[\Big)> (\tau_K+1)\varepsilon b-\varepsilon b,$$
c'est à dire
$$\lambda\Big(\left( x+B_0\right)\cap \left]\max A'_k+b_+,\min A'_{k+1}\right[\Big)> \tau_{K}\varepsilon  b.$$
\textbf{\underline{Cas 4} : $\boldsymbol{\left( x+B_0\right)\cap \left]-\infty,\max A'_k+b_+\right[=\varnothing$ et $\left( x+B_0\right)\cap \left]\min A'_{k+1},+\infty\right[\neq\varnothing}$} \\
\\
$$\begin{minipage}[l]{17cm}
\includegraphics[height=1.6cm]{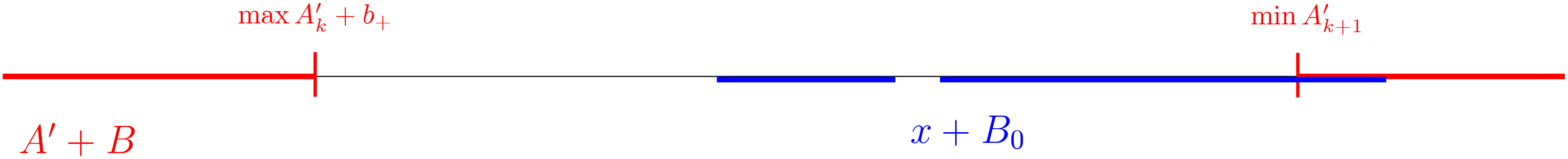}
\end{minipage}$$
\\
Ce cas est similaire au précédent et on a
$$ \lambda\Big(\left( x+B_0\right)\cap \left]\max A'_k+b_+,\min A'_{k+1}\right[\Big)> \tau_K\varepsilon b.$$

Ainsi quel que soit le cas dans lequel on se trouve, on obtient toujours une contradiction du lemme \ref{maj_erreur_AsansA'} et donc l'absurdité. Ainsi $d\left( x, A'\right)\leqslant (\tau_K+1)\varepsilon b$ ce qui termine la preuve de ce lemme.
\end{proof}
Finalement les lemmes \ref{AsansA'pas_trop_avant}, \ref{AsansA'pastropaprès} et \ref{A_dans_ecartsdumilieu} nous amènent à la conclusion suivante 
$$A\subseteq A'+\left[-(\tau_{K}+1)\varepsilon b,(\tau_{K}+1)\varepsilon b\right], $$
où $\tau_{K}= (K(K+1)\log (K)-K\big(K(1+\log 4)-6+\log 4\big)+7 $ et
$$ A'= a'+\bigsqcup\limits_{\substack{k=0}}^{K-1} \left[ k-kb_-, k+(\delta+\varepsilon) b +(K-k-1)b_+ \right] .$$

\subsubsection{Affinement}

Pour tout $i\in\left\lbrace 1,...,K+1 \right\rbrace$, on pose $S_i=(A+B)\cap\big( a'-Kb_-+\left[i-1,i\right[\big)$ et pour tout $i\in\left\lbrace 1,...,K \right\rbrace$, on pose 
$$A_i=A\cap\Big( a'\big[ i-ib_--(\tau_{K}+1)\varepsilon b,i+(K-i-1)b_+ +\delta b+(\tau_{K}+2)\varepsilon b\big]\Big)$$  
et 
$$f_i(\varepsilon)b=\lambda(A_i)-(K-i)\lambda(B_0)-(i-1)\lambda(B_1)-\delta b.$$
On a donc
\begin{align*}
\lambda(A) & =\sum\limits_{\substack{i=1}}^{K}\lambda(A_i)=\sum\limits_{\substack{i=1}}^{K}\big( (K-i)\lambda(B_0)+(i-1)\lambda(B_1)+\delta b+f_i(\varepsilon)b\big) \\ & =\dfrac{K(K-1)}{2}b+K\delta b+\sum\limits_{\substack{i=1}}^{K}f_i(\varepsilon)b = \lambda(A)+\sum\limits_{\substack{i=1}}^{K}f_i(\varepsilon)b,
\end{align*}
et donc $\sum\limits_{\substack{i=1}}^{K}f_i(\varepsilon)=0$. Ainsi il existe $i_0\in\left\lbrace 1,...,K \right\rbrace$ tel que $f_{i_0}(\varepsilon)\geqslant 0.$
\begin{lemma}\label{bonne_taille_des_morceaux}
Quel que soit $k\in\left\lbrace 1,...,K \right\rbrace$, $A_k$ est inclus dans un intervalle de mesure $(K-k)\lambda(B_0)+(k-1)\lambda(B_1)+\delta b+\varepsilon b$.
\end{lemma}
\begin{proof}
Soit $k\in\left\lbrace 1,...,K \right\rbrace$. Commençons par donner une majoration de la mesure de $A_k$. Par l'hypothèse \eqref{hyp_4}, pour tout $i\in\left\lbrace 1,...,K \right\rbrace$, on a $A_i+B_0\subseteq S_i$ et $A_i+B_1\subseteq S_{i+1}$. De plus, quel que soit $i\in\left\lbrace 1,...,K \right\rbrace$, $a'+i+\tilde{A}_K\subseteq A_i$ et donc en particulier $A_i\neq\varnothing$. Ainsi
\begin{align*}
\lambda(A+B) & =\sum\limits_{\substack{i=1}}^{K+1}\lambda(S_i)\geqslant\sum\limits_{\substack{i=1}}^{k}\lambda(A_i+B_0)+\sum\limits_{\substack{i=k}}^{K}\lambda(A_i+B_1) \\ & \geqslant \sum\limits_{\substack{i=1}}^{k}\lambda(A_i)+k\lambda(B_0)+\sum\limits_{\substack{i=k}}^{K}\lambda(A_i)+(K-k+1)\lambda(B_1) \\ & \geqslant \lambda(A)+k\lambda(B_0)+(K-k+1)\lambda(B_1)+\lambda(A_k),
\end{align*}
or $\lambda(A+B)=\lambda(A)+(K+\delta+\varepsilon)b$ par hypothèse, donc
$$f_k(\varepsilon)b=\lambda(A_k)-(K-k)\lambda(B_0)-(k-1)\lambda(B_1)-\delta b\leqslant \varepsilon b .$$
De plus, comme $\sum\limits_{\substack{i=1}}^{K}f_i(\varepsilon)=0$, on a $f_k(\varepsilon)=-\sum\limits_{\substack{i\neq k}}f_i(\varepsilon) \geqslant -(K-1)\varepsilon $.
Ainsi, par l'hypothèse \eqref{hyp_3}, on a
\begin{equation}\label{onpeututiliserAnnePablo}
0\leqslant\varepsilon b-f_k(\varepsilon)b\leqslant K\varepsilon b\leqslant \rho_0.
\end{equation}
D'autre part, on a
\begin{align*}
\lambda(A+B) & =\sum\limits_{\substack{i=1}}^{K+1}\lambda(S_i)\geqslant\sum\limits_{\substack{i=1}}^{k}\lambda(A_i+B_0)+\sum\limits_{\substack{i=k}}^{K}\lambda(A_i+B_1) \\ & \geqslant \sum\limits_{\substack{i=1}}^{k}\lambda(A_i)+k\lambda(B_0)+\lambda(A_k+B_1)+\sum\limits_{\substack{i=k+1}}^{K}\lambda(A_i)+(K-k)\lambda(B_1) \\ & \geqslant \lambda(A)+k\lambda(B_0)+(K-k)\lambda(B_1)+\lambda(A_k+B_1).
\end{align*}
Or encore une fois, $\lambda(A+B)=\lambda(A)+(K+\delta+\varepsilon)\lambda(B)$ par hypothèse, donc
$$\lambda(A_k+B_1)\leqslant (K-k)\lambda(B_0)+k\lambda(B_1)+\delta b+\varepsilon b,$$
d'où par \eqref{onpeututiliserAnnePablo}
\begin{equation}\label{Pour_B1}
\lambda(A_k+B_1)\leqslant \lambda(A_k)-f_k(\varepsilon )b+\lambda(B_1)+\varepsilon b< \lambda(A_k)+\lambda(B_1)+\rho_0.
\end{equation}
De même, on a les inégalités 
\begin{align*}
\lambda(A+B) & =\sum\limits_{\substack{i=1}}^{K+1}\lambda(S_i)\geqslant\sum\limits_{\substack{i=1}}^{k}\lambda(A_i+B_0)+\sum\limits_{\substack{i=k}}^{K}\lambda(A_i+B_1) \\ & \geqslant \sum\limits_{\substack{i=1}}^{k-1}\lambda(A_i)+(k-1)\lambda(B_0)+\lambda(A_k+B_0)+\sum\limits_{\substack{i=k}}^{K}\lambda(A_i)+(K-k+1)\lambda(B_1) \\ & \geqslant \lambda(A)+(k-1)\lambda(B_0)+(K-k+1)\lambda(B_1)+\lambda(A_k+B_0),
\end{align*}
qui entraînent par \eqref{onpeututiliserAnnePablo}
\begin{equation}\label{pour_B0}
\lambda(A_k+B_0)\leqslant \lambda(A_k)+\lambda(B_0)+\big(\varepsilon-f_k(\varepsilon )\big) b< \lambda(A_k)+\lambda(B_0)+\rho_0.
\end{equation}
Comme $\max\big( \lambda(B_0),\lambda(B_1) \big)\geqslant b/2>K\varepsilon b\geqslant \big(\varepsilon -f_k(\varepsilon )\big)b$, au moins l'une des inéquations \eqref{Pour_B1} et \eqref{pour_B0} nous permet d'utiliser le théorème \ref{Anne_Pablo}. Sans perdre en généralité, supposons donc que $\lambda(B_0)>K\varepsilon b$. Par \eqref{pour_B0} et d'après le théorème \ref{Anne_Pablo}, il existe $n_k\in\mathbb{N}^*$ tel que $n_k B_0\subseteq J'_k$, $n_k A_k\subseteq I'_k$, $\mu(J'_k)\leqslant\mu(B_0)+\big(\varepsilon-f_k(\varepsilon )\big) b$ et $\mu(I'_k)\leqslant\mu(A_k)+\big(\varepsilon-f_k(\varepsilon )\big) b$. Cependant, $B_0\subseteq\left[ 0,b_+\right]$ donc d'après le lemme \ref{tildes-presque-intervalles}, nécessairement $n_k=1$. Finalement, $A_k\subseteq I'_k$ et par définition de $f_k(\varepsilon )$, on a
$$\mu(I'_k)  \leqslant\mu(A_k)+\big(\varepsilon-f_k(\varepsilon )\big) b \leqslant (K-k)\lambda(B_0)+(k-1)\lambda(B_1)+\delta b+\varepsilon b.$$
\end{proof}
Ce lemme signifie que chaque morceau de $A$ ne peut pas être plus gros que le morceau de $A_0$ correspondant plus $\varepsilon b$. Il reste à montrer plus précisément que $A\subseteq a'+A_0+\left[0,\varepsilon b\right]$. Commençons par donner une majoration de $\lambda(A_k+B_0)$ pour tout $k\in\left\lbrace 1,...,K \right\rbrace$, ce sera l'objet du lemme \ref{min_Ak+B0} ci après.

Quel que soit $k\in\left\lbrace 1,...,K \right\rbrace$, on rappelle qu'on a posé $f_k(\varepsilon )b=\lambda(A_k)-\big( (K-k)\lambda(B_0)+(k-1)\lambda(B_1)+\delta b\big)$, on pose également $g_k(\varepsilon) b=\operatorname{diam}(A_k)-\lambda(A_k)$. Rappelons qu'on a
$$-(K-1)\varepsilon\leqslant f_k(\varepsilon )\leqslant\varepsilon.$$
Ainsi, par le lemme \ref{bonne_taille_des_morceaux}, on a 
\begin{equation}\label{encadrement_gk}
0\leqslant g_k(\varepsilon )\leqslant K\varepsilon.
\end{equation}

Ruzsa \cite{Ruzsa_minoration_AplusB} a prouvé le lemme suivant.
\begin{lemma}[Ruzsa]\label{lemmeRuzsa2}
Soient $E$ et $F$ deux ensembles bornés, non vides de réels. On a soit
$$ \lambda(E+F) \geqslant \lambda(E)+\operatorname{diam}(F)$$
soit
$$\lambda(E+F) \geqslant \frac{k+1}{k} \lambda(E)+\frac{k+1}{2} \lambda(F)$$
où $k$ est l'entier naturel défini par
$$ k=\max \left\{k^{\prime} \in \mathbb{N} \mid \exists x \in\left[0, \operatorname{diam}(F)\right[, \#\left\{n \in \mathbb{N} \mid x+n \operatorname{diam}(F) \in E\right\} \geqslant k^{\prime}\right\}. $$
\end{lemma}
Quitte à appliquer $\chi$, on peut supposer que $b_+\geqslant b_-$.
\begin{lemma}\label{min_Ak+B0}
Quel que soit $k\in\left\lbrace 1,...,K \right\rbrace$, on a
$$\lambda(A_k+B_0)\geqslant \lambda(A_k)+\lambda(B_0)+g_k(\varepsilon)b.$$
\end{lemma}
\begin{proof}
Soit $k\in\left\lbrace 1,...,K-1 \right\rbrace$, nous traiterons le cas $k=K$ à part. Par hypothèse $\delta>\varepsilon$, donc 
$$\operatorname{diam}(A_k)\geqslant\lambda(A_k)>\lambda(B_0)+\varepsilon b\geqslant \operatorname{diam}(B_0).$$
Ainsi
$$\max \left\{k^{\prime} \in \mathbb{N} \mid \exists x \in\left[0, \operatorname{diam}(A_k)\right[, \#\left\{n \in \mathbb{N} \mid x+n \operatorname{diam}(A_k) \in B_0\right\} \geqslant k^{\prime}\right\}=1,$$
et donc le lemme \ref{lemmeRuzsa2} implique
$$\lambda(A_k+B_0)\geqslant\min\big( \lambda(B_0)+\operatorname{diam}(A_k),\lambda(A_k)+2\lambda(B_0)\big).$$
Or $g_k(\varepsilon )\leqslant K\varepsilon$ (cf. \eqref{encadrement_gk}) et $K\varepsilon\leqslant\lambda(B_0)$ par hypothèse (et parce qu'on a supposé $b_+\geqslant b_-$), donc $\lambda(B_0)+\operatorname{diam}(A_k)\leqslant\lambda(A_k)+2\lambda(B_0)$, et
$$\lambda(A_k+B_0)\geqslant \lambda(B_0)+\operatorname{diam}(A_k)=\lambda(A_k)+\lambda(B_0)+g_k(\varepsilon )b.$$
Il ne reste qu'à traiter le cas $k=K$. Posons
$$l=\max \left\{k^{\prime} \in \mathbb{N} \mid \exists x \in\left[0, \operatorname{diam}(A_K)\right[, \#\left\{n \in \mathbb{N} \mid x+n \operatorname{diam}(A_K) \in B_0\right\} \geqslant k^{\prime}\right\}.$$
Si $l=1$ alors il suffit de reprendre le stratégie précédente. Supposons donc désormais que $l\geqslant 2$. Par le lemme \ref{lemmeRuzsa2}, on a
$$\lambda(A_K+B_0)\geqslant\min\big( \lambda(B_0)+\operatorname{diam}(A_k),\dfrac{l+1}{2}\lambda(A_K)+\dfrac{l+1}{l}\lambda(B_0)\big) .$$
Or comme $l\geqslant 2$, on a
$$\dfrac{l+1}{2}\lambda(A_K)+\dfrac{l+1}{l}\lambda(B_0)\geqslant\lambda(A_K)+\dfrac{1}{2}\lambda(A_K)+\lambda(B_0).$$
De plus 
$$\lambda(A_K)=(K-1)\lambda(B_1)+\delta b+f_K(\varepsilon)b\geqslant\delta b-K\varepsilon b,$$
et $\delta\geqslant (3K+2)\varepsilon$ par l'hypothèse \eqref{hyp_1}. Ainsi $\frac{1}{2}(\delta-K\varepsilon)b\geqslant (K+1)\varepsilon b$ et donc
\begin{align*}
\dfrac{1}{2}\lambda(A_K)\geqslant (K+1)\varepsilon b\geqslant g_K(\varepsilon)b.
\end{align*}
Finalement on a bien également
$$\lambda(A_K+B_0)\geqslant \lambda(A_K)+\lambda(B_0)+g_K(\varepsilon)b.$$
\end{proof}
Nous sommes désormais prêts à conclure. Comme $(A_i+B_0)\sqcup\big((A_{i-1}+B_1)\setminus(A_i+B_0)\big)\subseteq S_i$ pour tout $i\in\left\lbrace 2,...,K \right\rbrace$, on a
$$\lambda(S_i)\geqslant\lambda(A_i+B_0)+\lambda\big((A_{i-1}+B_1)\setminus(A_i+B_0)\big).$$
Ainsi par hypothèse, on a
\begin{align*}
\lambda(A)+(K+\delta+\varepsilon)b & =\lambda(A+B) =\sum\limits_{\substack{i=1}}^{K+1}\lambda(S_i) \\ & \geqslant\sum\limits_{\substack{i=1}}^{K}\lambda(A_i+B_0)+\sum\limits_{\substack{i=2}}^{K}\lambda\big((A_{i-1}+B_1)\setminus(A_i+B_0)\big)+\lambda(A_K+B_1),
\end{align*}
et donc par le lemme \ref{min_Ak+B0} et par définition de $f_K(\varepsilon)$
\begin{align*}
\lambda(A)+(K+ & \delta +\varepsilon)b \\ & \geqslant \sum\limits_{\substack{i=1}}^{K}\big(\lambda(A_i)+\lambda(B_0)+g_i(\varepsilon)b\big)+\sum\limits_{\substack{i=2}}^{K}\lambda\big((A_{i-1}+B_1)\setminus(A_i+B_0)\big) +\lambda(A_K)+\lambda(B_1) \\ & \geqslant \lambda(A)+(K+\delta)b+\sum\limits_{\substack{i=2}}^{K}\lambda\big((A_{i-1}+B_1)\setminus(A_i+B_0)\big) + \sum\limits_{\substack{i=1}}^{K}g_i(\varepsilon)b+f_K(\varepsilon)b.
\end{align*}
Ceci entraîne
$$0\leqslant\sum\limits_{\substack{i=2}}^{K}\lambda\big((A_{i-1}+B_1)\setminus(A_i+B_0)\big)\leqslant \varepsilon b-\sum\limits_{\substack{i=1}}^{K}g_i(\varepsilon)b-f_K(\varepsilon)b,$$
et donc il existe des réels positifs $\varepsilon^4_2,...,\varepsilon^4_K$ tels que
\begin{equation}\label{somme_eps4}
\sum\limits_{\substack{i=2}}^{K}\varepsilon^4_i=\varepsilon -\sum\limits_{\substack{i=1}}^{K}g_i(\varepsilon)-f_K(\varepsilon),
\end{equation}
et pour tout $i\in\left\lbrace 2,...,K \right\rbrace$
\begin{equation}\label{maj_par_eps4}
\lambda\big((A_{i-1}+B_1)\setminus(A_i+B_0)\big)\leqslant\varepsilon^4_i b.
\end{equation}
De plus pour tout $i\in\left\lbrace 2,...,K \right\rbrace$, on a
\begin{align*}
& (A_{i-1}+B_1)\cap\big(\left[\min(A_{i-1}+B_1) ,\min (A_i+B_0) \right]\sqcup\left[\max(A_i+B_0) ,\max (A_{i-1}+B_1) \right]\big) \\ & \ \ \ \subseteq \big((A_{i-1}+B_1)\setminus(A_i+B_0)\big) ,
\end{align*}
notons que les intervalles peuvent être vides. On peut donc définir deux réels positifs ou nuls $\varepsilon^<_i$ et $\varepsilon^>_i$ tels que
\begin{equation}\label{maj_eps-}
\lambda\big((A_{i-1}+B_1)\cap\left[\min(A_{i-1}+B_1) ,\min (A_i+B_0) \right]\big)=\varepsilon^<_i b,
\end{equation}
et
\begin{equation}\label{maj_eps+}
\lambda\big((A_{i-1}+B_1)\cap\left[\max(A_i+B_0) ,\max (A_{i-1}+B_1) \right]\big)=\varepsilon^>_i b.
\end{equation}
On a alors
\begin{equation}\label{eps-+eps+}
\varepsilon^<_i+\varepsilon^>_i\leqslant\varepsilon^4_i ,
\end{equation}
On rappelle que $\operatorname{diam}(A_{i-1})=\lambda(A_{i-1})+g_{i-1}(\varepsilon) b$. Comme $A_{i-1}+1-b_-\subseteq (A_{i-1}+B_1)$, $A_{i-1}+1\subseteq (A_{i-1}+B_1)$ et 
\begin{align*}
d\Big(\left[\min(A_{i-1}+B_1) ,\min (A_i+B_0) \right],\left[\max(A_i+B_0) ,\max (A_{i-1}+B_1) \right]\Big) & =\operatorname{diam}(A_{i}+B_0)\\ & >b_-,
\end{align*}
il existe deux réels positifs $g^<_{i-1}(\varepsilon)$ et $g^>_{i-1}(\varepsilon)$ tels que
\begin{align*}
& \lambda\big((A_{i-1}+1-b_-)\cap\left[\min(A_{i-1}+B_1) ,\min (A_i+B_0) \right]\big) \\ & \hspace*{6.5cm} \geqslant \lambda\big(\left[\min(A_{i-1}+B_1) ,\min (A_i+B_0) \right]\big)  - g^<_{i-1}(\varepsilon)b ,
\end{align*}
\begin{align*}
& \lambda\big((A_{i-1}+1)\cap\left[\max(A_i+B_0) ,\max (A_{i-1}+B_1) \right]\big) \\ & \hspace*{6.5cm} \geqslant \lambda\big(\left[\max(A_i+B_0) ,\max (A_{i-1}+B_1) \right]\big)  - g^>_{i-1}(\varepsilon)b ,
\end{align*}
et 
\begin{equation}\label{g<+g>=g}
g^<_{i-1}(\varepsilon)+g^>_{i-1}(\varepsilon)\leqslant g_{i-1}(\varepsilon).
\end{equation}
Or soit $\min(A_{i-1}+B_1) >\min (A_i+B_0)$ soit
\begin{align*}
\lambda\big(\left[\min(A_{i-1}+B_1) ,\min (A_i+B_0) \right]\big) & = \min (A_i+B_0)-\min(A_{i-1}+B_1) \\ & = \min A_i-(\min A_{i-1}+1-b_-),
\end{align*}
et donc par \eqref{maj_eps-}
$$\min A_i-(\min A_{i-1}+1-b_-)\leqslant \varepsilon^<_i b+g^<_{i-1}(\varepsilon)b.$$
Finalement comme $\varepsilon^<_i b$ et $g^<_{i-1}(\varepsilon)b$ sont positifs, dans tous les cas on a
\begin{equation}\label{min_etage_dessus}
\min A_{i-1}\geqslant \min A_i-1+b_--\varepsilon^<_i b-g^<_{i-1}(\varepsilon)b.
\end{equation}
De la même manière, par \eqref{maj_eps+} on établit l'inégalité
\begin{equation}\label{max_etage_dessus}
\max A_{i-1}\leqslant \max A_i-1+b_++\varepsilon^>_i b+g^>_{i-1}(\varepsilon)b.
\end{equation}
Ces deux dernières inégalités étaient le maillon manquant afin de terminer cette preuve. En effet, en posant
$$a_0=\min A_K-\sum\limits_{\substack{i=2}}^{K}\varepsilon^<_ib-\sum\limits_{\substack{i=1}}^{K-1}g^<_i(\varepsilon)b,$$
on a $\operatorname{diam}(A_{K})=(K-1)\lambda(B_1)+\delta b +f_K(\varepsilon)b+g_K(\varepsilon)b$ et donc par \eqref{eps-+eps+}, \eqref{g<+g>=g} puis \eqref{somme_eps4}
\begin{align*}
\operatorname{diam}(A_{K})+\lambda(\left[a_0,\min A_K\right]) & =\operatorname{diam}(A_{K})+\sum\limits_{\substack{i=2}}^{K}\varepsilon^<_ib+\sum\limits_{\substack{i=1}}^{K-1}g^<_i(\varepsilon)b \\ & \leqslant \operatorname{diam}(A_{K})+\sum\limits_{\substack{i=2}}^{K}\varepsilon^4_ib+\sum\limits_{\substack{i=1}}^{K-1}g_i(\varepsilon)b \\ & \leqslant \operatorname{diam}(A_{K})+\varepsilon b-(f_K(\varepsilon)b+g_K(\varepsilon)b) \\ & \leqslant (K-1)\lambda(B_1)+\delta b +\varepsilon b.
\end{align*}
Ainsi
$$A_K\subseteq a_0+\left[K-1-(K-1)b_-,K-1+\delta b+\varepsilon b\right],$$
et donc $A_K\subseteq a_0+A'+\left[ 0,\varepsilon b\right].$ De même, par \eqref{min_etage_dessus} et \eqref{max_etage_dessus}, quel que soit $k\in\left\lbrace 1,...,K-1 \right\rbrace$ on a
$$A_k\subseteq a_0+A'+\left[ 0,\varepsilon b\right],$$
et donc finalement $A$ est inclus dans un translaté de $A'+\left[ 0,\varepsilon b\right]$, ce qui termine la preuve du théorème \ref{main_result}.

\section{Remarques}

Dans le cas symétrique $A=B$ le théorème \ref{main_result} rejoint la conclusion d'un précédent résultat de Candela et de Roton (\cite{Anne_Pablo} théorème 4.1). 
\begin{theorem}[Candela et de Roton]
Soit $0<\varepsilon<1/3$ tel que $\varepsilon<10^{-4}$. Soit $A\subset\left[0,1\right]$ un ensemble fermé, de diamètre $1$, de mesure non nulle et tel que $\lambda(A+A)\leqslant (3+\varepsilon )\lambda(A)$ et $\lambda(A)<\frac{1}{2(1+\varepsilon)}$. Alors $A\mod 1$ est inclus dans un intervalle $I\subset\mathbb{T}$ tel que $\mu(I)\leqslant (1+\varepsilon)\lambda(A)$.
\end{theorem}
Même en dehors du cas symétrique, la conclusion du théorème \ref{main_result} est en fait optimale en un certain sens car si nous prenons $B=\left[0,\lambda(B)\right]\cup\left\lbrace 1 \right\rbrace$ et pour tout $i\in\left\lbrace 1,...,K \right\rbrace$ et tout $0\leqslant\varepsilon'\leqslant\varepsilon$,
$$A_i(\varepsilon')=A_0\cup \Big( i-1+(K-i)b+\delta b+\left[ 0 ,\varepsilon' b\right]\Big)\cup \Big( (i-1)+\left[ -(\varepsilon-\varepsilon') b,0\right]\Big),$$
où on rappelle que $b=\lambda(B)$ et
$$A_0=\bigcup\limits_{\substack{k=1}}^{K}\left[ k-1 ,k-1+(K-k)b+\delta b\right],$$
chaque couple d'ensembles $\big( A_i(\varepsilon'),B\big)$ respecte les conditions du théorème \ref{main_result} (en particulier on a bien $\lambda(A_i(\varepsilon')+B)=(K+\delta+\varepsilon)b$) et le plus petit ensemble (au sens de l'inclusion) contenant tous les $A_i(\varepsilon')$ à translation près est bien $A_0+\left[ 0,\varepsilon b\right]$.

Il est toutefois possible d'améliorer le théorème \ref{main_result} en affaiblissant les hypothèses sur $\varepsilon$. En ce sens, pour plus de lisibilité nous considérons l'hypothèse $\big(K+\delta+ (K^2\log K+12)\varepsilon\big)\lambda(B)<D_B$ alors que dans la démonstration, nous n'utilisons que l'hypothèse plus faible 
$$\big(K+\delta+ (K^2\log (K)-K\big(K(1+\log 4)-\log K-7+\log 4\big)+8)\varepsilon\big)\lambda(B)<D_B.$$

\newpage
\def\refname{R\'ef\'erences}
\bibliographystyle{plain}
\bibliography{/Users/robin/Desktop/mabiblio.bib}
\end{document}